\documentclass{article}
\usepackage{amsmath}
\usepackage{amssymb}
\usepackage{amsthm}
\usepackage{graphicx}
\usepackage{cite}
\usepackage[arrow,curve,matrix,arc,2cell]{xy}
\UseAllTwocells
\usepackage{hyperref}
\DeclareFontFamily{U}{rsfs}{} \DeclareFontShape{U}{rsfs}{n}{it}{<->
rsfs10}{} \DeclareSymbolFont{mscr}{U}{rsfs}{n}{it}
\DeclareSymbolFontAlphabet{\scr}{mscr}
\def\mathscr{\scr}
\begin{document}
\def\e#1\e{\begin{equation}#1\end{equation}}
\def\ea#1\ea{\begin{align}#1\end{align}}
\def\eq#1{{\rm(\ref{#1})}}
\theoremstyle{plain}
\newtheorem{thm}{Theorem}[section]
\newtheorem{prop}[thm]{Proposition}
\newtheorem{lem}[thm]{Lemma}
\newtheorem{conj}[thm]{Conjecture}
\theoremstyle{definition}
\newtheorem{dfn}[thm]{Definition}
\newtheorem{ex}[thm]{Example}
\newtheorem{rem}[thm]{Remark}
\numberwithin{equation}{section}
\def\dim{\mathop{\rm dim}\nolimits}
\def\vdim{\mathop{\rm vdim}\nolimits}
\def\supp{\mathop{\rm supp}\nolimits}
\def\id{\mathop{\rm id}\nolimits}
\def\Sym{\mathop{\rm Sym}\nolimits}
\def\Ker{\mathop{\rm Ker}}
\def\Coker{\mathop{\rm Coker}}
\def\HP{\mathop{\rm HP}\nolimits}
\def\HN{\mathop{\rm HN}\nolimits}
\def\HC{\mathop{\rm HC}\nolimits}
\def\NC{\mathop{\rm NC}\nolimits}
\def\PC{\mathop{\rm PC}\nolimits}
\def\Crit{\mathop{\rm Crit}}
\def\inf{{\rm inf}}
\def\Re{\mathop{\rm Re}\nolimits}
\def\Im{\mathop{\rm Im}\nolimits}
\def\Hom{\mathop{\rm Hom}\nolimits}
\def\Ho{\mathop{\rm Ho}\nolimits}
\def\Sch{\mathop{\bf Sch}\nolimits}
\def\dSt{\mathop{\bf dSt}\nolimits}
\def\dArt{\mathop{\bf dArt}\nolimits}
\def\dSch{\mathop{\bf dSch}\nolimits}
\def\red{{\rm red}}
\def\cl{{\rm cl}}
\def\Spec{\mathop{\rm Spec}\nolimits}
\def\bSpec{\mathop{\bs{\rm Spec}}\nolimits}
\def\qcoh{{\mathop{\rm qcoh}}}
\def\Pol{{\mathop{\rm Pol}}}
\def\bs{\boldsymbol}
\def\ge{\geqslant}
\def\le{\leqslant\nobreak}
\def\bA{{\mathbin{\mathbb A}}}
\def\bL{{\mathbin{\mathbb L}}}
\def\bP{{\mathbin{\mathbb P}}}
\def\bT{{\mathbin{\mathbb T}}}
\def\bU{{\bs U}}
\def\bV{{\bs V}}
\def\bW{{\bs W}}
\def\bX{{\bs X}}
\def\bY{{\bs Y}}
\def\bZ{{\bs Z}}
\def\cA{{\mathbin{\cal A}}}
\def\cF{{\mathbin{\cal F}}}
\def\cL{{\mathbin{\cal L}}}
\def\cM{{\mathbin{\cal M}}}
\def\O{{\mathbin{\cal O}}}
\def\cP{{\mathbin{\cal P}}}
\def\PV{{\mathbin{\cal{PV}}}}
\def\bH{{\mathbin{\mathbb H}}}
\def\C{{\mathbin{\mathbb C}}}
\def\K{{\mathbin{\mathbb K}}}
\def\N{{\mathbin{\mathbb N}}}
\def\Z{{\mathbin{\mathbb Z}}}
\def\al{\alpha}
\def\be{\beta}
\def\ga{\gamma}
\def\de{\delta}
\def\io{\iota}
\def\ep{\epsilon}
\def\la{\lambda}
\def\ka{\kappa}
\def\th{\theta}
\def\ze{\zeta}
\def\up{\upsilon}
\def\vp{\varphi}
\def\si{\sigma}
\def\om{\omega}
\def\De{\Delta}
\def\La{\Lambda}
\def\Th{\Theta}
\def\Om{\Omega}
\def\Ga{\Gamma}
\def\Si{\Sigma}
\def\Up{\Upsilon}
\def\pd{\partial}
\def\ts{\textstyle}
\def\st{\scriptstyle}
\def\sm{\setminus}
\def\bu{\bullet}
\def\op{\oplus}
\def\ot{\otimes}
\def\boxt{\boxtimes}
\def\ov{\overline}
\def\ul{\underline}
\def\bigop{\bigoplus}
\def\iy{\infty}
\def\es{\emptyset}
\def\ra{\rightarrow}
\def\ab{\allowbreak}
\def\longra{\longrightarrow}
\def\hookra{\hookrightarrow}
\def\bs{\boldsymbol}
\def\t{\times}
\def\w{\wedge}
\def\ci{\circ}
\def\ti{\tilde}
\def\d{{\rm d}}
\def\ha{{\ts\frac{1}{2}}}
\def\an#1{\langle #1 \rangle}
\def\dd{{\rm d}_{\rm dR}}
\def\cdga{\mathop{\bf cdga}\nolimits}
\title{A Lagrangian Neighbourhood Theorem for shifted symplectic derived schemes}
\author{Dominic Joyce and Pavel Safronov}
\date{}
\maketitle
\begin{abstract} Pantev, To\"en, Vaqui\'e and Vezzosi \cite{PTVV} defined $k$-shifted symplectic derived schemes and stacks $\bX$ for $k\in\Z$, and Lagrangians $\bs f:\bs L\ra\bX$ in them. They have important applications to Calabi--Yau geometry and quantization. Bussi, Brav and Joyce \cite{BBJ} and Bouaziz and Grojnowski \cite{BoGr} proved `Darboux Theorems' giving explicit Zariski or \'etale local models for $k$-shifted symplectic derived schemes $\bX$ for $k<0$ presenting them as twisted shifted cotangent bundles.

We prove a `Lagrangian Neighbourhood Theorem' which gives explicit Zariski or \'etale local models for Lagrangians $\bs f:\bs L\ra\bX$ in $k$-shifted symplectic derived schemes $\bX$ for $k<0$, relative to the `Darboux form' local models of \cite{BBJ} for $\bX$. That is, locally such Lagrangians can be presented as twisted shifted conormal bundles. We also give a partial result when~$k=0$.

We expect our results will have future applications to shifted Poisson geometry \cite{CPTVV}, and to defining `Fukaya categories' of complex or algebraic symplectic manifolds, and to the categorification of Donaldson--Thomas theory of Calabi--Yau 3-folds and `Cohomological Hall Algebras'.
\end{abstract}

\setcounter{tocdepth}{2}
\tableofcontents

\section{Introduction}
\label{ln1}

Using To\"en and Vezzosi's theory of Derived Algebraic Geometry \cite{Toen1,Toen2,Toen3,ToVe1,ToVe2}, Pantev, To\"en, Vaqui\'e and Vezzosi \cite{PTVV} defined $k$-{\it shifted symplectic structures\/} $\om_\bX$ on a derived scheme or stack $\bX$, for $k\in\Z$. If $\bX$ is a derived scheme and $\om_\bX$ a 0-shifted symplectic structure, then $\bX=X$ is a smooth classical scheme and $\om_X\in H^0(\La^2T^*X)$ a classical symplectic structure on $X$. They proved that if $Y$ is a Calabi--Yau $m$-fold, then derived moduli stacks $\bs\cM$ of (complexes of) coherent sheaves on $Y$ have natural $(2-m)$-shifted symplectic structures~$\om_{\bs\cM}$.

Pantev et al.\ \cite{PTVV} also defined {\it Lagrangians\/} $\bs f:\bs L\ra\bX$ in a $k$-shifted symplectic derived stack $(\bX,\om_\bX)$, and showed that fibre products $\bs L\t_\bX\bs M$ of Lagrangians $\bs f:\bs L\ra\bX$, $\bs g:\bs M\ra\bX$ are $(k-1)$-shifted symplectic. Calaque \cite{Cala} proved that if $X$ is a Fano $(m+1)$-fold and $Y\subseteq X$ a smooth anticanonical divisor, so that $Y$ is a Calabi--Yau $m$-fold, and $\bs\cL,\bs\cM$ are derived moduli stacks of (complexes of) coherent sheaves on $X,Y$ with derived restriction morphism $\bs f:\bs\cL\ra\bs\cM$, then $\bs\cL$ is Lagrangian in the $(2-m)$-shifted symplectic $(\bs\cM,\om_{\bs\cM})$. 

Recently, Calaque, Pantev, To\"en, Vaqui\'e and Vezzosi \cite{CPTVV} have also developed a related theory of $k$-{\it shifted Poisson structures\/} $\pi_\bX$ on a derived scheme or stack $\bX$, for $k\in\Z$, and {\it coisotropics\/} $\bs f:\bs C\ra\bX$ in $(\bX,\pi_\bX)$. They prove \cite[Th.~3.2.4]{CPTVV} that the spaces of $k$-shifted symplectic structures $\om_\bX$ and nondegenerate $k$-shifted Poisson structures $\pi_\bX$ on $\bX$ are equivalent. Costello--Rozenblyum and Pridham \cite{Prid} have also announced similar results.

Given a symplectic manifold $(X,\om)$, the classical Darboux Theorem chooses local coordinates $(x_1,\ldots,x_n,y_1,\ldots,y_n)$ on $X$ with $\om=\sum_{j=1}^n\dd x_j\dd y_j$. Bussi, Brav and Joyce \cite[Th.~5.18]{BBJ} proved a `$k$-shifted symplectic Darboux Theorem', which for a $k$-shifted symplectic derived $\K$-scheme $(\bX,\om_\bX)$ with $k<0$ chooses a cdga $A^\bu$, a Zariski open inclusion $\bs i:\bSpec A^\bu\hookra\bX$, and coordinates $x_j^i,y_j^{k-i}\in A^\bu$ with $\bs i^*(\om_\bX)\simeq(\om^0,0,0,\ldots)$ for $\om^0=\sum_{i,j}\dd x_j^i\dd y_j^{k-i}$. (Actually, all this only holds for $k\not\equiv 2\mod 4$, and for $k\equiv 2\mod 4$ there is a more complicated expression also involving coordinates~$z_j^{k/2}$.)

This was the foundation for a series of papers \cite{BeBa,BBBJ,BoJo,Buss1,Buss2,BBDJS,BBJ,BJM,Joyc} concerning generalizations of Donaldson--Thomas theory for Calabi--Yau 3- and 4-folds, involving perverse sheaves, motives, and new enumerative invariants. It can also be used as part of a proof that $k$-shifted symplectic derived schemes carry nondegenerate $k$-shifted Poisson structures, though this was not used in~\cite{CPTVV,Prid}.

Given a Lagrangian $L\hookra X$ in a symplectic manifold $(X,\om)$, the classical Lagrangian Neighbourhood Theorem describes $L,X,\om$ in local coordinates.

The purpose of this paper is to prove a `$k$-shifted symplectic Lagrangian Neighbourhood Theorem', Theorem \ref{ln3thm2} below, which given a Lagrangian $\bs f:\bs L\ra\bX$ in a $k$-shifted symplectic derived $\K$-scheme $(\bX,\om_\bX)$ for $k<0$, and a `Darboux form' local description $\bs i:\bSpec A^\bu\hookra\bX$, $x_j^i,y_j^{k-i}\in A^\bu$, $\om^0=\sum_{i,j}\dd x_j^i\dd y_j^{k-i}$ for $(\bX,\om_\bX)$ as in \cite{BBJ}, chooses a cdga $B^\bu$, coordinates $\ti x_j^i,\ab u_j^i,\ab v_j^{k-1-i}\in B^\bu$, a Zariski open inclusion $\bs j:\bSpec B^\bu\hookra\bX$, and a cdga morphism $\al:A^\bu\ra B^\bu$ with $\ti x^i_j=\al(x^i_j)$ in a homotopy commutative diagram
\e
\begin{gathered}
\xymatrix@C=90pt@R=15pt{ *+[r]{\bSpec B^\bu\,} \ar[d]^{\bSpec\al} \ar[r]_{\bs j} & *+[l]{\bs L} \ar[d]_{\bs f}  \\ *+[r]{\bSpec A^\bu\,} \ar[r]^{\bs i}  & *+[l]{\bX,\!} }
\end{gathered}
\label{ln1eq1}
\e
such that the pullback $\bs j^*(h_{\bs L})$ of the Lagrangian structure $h_{\bs L}$ on $\bs f:\bs L\ra\bX$ to $\bSpec\al:\bSpec B^\bu\ra\bSpec A^\bu$ using \eq{ln1eq1} has $\bs j^*(h_{\bs L})\simeq(h^0,0,0,\ldots)$ with $h^0=\sum_{i,j}\dd u_j^i\dd v_j^{k-1-i}$. (Actually, all this only holds for $k\not\equiv 3\mod 4$, and for $k\equiv 3\mod 4$ there is a more complicated expression also involving coordinates $w_j^{(k-1)/2}$.) Theorem \ref{ln3thm3} also gives a partial result for~$k=0$.

Bouaziz and Grojnowski \cite{BoGr} proved their own $k$-shifted symplectic Darboux Theorem independently of \cite{BBJ}, showing that a $k$-shifted symplectic derived $\K$-scheme $(\bX,\om_\bX)$ for $k<0$ is (at least for $k\not\equiv 2\mod 4$) \'etale locally equivalent to a {\it twisted\/ $k$-shifted cotangent bundle\/} $T^*_t[k]\bY$, where $\bY$ is an affine derived $\K$-scheme, and $t\in \O_\bY^{k+1}$ with $\d t=0$ is used to `twist' the $k$-shifted cotangent bundle $T^*[k]\bY$. Remark \ref{ln2rem3} below relates their picture to that of~\cite{BBJ}.

In Remark \ref{ln3rem1} we interpret our `$k$-shifted Lagrangian Neighbourhood Theorem' in the style of Bouaziz and Grojnowski \cite{BoGr}, by saying that if $\bs f:\bs L\ra\bX$ is Lagrangian in a $k$-shifted symplectic $(\bX,\om_\bX)$ for $k<0$, and $\bX$ is locally modelled on $T^*_t[k]\bY$, then $\bs f:\bs L\ra\bX$ is (at least for $k\not\equiv 3\mod 4$) locally modelled on the inclusion morphism $N^*_{u/t}[k](\bZ/\bY)\ra T^*_t[k]\bY$, where $N^*_{u/t}[k](\bZ/\bY)$ is the {\it twisted\/ $k$-shifted conormal bundle\/} of a morphism of affine derived $\K$-schemes $\bs g:\bZ\ra\bY$, and $u\in\O_\bZ^k$ with $\d u=-\bs g^*(t)$ is used to `twist'~$N^*[k](\bZ/\bY)$.

If the $k$-shifted symplectic derived $\K$-scheme $(\bX,\om_\bX)$ is a point $(\Spec\K,0)$ then Lagrangians $\bs f:\bs L\ra\bX$ are just $(k-1)$-shifted symplectic derived $\K$-schemes $(\bs L,\om_{\bs L})$. In this case, our Lagrangian Neighbourhood Theorem reduces to the Darboux Theorem of \cite{BBJ}. So the proof in \S\ref{ln4} is a generalization of that in \cite[\S 5.6]{BBJ}, and runs parallel to \cite{BBJ} at several points.

Like the Darboux Theorem of \cite{BBJ,BoGr}, our Lagrangian Neighbourhood Theorem should have important applications. For example, it gives local models for moduli schemes of coherent sheaves on Fano $(m+1)$-folds $X$ with restriction morphisms to moduli schemes of coherent sheaves on a Calabi--Yau anticanonical hypersurface $Y\subset X$. We briefly discuss some conjectures which we hope our theorem will help to prove.

\begin{conj}
\label{ln1conj2}
Let\/ $(\bX,\om_\bX)$ be a $(-1)$-shifted symplectic derived\/ $\C$-scheme with an `orientation'. Then Bussi, Brav, Dupont, Joyce, and Szendr\H oi\/ {\rm\cite[Cor.~6.11]{BBDJS}} construct a natural perverse sheaf\/ $\cP^\bu_{\bX,\om_\bX}$ on $X=t_0(\bX),$ such that if\/ $(\bX,\om_\bX)$ is locally modelled on a critical locus $\bs\Crit(\Phi:U\ra\bA^1),$ then $\cP^\bu_{\bX,\om_\bX}$ is locally modelled on the perverse sheaf of vanishing cycles $\PV_{U,\Phi}^\bu$.

Suppose $\bs f:\bs L\ra\bX$ is a Lagrangian, with an `orientation' relative to that of\/ $\bX,$ and $\bs f$ is proper. Then we can define a natural element\/ $\la_{\bs L}$ in the hypercohomology $\bH^{\vdim\bs L}(\cP^\bu_{\bX,\om_\bX})$. These $\la_{\bs L}$ satisfy certain composition laws for composition of Lagrangian correspondences.
\end{conj}

The first author has an outline of a proof of Conjecture \ref{ln1conj2}.

As suggested in \cite[Rem.~6.15]{BBDJS}, we would like to define a `Fukaya category' $\cF(S)$ of (derived) complex or algebraic Lagrangians $L\ra S$ in a complex or algebraic symplectic manifold $(S,\om)$ of dimension $2n$, such that if $L,M$ are oriented Lagrangians in $S$ then the morphisms $L\ra M$ in $\cF(S)$ are $\Hom^*(L,M)=\bH^{*-n}(\cP_{L,M}^\bu)$, where $\cP_{L,M}^\bu$ is the perverse sheaf on the $-1$-shifted symplectic $\bX=L\t_SM$ described above. 

As in Ben-Bassat \cite{BeBa}, if $L,M,N$ are (derived) Lagrangians in $(S,\om)$, then $\bs Y=L\t_S M\t_S N\ra (L\t_SM)\t(M\t_SN)\t(N\t_SL)$ is Lagrangian in $-1$-shifted symplectic. The hypercohomology class $\la_{\bs Y}$ associated to this in Conjecture \ref{ln1conj2} is what we need to define composition $\Hom^*(M,N)\t\Hom^*(L,M)\ra \Hom^*(L,N)$ of morphisms in the `Fukaya category' $\cF(S)$. Amorim and Ben-Bassat \cite{AmBB} discuss this proposal and Conjecture \ref{ln1conj2} in detail.

A stacky version of Conjecture \ref{ln1conj2} is what we need to define multiplication in a `Cohomological Hall Algebra' associated to a Calabi--Yau 3-fold in the sense of Kontsevich--Soibelman \cite{KoSo}, defined using the perverse sheaves on Calabi--Yau 3-fold moduli stacks constructed in Ben-Bassat, Bussi, Brav and Joyce~\cite{BBBJ}.

\begin{conj} 
\label{ln1conj3}
Let\/ $U$ be a smooth\/ $\K$-scheme and\/ $\Phi:U\ra\bA^1$ a regular function. Then the derived critical locus $\bs\Crit(\Phi)$ is a $-1$-shifted symplectic derived\/ $\K$-scheme. We can also define the $\Z_2$-graded dg-category of matrix factorizations $\mathop{\rm MF}(U,\Phi),$ as in Preygel\/ {\rm\cite{Prey}} for instance.

Suppose $\bs f:\bs L\ra\bs\Crit(\Phi)$ is a Lagrangian, with\/ $\vdim\bs L-\dim U$ even, equipped with an `orientation' and a `spin structure', and that\/ $\bs f$ is proper. Then we can define an object\/ $\mu_{\bs L}\in\mathop{\rm MF}(U,\Phi)$ associated to $\bs L$. In this way we interpret\/ $\mathop{\rm MF}(U,\Phi)$ as a kind of `Fukaya category' of the $-1$-shifted symplectic derived\/ $\K$-scheme\/~$\bs\Crit(\Phi)$.
\end{conj}

This is connected to the programme of Kapustin and Rozansky \cite{KaRo} for associating a  2-category to a complex symplectic manifold, locally described using matrix factorization categories.

For each of the Conjectures \ref{ln1conj2}--\ref{ln1conj3}, using our Lagrangian Neighbourhood Theorem we can write down local models on $\bs L$ for the coisotropic structure, and for $\la_{\bs L}$ and $\mu_{\bs L}$. The problem is to glue these local models together globally.

We begin in \S\ref{ln2} with background material on Derived Algebraic Geometry and Pantev--To\"en--Vaqui\'e--Vezzosi's shifted symplectic geometry. Section \ref{ln3} gives our main results. Theorem \ref{ln3thm1} in \S\ref{ln31} shows that a morphism $\bs f:\bX\ra\bY$ of derived $\K$-schemes is locally modelled on $\bSpec\al:\bSpec A^\bu\ra \bSpec B^\bu$, where $A^\bu,B^\bu$ are cdgas and $\al:B^\bu\ra A^\bu$ a morphism, all in a particularly nice form. 

Theorem \ref{ln3thm2} in \S\ref{ln33} is our `Lagrangian Neighbourhood Theorem', showing that Lagrangians $\bs f:\bs L\ra\bX$ in $k$-shifted symplectic derived $\K$-schemes $(\bX,\om_\bX)$ for $k<0$ are locally modelled on explicit `Lagrangian Darboux form' examples given in Examples \ref{ln3ex1} and \ref{ln3ex2} in \S\ref{ln32}. Theorem \ref{ln3thm3} in \S\ref{ln34} also gives a partial result for $k=0$. Section \ref{ln4} proves Theorems \ref{ln3thm1} and~\ref{ln3thm2}.
\medskip

\noindent{\bf Conventions.} Throughout $\K$ will be an algebraically closed field with characteristic zero. All classical $\K$-schemes are assumed locally of finite type, and all derived $\K$-schemes $\bX$ are assumed to be locally finitely presented. Our sign conventions for cdgas, exterior forms, etc., follow Bussi, Brav and Joyce~\cite{BBJ}.

\medskip

\noindent{\bf Acknowledgements.} We would like to thank Dennis Borisov, Chris Brav, and Ian Grojnowski for helpful conversations. This research was supported by EPSRC Programme Grant EP/I033343/1.

\section{Background material}
\label{ln2}

We begin with some background material and notation needed later. Some references are To\"en and Vezzosi \cite{Toen1,Toen2,Toen3,ToVe1,ToVe2} for \S\ref{ln21}--\S\ref{ln22}, and Pantev, To\"en, Vezzosi and Vaqui\'e \cite{PTVV} for \S\ref{ln23}--\S\ref{ln24}, and Brav, Bussi and Joyce \cite{BBJ} for \S\ref{ln25}. Throughout the paper, $\K$ will be an algebraically closed field of characteristic zero.

\subsection{Commutative differential graded algebras}
\label{ln21}

\begin{dfn} 
\label{ln2def1}
Write $\cdga_\K$ for the category of commutative differential graded $\K$-algebras in nonpositive degrees, and $\cdga_\K^{\bf op}$ for its opposite category. 
Objects of $\cdga_\K$ are of the form $\cdots \ra A^{-2}\,{\buildrel\d\over\longra}\,A^{-1}\,{\buildrel\d\over\longra}\,A^0$. Here $A^k$ for $k=0,-1,-2,\ldots$ is the $\K$-vector space of degree $k$ elements of $A^\bu$, and we have a $\K$-bilinear, associative, supercommutative multiplication $\cdot:A^k\t A^l\ra A^{k+l}$ for $k,l\le 0$, an identity $1\in A^0$, and differentials $\d:A^k\ra A^{k+1}$ for $k<0$ satisfying
\begin{equation*}
\d(a\cdot b)=(\d a)\cdot b+(-1)^ka\cdot(\d b)
\end{equation*}
for all $a\in A^k$, $b\in A^l$. We write such objects as $A^\bu$ or $(A^*,\d)$. 

Here and throughout we will use the superscript `$\,{}^*\,$' to denote {\it graded\/} objects (e.g.\ graded algebras or vector spaces), where $*$ stands for an index in $\Z$, so that $A^*$ means $(A^k$, $k\in\Z)$. We will use the superscript `$\,{}^\bu\,$' to denote {\it differential graded\/} objects (e.g.\ differential graded algebras or complexes), so that $A^\bu$ means $(A^*,\d)$, the graded object $A^*$ together with the differential $\d$. 

{\it Morphisms\/} $\al:A^\bu\ra B^\bu$ in $\cdga_\K$ are $\K$-linear maps $\al^k:A^k\ra B^k$ for all $k\le 0$ commuting with all the structures on $A^\bu,B^\bu$. A morphism $\al:A^\bu\ra B^\bu$ is a {\it quasi-isomorphism\/} if $H^k(\al):H^k(A^\bu)\ra H^k(B^\bu)$ is an isomorphism on cohomology groups for all~$k\le 0$.
\end{dfn}

\begin{rem}
\label{ln2rem1}
A fundamental principle of derived algebraic geometry is that $\cdga_\K$ is not really the right category to work in, but instead one wants to define a new category (or better, $\iy$-category) by inverting (localizing) quasi-isomorphisms in $\cdga_\K$.

In fact $\cdga_\K$ has the additional structure of a simplicial model category, with weak equivalences quasi-isomorphisms, in which all objects are fibrant, and in which cdgas $A^\bu$ with $A^*$ free as a commutative graded $\K$-algebra are cofibrant. The $n$-simplices of the mapping space between two cdgas $A^\bu$ and $B^\bu$ are given by morphisms $A^\bu\ra B^\bu\ot \Om^\bu(\De^n)$, where $\Om^\bu(\De^n)$ is the cdga generated by elements $s_i$ of degree $0$ and $t_i$ of degree $1$ for $i = 0,\ldots,n$ with the relations $\sum s_i = 1$ and $\sum t_i = 0$ and the differential $\d s_i = t_i$. Note that $\Om^\bu(\De^n)$ are concentrated in positive degrees, and are not elements of $\cdga_\K$.

We will write $\cdga^\iy_\K$ for the associated $\iy$-category, so that the homotopy category $\Ho(\cdga^\iy_\K)$ is the localized category $\cdga_\K[{\mathcal Q}^{-1}]$ with quasi-isomorphisms inverted, an ordinary category. We will not go into any detail about model categories and $\iy$-categories below, but here is some basic orientation on one issue relevant to this paper, for readers unfamiliar with these ideas. The objects of $\cdga_\K,\cdga^\iy_\K,\Ho(\cdga^\iy_\K)$ are the same. If $A^\bu,B^\bu$ are objects, a morphism $\phi:A^\bu\ra B^\bu$ in $\cdga_\K$ is also a morphism in $\cdga^\iy_\K$ and $\Ho(\cdga^\iy_\K)$. However, a morphism $\phi^\iy:A^\bu\ra B^\bu$ in $\cdga^\iy_\K$ (or equivalently, in $\Ho(\cdga^\iy_\K)$) need not correspond to any morphism $\phi:A^\bu\ra B^\bu$ in $\cdga_\K$, unless $A^\bu$ is cofibrant. If $A^\bu$ is cofibrant, the mapping space in $\cdga^\iy_\K$ is given by the mapping space in $\cdga_\K$.

Standard model cdgas $A^\bu$ are `nearly cofibrant'. They have the property that if $\phi^\iy:A^\bu\ra B^\bu$ is a morphism in $\cdga_\K^\iy$ with $A^\bu$ standard model, such that $H^0(\phi^\iy):H^0(A^\bu)\ra H^0(B^\bu)$ can be lifted to a $\K$-algebra morphism $\phi^0:A^0\ra B^0$, then $\phi^\iy$ can be lifted to $\phi:A^\bu\ra B^\bu$ in~$\cdga_\K$.

All this will be important because if $\bX\simeq\bSpec A^\bu$ and $\bY\simeq\bSpec B^\bu$ are affine derived $\K$-schemes and $\bs f:\bY\ra\bX$ is a morphism, then $\bs f\simeq\bSpec\phi^\iy$ for some morphism $\phi^\iy:A^\bu\ra B^\bu$ in $\cdga_\K^\iy$. For our Lagrangian Neighbourhood Theorem in \S\ref{ln33}, we want to lift $\phi^\iy$ to $\phi:A^\bu\ra B^\bu$ in~$\cdga_\K$.
\end{rem}

\begin{dfn} 
\label{ln2def2}
Let $A^\bu\in\cdga_\K$, and write $D(\mathop{\rm mod}A)$ for the derived category of dg-modules over $A^\bu$. Define a {\it derivation of degree $k$} from $A^\bu$ to an $A^\bu$-module $M^\bu$ to be a $\K$-linear map $\de:A^\bu\ra M^\bu$ that is homogeneous of degree $k$ with
\begin{equation*}
\de(fg)=\de(f)g+(-1)^{k\deg f}f\de(g).
\end{equation*}
Just as for ordinary commutative algebras, there is a universal
derivation into an $A^\bu$-module of {\it K\"ahler differentials\/}
$\Om^1_{A^\bu}$, which can be constructed as $I/I^2$ for $I=\Ker(m:A^\bu\ot A^\bu
\ra A^\bu)$. The universal derivation $\de:A^\bu\ra \Om^1_{A^\bu}$ is $\de(a)=a\ot 1-1\ot a \in I/I^2$. One checks that $\de$ is a universal degree $0$ derivation, so that ${}\ci
\de:\Hom^\bu_{A^\bu}(\Om^1_{A^\bu},M^\bu) \ra \mathop{\rm Der}^\bu(A,M^\bu)$ is an
isomorphism of dg-modules. 

Note that $\Om^1_{A^\bu}=\bigl((\Om^1_{A^\bu})^*,\d\bigr)$ is canonical up to strict isomorphism, not just up to quasi-isomorphism of complexes, or up to equivalence in $D(\mathop{\rm mod} A)$. Also, the underlying graded vector space $(\Om^1_{A^\bu})^*$, as a module over the graded algebra $A^*$, depends only on $A^*$ and not on the differential $\d$ in $A^\bu=(A^*,\d)$.

Similarly, given a morphism of cdgas $\Phi:A^\bu\ra B^\bu$, we can define the {\it relative K\"ahler differentials\/} $\Om^1_{B^\bu/A^\bu}$.

The {\it cotangent complex\/} $\bL_{A^\bu}$ of $A^\bu$ is related to the K\"ahler differentials $\Om^1_{A^\bu}$, but is not quite the same. If $\Phi:A^\bu\ra B^\bu$ is a quasi-isomorphism of cdgas over $\K$, then $\Phi_*:(\Om^1_{A^\bu})\ot_{A^\bu}B^\bu\ra \Om^1_{B^\bu}$ may not be a quasi-isomorphism of $B^\bu$-modules. So K\"ahler differentials are not well-behaved under localizing quasi-isomorphisms of cdgas, which is bad for doing derived algebraic geometry. 

The cotangent complex $\bL_{A^\bu}$ is a substitute for $\Om^1_{A^\bu}$ which is well-behaved under localizing quasi-isomorphisms. It is an object in $D(\mathop{\rm mod} A)$, canonical up to equivalence. We can define it by replacing $A^\bu$ by a quasi-isomorphic, cofibrant cdga $B^\bu$, and then setting $\bL_{A^\bu}=(\Om^1_{B^\bu})\ot_{B^\bu}A^\bu$. We will be interested in the $p^{\rm th}$ exterior power $\La^p\bL_{A^\bu}$, and the dual $(\bL_{A^\bu})^\vee$, which is called the {\it tangent complex}, and written $\bT_{A^\bu}=(\bL_{A^\bu})^\vee$. 

There is a {\it de Rham differential\/} $\dd:\La^p\bL_{A^\bu}\ra\La^{p+1}\bL_{A^\bu}$, a morphism of complexes, with $\dd^2=0:\La^p\bL_{A^\bu}\ra\La^{p+2}\bL_{A^\bu}$. Note that each $\La^p\bL_{A^\bu}$ is also a complex with its own internal differential $\d:(\La^p\bL_{A^\bu})^k\ra(\La^p\bL_{A^\bu})^{k+1}$, and $\dd$ being a morphism of complexes means that $\d\ci\dd=\dd\ci\d$.

Similarly, given a morphism of cdgas $\Phi:A^\bu\ra B^\bu$, we can define the {\it relative cotangent complex\/} $\bL_{B^\bu/A^\bu}$.
\end{dfn}

\begin{dfn} 
\label{ln2def3}
Following \cite[Def.~2.9]{BBJ}, we will call $A^\bu\in\cdga_\K$ of {\it standard form\/} if $A^0$ is a smooth finitely generated $\K$-algebra, and the cotangent module $\Om^1_{A^0}$ is a free $A^0$-module of finite rank, and the graded $\K$-algebra $A^*$ is freely generated over $A^0$ by finitely many generators, all in negative degrees.

More explicitly, as $A^0$ is a smooth $\K$-algebra, $U=\Spec A^0$ is a smooth $\K$-scheme. Suppose that $U$ admits \'etale coordinates $(x_1^0,\ldots,x_{m_0}^0):U\ra\bA^{m_0}$. Then $\Om^1_{A^0}\cong A^0\ot_\K\an{\dd x_1^0,\ldots,\dd x_{m_0}^0}_\K$ is a free $A^0$-module of rank $m_0$. Suppose we are given elements $x_1^i,\ldots,x_{m_i}^i$ in $A^i$ for $i=-1,-2,\ldots,k$, such that $A^*=A^0[x^i_j:i=-1,\ldots,k$, $j=1,\ldots,m_i]$ is the graded $\K$-algebra freely generated over $A^0$ by the generators $x_j^i$ in degree $i<0$. Then $A^\bu=(A^*,\d)$ is a standard form cdga. The differential $\d$ on $A^*$ is determined uniquely by the elements $\d x_j^i\in A^{i+1}$ for $i=-1,-2,\ldots,k$ and~$j=1,\ldots,m_i$.

The {\it virtual dimension\/} of $A^\bu$ is $\vdim A^\bu=\sum_{i=0}^d(-1)^im_i\in\Z$.

Then the K\"ahler differentials $\Om^1_{A^\bu}$ are given as an $A^*$-module by
\e
\Om_{A^\bu}^1\cong A^*\ot_\K\an{\dd x_j^i:i=0,-1,\ldots,k,\; j=1,\ldots,m_i}_\K.
\label{ln2eq1}
\e
As in \cite[\S 2.3]{BBJ}, an important property of standard form cdgas $A^\bu$ is that they are sufficiently cofibrant that the K\"ahler differentials $\Om^1_{A^\bu}$ provide a model for the cotangent complex $\bL_{A^\bu}$, so we can take $\Om^1_{A^\bu}=\bL_{A^\bu}$, without having to replace $A^\bu$ by an unknown cdga $B^\bu$. Thus standard form cdgas are convenient for doing explicit computations with cotangent complexes.

We say that a standard form cdga $A^\bu$ is {\it minimal\/} at $p\in\bSpec A^\bu$ if all the differentials in the complex of $\K$-vector spaces $\Om^1_{A^\bu}\vert_p$ are zero. This means that $m_i=\dim H^i\bigl(\bL_{A^\bu}\vert_p\bigr)$ for $i=0,-1,\ldots,d$, and $A^\bu$ is defined using the minimum number of variables $x_j^i$ in each degree $i=0,-1,\ldots,$ compared to all other cdgas locally equivalent to $A^\bu$ near~$p$. 
\end{dfn}

\subsection{Derived algebraic geometry and derived schemes}
\label{ln22}

\begin{dfn} 
\label{ln2def4}
Write $\dSt_\K$ for the $\iy$-category of {\it derived\/ $\K$-stacks} (or $D^-$-{\it stacks\/}) defined by To\"en and Vezzosi \cite[Def.~2.2.2.14]{ToVe1}, \cite[Def.~4.2]{Toen1}. Objects $\bX$ in $\dSt_\K$ are $\iy$-functors
\begin{equation*}
\bX:\{\text{simplicial commutative $\K$-algebras}\}\longra
\{\text{simplicial sets}\}
\end{equation*}
satisfying sheaf-type conditions. There is a {\it spectrum functor\/}
\begin{equation*}
\bSpec:(\cdga^\iy_\K)^{\bf op}\longra\dSt_\K.
\end{equation*}
A derived $\K$-stack $\bX$ is called an {\it affine derived\/ $\K$-scheme\/} if $\bX$ is equivalent in $\dSt_\K$ to $\bSpec A^\bu$ for some cdga $A^\bu$ over $\K$. As in \cite[\S 4.2]{Toen1}, a derived $\K$-stack $\bX$ is called a {\it derived\/ $\K$-scheme\/} if it may be covered by Zariski open $\bY\subseteq\bX$ with $\bY$ an affine derived $\K$-scheme. Write $\dSch_\K$ for the full $\iy$-subcategory of derived $\K$-schemes in $\dSt_\K$, and $\dSch_\K^{\bf aff}\subset\dSch_\K$ for the full $\iy$-subcategory of affine derived $\K$-schemes. Then $\bSpec$ is an equivalence $(\cdga_\K^\iy)^{\bf op}\,{\buildrel\sim\over\longra}\,\dSch_\K^{\bf aff}$.

We shall assume throughout this paper that all derived $\K$-schemes $\bX$ are {\it locally finitely presented\/} in the sense of To\"en and Vezzosi~\cite[Def.~1.3.6.4]{ToVe1}. 

With this assumption, derived schemes have a {\it virtual dimension\/} $\vdim\bX$, which is a locally constant function $\vdim\bX:\bX\ra\Z$. If $\bX=\bSpec A^\bu$ for $A^\bu$ a standard form cdga then $\vdim\bX=\vdim A^\bu$, for $\vdim A^\bu$ as in Definition~\ref{ln2def3}.

There is a {\it classical truncation functor\/} $t_0:\dSch_\K\ra\Sch_\K$ taking a derived $\K$-scheme $\bX$ to the underlying classical $\K$-scheme $X=t_0(\bX)$. On affine derived schemes $\dSch_\K^{\bf aff}$ this maps~$t_0:\bSpec A^\bu\mapsto\Spec H^0(A^\bu)=\Spec (A^0/\d(A^{-1}))$.

To\"en and Vezzosi show that a derived $\K$-scheme $\bX$ has a {\it cotangent complex\/} $\bL_\bX$ \cite[\S 1.4]{ToVe1}, \cite[\S 4.2.4--\S 4.2.5]{Toen1} in a stable $\iy$-category $L_\qcoh(\bX)$ defined in \cite[\S 3.1.7, \S 4.2.4]{Toen1}. We will be interested in the $p^{\rm th}$ exterior power $\La^p\bL_\bX$, and the dual $(\bL_\bX)^\vee$, which is called the {\it tangent complex\/} $\bT_\bX$.

By a {\it point\/} of a derived $\K$-scheme $\bX$, written $x\in\bX$, we will always mean that $x\in X(\K)$ is a $\K$-point of the underlying classical $\K$-scheme $X=t_0(\bX)$.

When $\bX=X$ is a classical scheme, the homotopy category of $L_\qcoh(\bX)$ is the triangulated category $D_{\qcoh}(X)$ of complexes of quasicoherent sheaves. These have the usual properties of (co)tangent complexes. For instance, if $\bs f:\bX\ra\bY$ is a morphism in $\dSch_\K$ there is a distinguished triangle
\begin{equation*}
\smash{\xymatrix@C=30pt{ \bs f^*(\bL_\bY) \ar[r]^(0.55){\bL_{\bs f}} &
\bL_\bX \ar[r] & \bL_{\bX/\bY} \ar[r] & \bs f^*(\bL_\bY)[1], }}
\end{equation*}
where $\bL_{\bX/\bY}$ is the {\it relative cotangent complex\/} of~$\bs f$. 

Now suppose $A^\bu$ is a cdga over $\K$, and $\bX$ a derived $\K$-scheme with $\bX\simeq\bSpec A^\bu$ in $\dSch_\K$. Then we have an equivalence of triangulated categories $\Ho(L_\qcoh(\bX))\simeq \ab D(\mathop{\rm mod}A^\bu)$, which identifies cotangent complexes $\bL_\bX\simeq\bL_{A^\bu}$. If also $A^\bu$ is of standard form then $\bL_{A^\bu}\simeq\Om^1_{A^\bu}$, so~$\bL_\bX\simeq\Om^1_{A^\bu}$.
\end{dfn}

Bussi, Brav and Joyce \cite[Th.~4.1]{BBJ} prove:

\begin{thm} 
\label{ln2thm1}
Suppose $\bX$ is a derived\/ $\K$-scheme (as always, assumed locally finitely presented), and\/ $x\in\bX$. Then there exist a standard form cdga $A^\bu$ over $\K$ which is minimal at\/ $p\in\bSpec A^\bu,$ in the sense of Definition\/ {\rm\ref{ln2def3},} and a Zariski open inclusion $\bs i:\bSpec A^\bu\hookra\bX$ with\/~$\bs i(p)=x$.
\end{thm}

They also explain \cite[Th.~4.2]{BBJ} how to compare two such standard form charts $\bSpec A^\bu\hookra\bX$, $\bSpec B^\bu\hookra\bX$ on their overlap in $\bX$, using a third chart.

\subsection{\texorpdfstring{PTVV's shifted symplectic geometry}{PTVV\textquoteright s shifted symplectic geometry}}
\label{ln23}

Next we summarize parts of the theory of shifted symplectic geometry, as developed by Pantev, To\"{e}n, Vaqui\'{e}, and Vezzosi in \cite{PTVV}. We explain them for derived $\K$-schemes $\bX$, although Pantev et al.\ work more generally with derived stacks.

Given a (locally finitely presented) derived $\K$-scheme $\bX$ and $p\ge 0$, $k\in\Z$, Pantev et al.\ \cite{PTVV} define complexes of $k$-{\it shifted\/ $p$-forms\/} $\cA^p_\K(\bX,k)$ and $k$-{\it shifted closed\/ $p$-forms\/} $\cA^{p,\cl}_\K(\bX,k)$. These are defined first for affine derived $\K$-schemes $\bY=\bSpec A^\bu$ for $A^\bu$ a cdga over $\K$, and shown to satisfy \'etale descent. Then for general $\bX$, $k$-shifted (closed) $p$-forms are defined as a mapping stack; basically, a $k$-shifted (closed) $p$-form $\om$ on $\bX$ is the functorial choice for all $\bY,\bs f$ of a $k$-shifted (closed) $p$-form $\bs f^*(\om)$ on $\bY$ whenever $\bY=\bSpec A^\bu$ is affine and $\bs f:\bY\ra\bX$ is a morphism. 

\begin{dfn} 
\label{ln2def5}
Let $\bY\simeq\bSpec A^\bu$ be an affine derived $\K$-scheme, for $A^\bu$ a cdga over $\K$. A $k$-{\it shifted\/ $p$-form\/} on $\bY$ for $k\in\Z$ is an element $\om_{A^\bu}^0\in(\La^p\bL_{A^\bu})^k$ with $\d\om_{A^\bu}^0=0$ in $(\La^p\bL_{A^\bu})^{k+1}$, so that $\om_{A^\bu}^0$ defines a cohomology class $[\om_{A^\bu}^0]\in H^k(\La^p\bL_{A^\bu})$. When $p=2$, we call $\om_{A^\bu}^0$ {\it nondegenerate\/} if the induced morphism $\om_{A^\bu}^0\cdot:\bT_{A^\bu}\ra\bL_{A^\bu}[k]$ is a quasi-isomorphism.

A {\it $k$-shifted closed\/ $p$-form\/} on $\bY$ is a sequence $\om_{A^\bu}=(\om^0_{A^\bu},\om^1_{A^\bu},\om^2_{A^\bu},\ldots)$ such that $\om^i_{A^\bu}\in(\La^{p+i}\bL_{A^\bu})^{k-i}$ for $i\ge 0$, with $\d\om_{A^\bu}^0=0$ and $\d\om_{A^\bu}^{1+i}+\dd\om_{A^\bu}^{i}=0$ in $(\La^{p+i+1}\bL_{A^\bu})^{k-i}$ for all $i\ge 0$. Note that if $\om_{A^\bu}=(\om^0_{A^\bu},\om^1_{A^\bu},\ldots)$ is a $k$-shifted closed $p$-form then $\om^0_{A^\bu}$ is a $k$-shifted $p$-form. 

When $p=2$, we call a $k$-shifted closed 2-form $\om_{A^\bu}$ a $k$-{\it shifted symplectic form\/} if the associated 2-form $\om^0_{A^\bu}$ is nondegenerate.

If $\bX$ is a general derived $\K$-scheme, then Pantev et al. \cite[\S 1.2]{PTVV} define $k$-{\it shifted\/ $2$-forms\/} $\om_\bX^0$, which may be {\it nondegenerate}, and $k$-{\it shifted closed\/ $2$-forms\/} $\om_\bX$, which have an associated $k$-shifted 2-form $\om_\bX^0$, and where $\om_\bX$ is called a $k$-{\it shifted symplectic form\/} if $\om^0_\bX$ is nondegenerate. We will not go into the details of this definition for general~$\bX$. 

The important thing for us is this: if $\bY\subseteq\bX$ is a Zariski open affine derived $\K$-subscheme with $\bY\simeq\bSpec A^\bu$ then a $k$-shifted symplectic form $\om_\bX$ on $\bX$ induces a $k$-shifted symplectic form $\om_{A^\bu}$ on $\bY$ in the sense above, where $\om_{A^\bu}$ is unique up to cohomology in the complex~$(\prod_{i\ge 0}(\La^{2+i}\bL_{A^\bu})^{*-i},\d+\dd)$. 
\end{dfn}

As in \cite[\S 2.1]{PTVV}, in the stacky case, an important source of examples of shifted symplectic derived stacks are Calabi--Yau moduli stacks:

\begin{thm} 
\label{ln2thm2}
Suppose $Y$ is a Calabi--Yau $m$-fold over $\K,$ and\/ $\bs\cM$ the derived moduli stack of complexes of coherent sheaves on $Y$. Then $\bs\cM$ has a natural\/ $(2-m)$-shifted symplectic form $\om_{\bs\cM}$.
\end{thm}

\subsection{Lagrangians in shifted symplectic derived schemes}
\label{ln24}

Following Pantev et al.\ \cite[\S 2.2]{PTVV}, we define:

\begin{dfn} 
\label{ln2def6}
Let $(\bX,\om_\bX)$ be a $k$-shifted symplectic derived $\K$-scheme, and $\bs f:\bs L\ra\bX$ a morphism of derived $\K$-schemes. An {\it isotropic structure\/} on $\bs f$ is a homotopy $h_{\bs L}$ from 0 to $\bs f^*(\om_\bX)$ in the complex $\cA^{2,\cl}_\K(\bs L,k)$, regarded as a simplicial set. Truncating to the first term $\cA^{2,\cl}_\K(\bs L,k)\ra \cA^2_\K(\bs L,k)$ gives a homotopy $h_{\bs L}^0$ from 0 to $\bs f^*(\om_\bX^0)$ in~$\cA^2_\K(\bs L,k)$. 

This induces a 2-commutative diagram in $L_\qcoh(\bs L)$:
\e
\begin{gathered}
\xymatrix@C=60pt@R=15pt{ *+[r]{\bT_{\bs L}} \drrtwocell_{}\omit^{^{h_{\bs L}^0\cdot\,\,\,\,\,\,\,\,\,\,\,\,\,\,\,\,\,}}\omit{} \ar[rr] \ar[d]^{\bT_{\bs f}} && *+[l]{0} \ar[d] \\
 *+[r]{\bs f^*(\bT_\bX)} \ar[r]^{\bs f^*(\om_\bX^0)\cdot}_\simeq & \bs f^*(\bL_\bX[k]) \ar[r]^{\bL_{\bs f}[k]} & *+[l]{\bL_{\bs L}[k].\!} }
\end{gathered}
\label{ln2eq2}
\e
We say that $h_{\bs L}^0$ is {\it nondegenerate\/} if \eq{ln2eq2} is homotopy Cartesian (equivalently, homotopy co-Cartesian), and then we say that $\bs L$ (with its morphism $\bs f:\bs L\ra\bX$ and isotropic structure $h_{\bs L}$) is {\it Lagrangian\/} in $(\bX,\om_\bX)$.

An alternative way to explain the nondegeneracy of $h_{\bs L}^0$ is to note that it induces a natural morphism $\chi:\bT_{\bs L/\bX}\ra\bL_{\bs L}[k-1]$ via the diagram
\e
\begin{gathered}
\xymatrix@C=60pt@R=15pt{ *+[r]{\bT_{\bs L}} \drrtwocell_{}\omit^{^{h_{\bs L}^0\cdot\,\,\,\,\,\,\,\,\,\,\,\,\,\,\,\,\,}}\omit{} \ar[rr] \ar[d]^{\bT_{\bs f}} && *+[l]{0} \ar[d] \\
 *+[r]{\bs f^*(\bT_\bX)} \ar[r]^{\bs f^*(\om_\bX^0)\cdot}_\simeq \ar[d] & \bs f^*(\bL_\bX[k]) \ar[r]^{\bL_{\bs f}[k]} & *+[l]{\bL_{\bs L}[k]} \\
*+[r]{\bT_{\bs L/\bX}[1],} \ar@{.>}[urr]_{\chi[1]}
 }
\end{gathered}
\label{ln2eq3}
\e
and $h_{\bs L}^0$ is nondegenerate if $\chi:\bT_{\bs L/\bX}\ra\bL_{\bs L}[k-1]$ is a quasi-isomorphism.

Now suppose that $\bX\simeq\bSpec A^\bu$ and $\bs L\simeq\bSpec B^\bu$ are affine, and $\bs f$ is induced by a morphism $\al:A^\bu\ra B^\bu$ in $\cdga_\K$, and $\om_\bX$ lifts to $\om_{A^\bu}=(\om^0_{A^\bu},\om^1_{A^\bu},\om^2_{A^\bu},\ldots)$ in $(\prod_{i\ge 0}(\La^{2+i}\bL_{A^\bu}[k])^{*-i},\d+\dd)$ as in Definition \ref{ln2def5}. Then we can write $h_{\bs L}$ as a sequence $(h^0,h^1,h^2,\ldots)$ with $h^i\in(\La^{2+i}\bL_{B^\bu})^{1+k-i}$ for $i=0,1,\ldots,$ where $h_{\bs L}$ an isotropic structure is equivalent to the equations
\e
\al_*(\om_{A^\bu}^0)=\d h^0,\qquad \al_*(\om^i_{A^\bu})=\d h^i+\dd h^{i-1},\qquad i=1,2,\ldots.
\label{ln2eq4}
\e

\end{dfn}

\begin{rem} 
\label{ln2rem2}
Let us discuss virtual dimensions of shifted symplectic derived $\K$-schemes and their Lagrangians. If $(\bX,\om_\bX)$ is a $k$-shifted symplectic derived $\K$-scheme, it is easy to show (e.g. using the `Darboux Theorem' in \S\ref{ln25}) that
\begin{itemize}
\setlength{\itemsep}{0pt}
\setlength{\parsep}{0pt}
\item[(i)] If $k\equiv 0\mod 4$ then $\vdim\bX$ is even in $\Z$.
\item[(ii)] If $k\equiv 1\mod 4$ then $\vdim\bX=0$.
\item[(iii)] If $k\equiv 2\mod 4$ then $\vdim\bX$ can take any value in $\Z$.
\item[(iv)] If $k\equiv 3\mod 4$ then $\vdim\bX=0$.
\end{itemize}

Now suppose $\bs f:\bs L\ra\bX$, $h_{\bs L}$ is Lagrangian in $(\bX,\om_\bX)$. Then we find that:
\begin{itemize}
\setlength{\itemsep}{0pt}
\setlength{\parsep}{0pt}
\item[(i$)'$] If $k\equiv 0\mod 4$ then $\vdim\bs L=\ha\vdim\bX$.
\item[(ii$)'$] If $k\equiv 1\mod 4$ then $\vdim\bs L$ can take any even value in $\Z$.
\item[(iii$)'$] If $k\equiv 2\mod 4$ then $\vdim\bX$ must be even (at least near the image of $\bs L$ in $\bX$), and $\vdim\bs L=\ha\vdim\bX$.
\item[(iv$)'$] If $k\equiv 3\mod 4$ then $\vdim\bs L$ can take any value in $\Z$.
\end{itemize}
So if $k\equiv 2\mod 4$ and $\vdim\bX$ is odd then no Lagrangians exist in $(\bX,\om_\bX)$.
\end{rem}

\begin{ex} Take $\bX=\Spec\K$ to be the point $*$, regarded as a $k$-shifted symplectic derived $\K$-scheme with symplectic form $\om_\bX=0$. Then Lagrangians $\bs L$ in $(*,0)$ are equivalent to $(k-1)$-shifted symplectic derived $\K$-schemes~$(\bs L,\om_{\bs L})$.

\label{ln2ex1}
\end{ex}

Pantev et al.\ \cite[Th.~2.10]{PTVV} prove:

\begin{thm} 
\label{ln2thm3}
Suppose $(\bX,\om_\bX)$ is a $k$-shifted symplectic derived\/ $\K$-scheme, and\/ $\bs f_1:\bs L_1\ra\bX$ and\/ $\bs f_2:\bs L_2\ra\bX$ are Lagrangians in $(\bX,\om_\bX)$. Then the fibre product $\bs L_1\t_{\bs f_1,\bX,\bs f_2}\bs L_2$ in $\dSch_\K$ has a natural\/ $(k-1)$-shifted symplectic structure.
\end{thm}

In the stacky case, Calaque \cite[\S 3.2]{Cala} extends Theorem~\ref{ln2thm2}:

\begin{thm} 
\label{ln2thm4}
Suppose\/ $X$ is a Fano $(m+1)$-fold over $\K,$ and\/ $Y\subseteq X$ is a smooth anticanonical divisor, so that\/ $Y$ is a Calabi--Yau $m$-fold. Write $\bs\cL,\bs\cM$ for the derived moduli stacks of complexes of coherent sheaves on $X,Y,$ and\/ $\bs f:\bs\cL\ra\bs\cM$ for the morphism of derived restriction from $X$ to $Y$. Theorem\/ {\rm\ref{ln2thm2}} gives a $(2-m)$-shifted symplectic structure $\om_{\bs\cM}$ on $\bs\cM$. Then there is a natural isotropic structure $h_{\bs\cL}$ on $\bs f:\bs\cL\!\ra\!\bs\cM$ making $\bs\cL$ into a Lagrangian in~$(\bs\cM,\om_{\bs\cM})$.

\end{thm}

\subsection{\texorpdfstring{A shifted symplectic `Darboux Theorem'}{A shifted symplectic \textquoteleft Darboux Theorem\textquoteright}}
\label{ln25}

Bussi, Brav and Joyce \cite{BBJ} prove `Darboux Theorems' for $k$-shifted symplectic derived $\K$-schemes $(\bX,\om_\bX)$ for $k<0$, which give explicit Zariski or \'etale local models for $(\bX,\om_\bX)$. We will explain their main result in Theorem \ref{ln2thm5} below. First, in Examples \ref{ln2ex2} and \ref{ln2ex3} we define families of explicit `Darboux form' $k$-shifted symplectic cdgas $A^\bu,\om$ for~$k<0$. 

\begin{ex} 
\label{ln2ex2}
Let $k=-1,-2,\ldots,$ and set $d=[(k+1)/2]$, so that $d=k/2$ if $k$ is even (giving $k=2d$), and $d=(k+1)/2$ if $k$ is odd (giving $k=2d-1$). Following \cite[Examples 5.8 \& 5.9]{BBJ}, we will define a simple class of standard form cdgas $A^\bu=(A^*,\d)$ equipped with explicit $k$-shifted symplectic forms $\om=(\om^0,0,0,\ldots)$, which we will call of {\it Darboux form}.

Fix nonnegative integers $m_0,m_{-1},m_{-2},\ldots,m_d$. Choose a smooth $\K$-algebra $A^0$ of dimension $m_0$. Localizing $A^0$ if necessary, we may assume that there exist $x^0_1,\ldots,x^0_{m_0}\in A^0$ such that $\dd x^0_1,\ldots,\dd x^0_{m_0}$ form a basis of $\Om^1_{A^0}$ over $A^0$. Geometrically, $U=\Spec A^0$ is a smooth $\K$-scheme of dimension
$m_0$, and $(x^0_1,\ldots,x^0_{m_0}):U\ra\bA^{m_0}$ are global \'etale coordinates on~$U$.

Define $A^*$ as a commutative graded $\K$-algebra to be the free graded algebra over $A^0$
generated by variables
\e
\begin{aligned}
&x_1^i,\ldots,x^i_{m_i} &&\text{in degree $i$ for
$i=-1,-2,\ldots,d$, and} \\
&y_1^{k-i},\ldots,y^{k-i}_{m_i} &&\text{in degree $k-i$ for
$i=0,-1,\ldots,d$.}
\end{aligned}
\label{ln2eq5}
\e
So the upper index $i$ in $x^i_j,y^i_j$ always indicates the degree. The variables come in pairs $x^i_j,y_j^{k-i}$, with total degree $k$. We will define the differential $\d$ in the cdga $A^\bu=(A^*,\d)$ later.

As in \S\ref{ln21}, the spaces $(\La^p\Om^1_{A^\bu})^k$ and the de Rham differential $\dd$ upon them depend only on the commutative graded algebra $A^*$, not on the (not yet defined)
differential $\d$. Note that $\Om^1_{A^\bu}$ is the free $A^*$-module with
basis $\dd x^i_j,\dd y^{k-i}_j$ for $i=0,-1,\ldots,d$ and
$j=1,\ldots,m_i$. Define an element
\e
\om^0=\sum_{i=0}^d\sum_{j=1}^{m_i}\dd x^i_j\,\dd y^{k-i}_j
\qquad \text{in $(\La^2\Om^1_{A^\bu})^k$.}
\label{ln2eq6}
\e
Clearly $\dd\om^0=0$ in $(\La^3\Om^1_{A^\bu})^k$.

Now choose a superpotential $\Phi$ in $A^{k+1}$, called the {\it Hamiltonian}, which we require to satisfy the {\it classical master equation\/}
\e
\sum_{i=-1}^d\sum_{j=1}^{m_i}\frac{\pd\Phi}{\pd x^i_j}\,
\frac{\pd\Phi}{\pd y^{k-i}_j}=0\qquad\text{in $A^{k+2}$.}
\label{ln2eq7}
\e
Define the differential $\d$ on $A^*$ by $\d=0$ on $A^0$, and
\e
\d x^i_j =(-1)^{(i+1)(k+1)}\,\frac{\pd\Phi}{\pd y^{k-i}_j}, \quad \d
y^{k-i}_j=\,\frac{\pd\Phi}{\pd x^i_j},\quad\begin{subarray}{l}\ts
i=0,\ldots,d,\\[6pt] \ts j=1,\ldots,m_i.\end{subarray}
\label{ln2eq8}
\e
Equation \eq{ln2eq7} implies that $\d\ci\d=0$.

Then $A^\bu=(A^*,\d)$ is a standard form cdga, as in Definition \ref{ln2def3}, with
\begin{equation*}
\vdim A^\bu=\begin{cases} 2\sum_{i=0}^d(-1)^im_i, & \text{$k$ even,} \\ 
0, & \text{$k$ odd,}\end{cases}
\end{equation*}
so that $\vdim A^\bu$ is always even (compare Remark \ref{ln2rem2}). Also $\d\om^0=\dd\om^0=0$, and $\om:=(\om^0,0,0,\ldots)$ is a $k$-shifted symplectic structure on $\bX=\bSpec A^\bu$, as in \cite[\S 5.3]{BBJ}. Define $\phi\in(\Om^1_{A^\bu})^k$ by 
\e
\phi=\sum_{i=0}^d\sum_{j=1}^{m_i}\bigl[i\,x^i_j\,\dd y^{k-i}_j+(-1)^{(i+1)(k+1)}(k-i)y^{k-i}_j\,\dd x^i_j\bigr].
\label{ln2eq9}
\e
Then we have
\e
\d\Phi=0\in A^{k+2},\;\> \dd\Phi+\d\phi=0\in (\Om^1_{A^\bu})^{k+1},\;\>
\dd\phi=k\om^0\in(\La^2\Om^1_{A^\bu})^k.
\label{ln2eq10}
\e
We say that $A^\bu,\om$ are in {\it Darboux form}.

In the first case $k=-1$, as in \cite[Prop.~5.7(b)]{BBJ} we impose an additional condition on $\Phi$. In this case $\Phi:U\ra\bA^1$ is a regular function, and $\bX=\bs\Crit(\Phi)$ is the derived critical locus of $\Phi$, so $X=t_0(\bX)=\Crit(\Phi)$ is the classical critical locus of $\Phi$. The restriction $\Phi\vert_{X^\red}:X^\red\ra\bA^1$ of $\Phi$ to the reduced $\K$-subscheme $X^\red$ of $X$ is locally constant. By adding a constant to $\Phi$ (which does not change $\bX$) and localizing, we may assume that~$\Phi\vert_{X^\red}=0$.
\end{ex}

\begin{rem} 
\label{ln2rem3}
Continue in the situation of Example \ref{ln2ex2}. The following notation was not defined in \cite{BBJ}, but will be important in \S\ref{ln3}--\S\ref{ln4}. Define $A^\bu_+$ to be the sub-cdga of $A^\bu$ generated (either as a cdga or equivalently as a graded algebra) by $A^0$ and the variables $x^i_j$ for $i=-1,-2,\ldots,d$ and $j=1,\ldots,m_i$. Then $A^\bu_+$ is of standard form with $\vdim A^\bu_+=\sum_{i=0}^d(-1)^im_i$. Write $\io:A^\bu_+\hookra A^\bu$ for the inclusion morphism, which is a submersion. Then we have a fibre sequence
\begin{equation*}
\xymatrix@C=40pt{ \bL_{A^\bu_+}\ot_{A^\bu_+}A^\bu \ar[r]^(0.6){\bL_\io} & \bL_{A^\bu} \ar[r] &  \bL_{A^\bu/A^\bu_+}. }
\end{equation*}
Taking $\bL_{A^\bu}=\Om^1_{A^\bu}$, $\bL_{A^\bu_+}=\Om^1_{A^\bu_+}$, and $\bL_{A^\bu/A^\bu_+}=\Om^1_{A^\bu/A^\bu_+}$ as $A^\bu,A^\bu_+$ are of standard form and $\io$ is a submersion, as in \eq{ln2eq1} we have 
\begin{align*}
\bL_{A^\bu}&\cong A^*\ot_\K\an{\dd x_j^i,\dd y_j^{k-i}:i\!=\!0,-1,\ldots,d,\; j\!=\!1,\ldots,m_i}_\K,\!\!\\
\bL_{A^\bu_+}\!\ot_{A^\bu_+}\!A^\bu&\cong A^*\ot_\K\an{\dd x_j^i:i=0,-1,\ldots,d,\; j=1,\ldots,m_i}_\K,\\
\bL_{A^\bu/A^\bu_+}&\cong A^*\ot_\K\an{\dd y_j^{k-i}:i=0,-1,\ldots,d,\; j=1,\ldots,m_i}_\K.
\end{align*}

We will also find it helpful to decompose $\Phi\in A^{k+1}$ into components. Observe that as $\deg(y_j^{k-i})\le k-d$ and $2(k-d)<k+1$, for degree reasons $\Phi$ can be at most linear in the variables $y_j^{k-i}$, so we may write
\e
\Phi=\Phi_++\sum_{i=-1}^d\sum_{j=1}^{m_i}\Phi_j^{i+1}y_j^{k-i},
\label{ln2eq11}
\e
where $\Phi_+\in A_+^{k+1}$ and $\Phi_j^{i+1}\in A_+^{i+1}$ for all $i,j$ do not involve the variables $y^i_j$. Then equation \eq{ln2eq7} is equivalent to the equations
\ea
\sum_{i=-1}^d\sum_{j=1}^{m_i}(-1)^{i+1}\Phi_j^{i+1}\,\frac{\pd\Phi_+}{\pd x^i_j}=0\qquad\text{in $A_+^{k+2}$,}
\label{ln2eq12}\\
\sum_{i=-1}^{i'+1}\sum_{j=1}^{m_i}(-1)^{i+1}\Phi_j^{i+1}\,\frac{\pd\Phi_{j'}^{i'+1}}{\pd x^i_j}\,=0\quad\text{in $A_+^{i'+2}$,}\;\>
\begin{subarray}{l}\ts
i'=-1,\ldots,d,\\[6pt] \ts j'=1,\ldots,m_{i'},\end{subarray}
\label{ln2eq13}
\ea
and \eq{ln2eq8} is equivalent to the equations for $i=0,\ldots,d$ and $j=1,\ldots,m_i$:
\e
\d x^i_j=(-1)^{i+1}\Phi_j^{i+1}, \;\> \d
y^{k-i}_j=\frac{\pd\Phi_+}{\pd x^i_j}+\sum_{i'=i-1}^d\sum_{j'=1}^{m_{i'}}
\frac{\pd\Phi_{j'}^{i'+1}}{\pd x^i_j}\,y_{j'}^{k-i'}.
\label{ln2eq14}
\e

Define
\e
\phi_+=-\sum_{i=0}^d\sum_{j=1}^{m_i}(-1)^{(i+1)(k+1)}y^{k-i}_j\,\dd x^i_j\qquad\text{in $(\Om^1_{A^\bu})^k$.}
\label{ln2eq15}
\e
Then as for \eq{ln2eq10}, calculation shows that
\e
\d\Phi_+=0,\quad \dd\Phi_++\d\phi_+=0,\quad\text{and}\quad
\dd\phi_+=-\om^0.
\label{ln2eq16}
\e

A nice interpretation of $\io:A^\bu_+\ra A^\bu$, which we will not actually use, is that $\bSpec\io:\bSpec A^\bu\ra\bSpec A^\bu_+$ is a {\it Lagrangian fibration\/} of~$(\bSpec A^\bu,\om)$. 

We can also use this example to explain the relation between the `Darboux Theorems' of Bussi, Brav and Joyce \cite{BBJ}, and Bouaziz and Grojnowski \cite{BoGr}. Bouaziz and Grojnowski show that any $k$-shifted symplectic derived $\K$-scheme $(\bX,\om_\bX)$ for $k<0$ with $k\not\equiv 2\mod 4$ is \'etale locally equivalent to a {\it twisted\/ $k$-shifted cotangent bundle\/} $T^*_t[k]\bY$, where $\bY$ is an affine derived $\K$-scheme, and $t\in\O_\bY^{k+1}$ with $\d t=0$ is used to `twist' the $k$-shifted cotangent bundle~$T^*[k]\bY$.

To make the two pictures correspond, we should identify $\bY$ with $\bSpec A^\bu_+$, and $\Phi_+\in A_+^{k+1}$ with $t\in \O_\bY^{k+1}$. The data $\Phi_j^{i+1}$ in $\Phi$ in \eq{ln2eq11} is used to define the differential $\d$ in $A^\bu_+=(A^*_+,\d)$, via $\d x^i_j=(-1)^{i+1}\Phi_j^{i+1}$ in \eq{ln2eq14}. The classical master equation \eq{ln2eq7} reduces to \eq{ln2eq12}--\eq{ln2eq13}, where \eq{ln2eq12} means that $\d\Phi_+=0$, and \eq{ln2eq13} means that $\d\ci\d=0$ in $A^\bu_+=(A^*_+,\d)$, necessary for $A_+^\bu$ to be a cdga and $\bY$ a derived scheme.

Remark \ref{ln3rem1} will explain how our `Lagrangian Neighbourhood Theorem' relates to Bouaziz and Grojnowski's `twisted cotangent bundle' picture.
\end{rem}

Bussi, Brav and Joyce \cite[Examples 5.10 \& 5.12]{BBJ} also give two variations on Example \ref{ln2ex2} when~$k\equiv 2\mod 4$:

\begin{ex} 
\label{ln2ex3}
Let $k=-2,-6,-10,\ldots$ be negative with $k\equiv 2\mod 4$, and set $d=k/2$, so that $d$ is negative and odd. Fix nonnegative integers $m_0,m_{-1},m_{-2},\ldots,m_d$. Choose $A^0,x^0_1,\ldots,x^0_{m_0}$ and $U$ as in Example~\ref{ln2ex2}.

Modifying \eq{ln2eq5}, define $A^*$ as a commutative graded $\K$-algebra to be the free graded algebra over $A^0$ generated by variables
\begin{align*}
&x_1^i,\ldots,x^i_{m_i} &&\text{in degree $i$ for
$i=-1,-2,\ldots,d+1$, and} \\
&z_1^d,\ldots,z^d_{m_d} &&\text{in degree $d$, and} \\
&y_1^{k-i},\ldots,y^{k-i}_{m_i} &&\text{in degree $k-i$ for
$i=0,-1,\ldots,d+1$.}
\end{align*}
Let $q_1,\ldots,q_{m_d}$ be invertible elements of $A^0$, and
generalizing \eq{ln2eq6} define
\e
\om^0=\sum_{i=0}^{d+1}\sum_{j=1}^{m_i}\dd
x^i_j\,\dd y^{k-i}_j+\sum_{j=1}^{m_d}\dd\bigl(q_j z^d_j\bigr)\,\dd
z^d_j\qquad \text{in $(\La^2\Om^1_{A^\bu})^k$.}
\label{ln2eq17}
\e

Choose a Hamiltonian $\Phi$ in $A^{k+1}$, which as in \eq{ln2eq7} we require to satisfy the {\it classical master equation\/}
\e
\sum_{i=-1}^{d+1}\sum_{j=1}^{m_i}\frac{\pd\Phi}{\pd x^i_j}\,
\frac{\pd\Phi}{\pd y^{k-i}_j}+\frac{1}{4}\sum_{j=1}^{m_d}
\frac{1}{q_j}\,\biggl(\frac{\pd\Phi}{\pd
z^d_j}\biggr)^2=0\quad\text{in $A^{k+2}$.}
\label{ln2eq18}
\e
As for \eq{ln2eq8}, define the differential $\d$ on $A^\bu$ by $\d=0$
on $A^0$, and
\begin{gather}
\d x^0_j=0,\quad \d y^k_j=\frac{\pd\Phi}{\pd
x^0_j}-\sum_{j'=1}^{m_d} \frac{z_{j'}^d}{2q_{j'}}\,\frac{\pd
q_{j'}}{\pd x^0_j}\,\frac{\pd\Phi}{\pd z^d_{j'}}, \quad
j=1,\ldots,m_0,
\nonumber\\
\d x^i_j =(-1)^{i+1}\frac{\pd\Phi}{\pd y^{k-i}_j}, \quad \d
y^{k-i}_j=\frac{\pd\Phi}{\pd x^i_j},\quad\begin{subarray}{l}\ts
i=-1,\ldots,d+1,\\[6pt] \ts j=1,\ldots,m_i,\end{subarray}
\nonumber\\
\text{and}\qquad\d z^d_j=\frac{1}{2q_j}\,\frac{\pd\Phi}{\pd
z^d_j}, \qquad j=1,\ldots,m_d.
\label{ln2eq19}
\end{gather}

Then $A^\bu=(A^*,\d)$ is of standard form, with $\vdim A^\bu=2\sum_{i=0}^{d+1}(-1)^im_i-m_d$. Also $\d\om^0=\dd\om^0=0$, and $\om:=(\om^0,0,0,\ldots)$ is a $k$-shifted symplectic structure on $\bX=\bSpec A^\bu$, as in \cite[\S 5.3]{BBJ}. Defining $\phi\in (\Om^1_{A^\bu})^k$ by
\e
\phi\!=\!\sum_{i=0}^{d+1}\sum_{j=1}^{m_i}\bigl[i\,x^i_j\,\dd y^{k-i}_j\!+\!(-1)^{i+1}(k\!-\!i)y^{k-i}_j\,\dd x^i_j\bigr]\!+\!k\sum_{j=1}^{m_d}q_j\,z^d_j\,\dd z^d_j,
\label{ln2eq20}
\e
as in \eq{ln2eq9}, then \eq{ln2eq10} holds. We say that $A^\bu,\om$ are in {\it weak Darboux form}.

If all the above holds with $q_j=1$ for $j=1,\ldots,m_d$, we say that $A^\bu,\om$ are in {\it strong Darboux form}.
\end{ex}

\begin{rem} 
\label{ln2rem4}
Actually, when $k\equiv 2\mod 4$, Bussi, Brav and Joyce \cite[\S 5.3]{BBJ} did not define `Darboux form' $A^\bu,\om$ as in Example \ref{ln2ex2} involving only variables $x^i_j,y^{k-i}_j$, but instead only defined `weak Darboux form' and `strong Darboux form' as in Example \ref{ln2ex3}, involving variables $x^i_j,y^{k-i}_j,z^d_j$.

We can relate Example \ref{ln2ex3} to Example \ref{ln2ex2} with $k\equiv 2\mod 4$. Let $A^\bu,\om$ be in strong Darboux form as in Example \ref{ln2ex3}, so that we have variables $x^i_j,y^{k-i}_j,z^d_j$ and $q_1=\cdots=q_{m_d}=1$, and suppose $m_d$ is even (equivalently, suppose $\vdim A^\bu$ is even). Then we may change variables from $z_1^d,\ldots,z_{m_d}^d$ to $x_1^d,\ldots,x_{m_d/2}^d,y_1^d,\ldots,y_{m_d/2}^d$, where
\begin{equation*}
x_j^d=z_{2j-1}^d+\sqrt{-1}z_{2j}^d,\quad y_j^d=z_{2j-1}^d-\sqrt{-1}z_{2j}^d,\quad j=1,\ldots,m_d/2,
\end{equation*}
and replace $m_d$ by $m_d/2$, and then the `strong Darboux form' of Example \ref{ln2ex3} is equivalent to the `Darboux form' of Example \ref{ln2ex2}. Here $\sqrt{-1}\in\K$ as $\K$ is algebraically closed.

As in Remark \ref{ln2rem2}, if $(\bX,\om_\bX^\bu)$ is $k$-shifted symplectic with $k\equiv 2\mod 4$ and $\vdim\bX$ is odd, then no Lagrangians exist in $(\bX,\om_\bX)$. Because of this, in this paper we are happy to use the local form of Example \ref{ln2ex2} when $k\equiv 2\mod 4$, which only works when $\vdim\bX$ is even, rather than that of Example \ref{ln2ex3}, which works for all~$\vdim\bX$.
\end{rem}

Here is the main result of Bussi, Brav and Joyce \cite[Th.~5.18]{BBJ}. They state only (i)--(iii), part (iv) is deduced from (iii) as in Remark \ref{ln2rem4}. The reason we need $\bs i$ to be \'etale in (iii) (and hence (iv)) is that to reduce from weak Darboux form to strong Darboux form in Example \ref{ln2ex3}, it is necessary to take square roots of the functions $q_1,\ldots,q_{m_d}$, and this is only possible \'etale locally rather than Zariski locally.

\begin{thm} 
\label{ln2thm5}
Let\/ $(\bX,\om_\bX)$ be a $k$-shifted symplectic derived\/ $\K$-scheme for $k<0,$ and\/ $x\in\bX$. Then there exists a standard form cdga $A^\bu$ over $\K$ which is minimal at\/ $p\in\bSpec A^\bu,$ a $k$-shifted symplectic form\/ $\om$ on $\bSpec A^\bu,$ and a morphism $\bs i:\bSpec A^\bu\ra\bX$ with\/ $\bs i(p)=x$ and\/ $\bs i^*(\om_\bX)\sim\om,$ such that:
\begin{itemize}
\setlength{\itemsep}{0pt}
\setlength{\parsep}{0pt}
\item[{\bf(i)}] If\/ $k\equiv 0,1$ or\/ $3\mod 4,$ then $\bs i$ is
a Zariski open inclusion, and\/ $A^\bu,\om$ are in Darboux form, as
in Example\/ {\rm\ref{ln2ex2}}.
\item[{\bf(ii)}] If\/ $k\equiv 2\mod 4,$ then $\bs i$ is a Zariski open inclusion, and\/ $A^\bu,\om$ are in weak Darboux form, as in Example\/~{\rm\ref{ln2ex3}}.
\item[{\bf(iii)}] Alternatively, if\/ $k\equiv 2\mod 4,$ we may
instead take $\bs i$ to be \'etale, and\/ $A^\bu,\om$ to be in strong Darboux form, as in Example\/~{\rm\ref{ln2ex3}}.
\item[{\bf(iv)}] Alternatively, if\/ $k\equiv 2\mod 4,$ and\/ $\vdim\bX$ is even near $x,$ we may take $\bs i$ to be \'etale, and\/ $A^\bu,\om$ to be in Darboux form, as in Example\/~{\rm\ref{ln2ex2}}.
\end{itemize}

\end{thm}

Following \cite[Examples 5.15 \& 5.16]{BBJ}, we explain Examples \ref{ln2ex2} and \ref{ln2ex3} in more detail in the first cases $k=-1$ and~$k=-2$.

\begin{ex} 
\label{ln2ex4}
Choose a smooth $\K$-algebra $A^0$ of dimension $m_0$
and elements $x^0_1,\ldots,x^0_{m_0}\in A^0$ such that $\dd
x^0_1,\ldots,\dd x^0_{m_0}$ form a basis of $\Om^1_{A^0}$ over
$A^0$. Choose an arbitrary Hamiltonian $\Phi\in A^0$.

Example \ref{ln2ex2} with $k=-1$ defines $A^\bu=A^0[y^{-1}_1,\ldots,
y^{-1}_{m_0}]$, where $y^{-1}_1,\ab\ldots,\ab y^{-1}_{m_0}$ are
variables of degree $-1$, with differential
\begin{equation*}
\d x^0_j=0,\quad \d y^{-1}_j=\frac{\pd\Phi}{\pd x^0_j}, \quad
j=1,\ldots,m_0,
\end{equation*}
and $-1$-shifted 2-form
\begin{equation*}
\om^0=\dd x^0_1\,\dd y^{-1}_1+\cdots+\dd x^0_{m_0}\,\dd y^{-1}_{m_0}.
\end{equation*}
Then $\om=(\om^0,0,0,\ldots)$ is a $-1$-shifted symplectic structure
on $\bX=\bSpec A^\bu$. We have $H^0(A^\bu)=A^0/(\frac{\pd\Phi}{\pd
x^0_1},\ldots,\frac{\pd\Phi}{\pd x^0_{m_0}})=A^0/(\dd\Phi)$.

Geometrically, $U=\Spec A^0$ is a smooth classical $\K$-scheme with \'etale coordinates $(x^0_1,\ldots,x^0_{m_0}):U\ra\bA^{m_0}$, and $\Phi:U\ra\bA^1$ is regular, and $\bX$ is the derived critical locus of $\Phi$, with $X=t_0(\bX)$ the classical critical locus of~$\Phi$.

Thus, {\it the important geometric data in writing a $-1$-shifted symplectic derived\/ $\K$-scheme $(\bX,\om)$ in Darboux form, is a smooth affine\/ $\K$-scheme $U$ and a regular function $\Phi:U\ra\bA^1,$ such that\/} $\bX\simeq\bs\Crit(\Phi)$. The remaining data is a choice of \'etale coordinates $(x^0_1,\ldots,x^0_{m_0}):U\ra \bA^{m_0}$, but this is not very
interesting geometrically.
\end{ex}

\begin{ex} 
\label{ln2ex5}
Choose a smooth $\K$-algebra $A^0$ of dimension $m_0$ and elements $x^0_1,\ldots,x^0_{m_0}$ in $A^0$ such that $\dd x^0_1,\ldots,\dd x^0_{m_0}$ form a basis of $\Om^1_{A^0}$ over $A^0$. Fix $m_{-1}\ge 0$, and as a graded algebra set
$A^*=A^0[y^{-2}_1,\ab\ldots,\ab y^{-2}_{m_0},\ab z^{-1}_1,\ab\ldots,\ab z^{-1}_{m_{-1}}]$, where $y^{-2}_j$ has degree $-2$ and $z^{-1}_j$ degree $-1$, as in Example \ref{ln2ex3} with~$k=-2$.

Choose invertible functions $q_1,\ldots,q_{m_{-1}}$ in $A^0$. Define
\begin{align*}
\om^0&=\dd x^0_1\,\dd y^{-2}_1+\cdots+\dd x^0_{m_0}\,\dd y^{-2}_{m_0} \\
&\qquad+\dd\bigl(q_1z^{-1}_1\bigr)\,\dd z^{-1}_1+\cdots+
\dd\bigl(q_{m_{-1}}z^{-1}_{m_{-1}}\bigr)\,\dd z^{-1}_{m_{-1}}
\end{align*}
in $(\La^2\Om^1_{A^\bu})^{-2}$, as in \eq{ln2eq17}. A general element $\Phi$
in $A^{-1}$ may be written
\begin{equation*}
\Phi=z^{-1}_1s_1+\cdots+z^{-1}_{m_{-1}}s_{m_{-1}},
\end{equation*}
for $s_1,\ldots,s_{m_{-1}}\in A^0$. Then the classical master equation
\eq{ln2eq18} reduces to
\e
\frac{(s_1)^2}{q_1}+\cdots+\frac{(s_{m_{-1}})^2}{q_{m_{-1}}}=0\quad\text{in
$A^0$.}
\label{ln2eq21}
\e
By \eq{ln2eq19}, the differential $\d$ on $A^\bu$ is given by
\begin{equation*}
\d x_i^0=0,\quad \d z_j^{-1}=\frac{s_j}{2q_j}, \quad \d y^{-2}_i=
\sum_{j=1}^{m_{-1}}z_j^{-1}\biggl(\frac{\pd s_j}{\pd x^0_i}
-\frac{s_j}{2q_j}\,\frac{\pd q_j}{\pd x^0_i}\biggr),
\end{equation*}
and $\d\ci\d y^{-2}_i=0$ follows from applying
$\frac{1}{4}\frac{\pd}{\pd x^0_i}$ to \eq{ln2eq21}. We have
\begin{equation*}
H^0(A^\bu)=A^0/\bigl(s_1/2q_1,\ldots,s_{m_{-1}}/2q_{m_{-1}}\bigr)=
A^0/(s_1,\ldots,s_{m_{-1}}),
\end{equation*}
as $q_1,\ldots,q_{m_{-1}}$ are invertible.

Geometrically, we have a smooth classical $\K$-scheme $U=\Spec A^0$
with \'etale coordinates $(x^0_1,\ldots,x^0_{m_0}):U\ra \bA^{m_0}$,
a trivial vector bundle $E\ra U$ with fibre $\K^{m_{-1}}$, a
nondegenerate quadratic form $Q$ on $E$ given by
$Q(e_1,\ldots,e_{m_{-1}})=\frac{1}{q_1}e_1^2+\cdots+\frac{1}{q_{m_{-1}}}e_{m_{-1}}^2$ for all regular functions $e_1,\ldots,e_{m_{-1}}:U\ra\bA^1$, and a section $s=(s_1,\ldots,s_{m_{-1}})$ in $H^0(E)$ with $Q(s,s)=0$ by \eq{ln2eq21}. The underlying classical $\K$-scheme $X=t_0(\bX)=\Spec H^0(A^\bu)$ is the $\K$-subscheme $s^{-1}(0)$ in~$U$.

Thus, {\it the important geometric data in writing a $-2$-shifted
symplectic derived\/ $\K$-scheme $(\bX,\om)$ in weak Darboux form, is a
smooth affine\/ $\K$-scheme $U,$ a vector bundle $E\ra U,$ a
nondegenerate quadratic form $Q$ on $E,$ and a section $s\in H^0(E)$
with\/ $Q(s,s)=0,$ such that\/} $X=t_0(\bX)\cong s^{-1}(0)\subseteq
U$. The remaining data is a choice of \'etale coordinates
$(x^0_1,\ldots,x^0_{m_0}):U\ra \bA^{m_0}$ and a trivialization
$E\cong U\t\bA^{m_{-1}}$, but these are not very interesting
geometrically.
\end{ex}

\section{The main results}
\label{ln3}

\subsection{A local standard form for derived scheme morphisms}
\label{ln31}

As in Theorem \ref{ln2thm1}, our favourite local model for a derived scheme $\bX$ near a point $x\in\bX$ is $\bSpec A^\bu\hookra\bX$ for $A^\bu$ a standard form cdga, and we can take $\bSpec A^\bu$ minimal at $x$. We used this in the Darboux Theorem, Theorem \ref{ln2thm5}.

We will need a favourite local model for a morphism $\bs f:\bY\ra\bX$ in $\dSch_\K$ near $y\in\bY$ with $\bs f(y)=x\in\bX$. For this we will use a homotopy commutative diagram \eq{ln3eq1} below, where $\al:A^\bu\ra B^\bu$ is a {\it submersion\/} in $\cdga_\K$,  following Borisov and Joyce \cite[\S 2.1]{BoJo}, which we take to be {\it minimal\/} at~$\bs j^{-1}(y)\in\bSpec B^\bu$.

\begin{dfn} 
\label{ln3def1}
A morphism $\al:A^\bu\ra B^\bu$ of standard form cdgas will be called a {\it submersion\/} if the corresponding morphism $\al_*:(\Om^1_{A^\bu})\ot_{A^\bu}B^\bu\ra\Om^1_{B^\bu}$ is injective in every degree. (By analogy, a smooth map of manifolds $f:X\ra Y$ is a submersion if $(\d f)^*:f^*(T^*Y)\ra T^*X$ is injective.)

If $\al:A^\bu\ra B^\bu$ is a submersion of standard form cdgas then the relative K\"ahler differentials $\Om^1_{B^\bu/A^\bu}$ are a model for the relative cotangent complex $\bL_{B^\bu/A^\bu}$, so we can take $\Om^1_{B^\bu/A^\bu}=\bL_{B^\bu/A^\bu}$. Thus submersions are a convenient class of morphisms for doing explicit computations with cotangent complexes.

In a similar way to Definition \ref{ln2def3}, we say that a submersion $\al:A^\bu\ra B^\bu$ is  {\it minimal\/} at $q\in\bSpec B^\bu$ if all the differentials in the complex of $\K$-vector spaces $\Om^1_{B^\bu/A^\bu}\vert_q$ are zero. This means that regarding $A^\bu$ as fixed, $B^\bu$ is defined using the minimum number of variables in each degree $i=0,-1,\ldots,$ compared to all other cdgas locally equivalent to $B^\bu$ near $q$ with submersions to $A^\bu$. 
\end{dfn}

Here is a relative analogue of Theorem \ref{ln2thm1}, which will be proved in~\S\ref{ln41}.

\begin{thm} 
\label{ln3thm1}
Let\/ $\bs f:\bY\ra\bX$ be a morphism in $\dSch_\K,$ and\/ $y\in\bY$ with\/ $\bs f(y)=x\in\bX$. Suppose $A^\bu$ is a standard form cdga over $\K,$ and\/ $p\in\bSpec A^\bu,$ and\/ $\bs i:\bSpec A^\bu\hookra\bX$ is a Zariski open inclusion with\/ $\bs f(p)=x$. This is possible by Theorem\/ {\rm\ref{ln2thm2}}. We do not assume $A^\bu$ is minimal at\/~$p$.

Then there exists a standard form cdga $B^\bu$ over $\K,$ a point\/ $q\in\bSpec A^\bu,$ a submersion $\al:A^\bu\ra B^\bu$ minimal at\/ $q$ with\/ $\bSpec\al(q)\!=\!p,$ and a Zariski open inclusion $\bs j:\bSpec B^\bu\hookra\bY$ with\/ $\bs j(q)=y$ in a homotopy commutative diagram
\e
\begin{gathered}
\xymatrix@C=90pt@R=15pt{ *+[r]{\bSpec B^\bu\,} \ar[d]^{\bSpec\al} \ar@{^{(}->}[r]_{\bs j} & *+[l]{\bY} \ar[d]_{\bs f}  \\ *+[r]{\bSpec A^\bu\,} \ar@{^{(}->}[r]^{\bs i}  & *+[l]{\bX.\!\!} }
\end{gathered}
\label{ln3eq1}
\e

If instead\/ $\bs i$ is \'etale rather than a Zariski open inclusion, then $\bs j$ is \'etale.
\end{thm}

\subsection[\texorpdfstring{$k$-shifted `Lagrangian Darboux form' local models for $k<0$}{\textit{k}-shifted \textquoteleft Lagrangian Darboux form\textquoteright\ local models for \textit{k}<0}]{$k$-shifted `Lagrangian Darboux form' local models}
\label{ln32}

In Examples \ref{ln3ex1} and \ref{ln3ex2} we will explain our local models for Lagrangians in $k$-shifted symplectic derived $\K$-schemes, which we call `Lagrangian Darboux form', for the cases $k<0$ with $k\not\equiv 3\mod 4$, and $k<0$ with $k\equiv 3\mod 4$, respectively. They are analogues of the `Darboux form' Examples \ref{ln2ex2} and \ref{ln2ex3} in \S\ref{ln25}, and work over a target in `Darboux form'. Theorem \ref{ln3thm2} in \S\ref{ln33} will show that Lagrangians $\bs f:\bs L\ra\bX$ in a $k$-shifted symplectic $\K$-scheme $(\bX,\om_\bX)$ for $k<0$ are (Zariski or \'etale) locally modelled on one of Examples \ref{ln3ex1} and~\ref{ln3ex2}.

The next example is rather long.

\begin{ex} 
\label{ln3ex1}
Let $k<0$ with $k\not\equiv 3\mod 4$, suppose $A^\bu,\om$ are in $k$-shifted Darboux form, as in Example \ref{ln2ex2}, and use the notation of Remark \ref{ln2rem3}. These define a standard form cdga $A^\bu$ over $\K$, a sub-cdga $A_+^\bu\subseteq A^\bu$, coordinates $x^i_j$ in $A_+^i\subseteq A^i$ and $y^{k-i}_j$ in $A^{k-i}$ for $i=0,-1,\ldots,d=[(k+1)/2]$ and $j=1,\ldots,m_i$, and a $k$-shifted 2-form $\om^0=\sum_{i=0}^d\sum_{j=1}^{m_i}\dd x^i_j\,\dd y^{k-i}_j$. They also define $\Phi\in A^{k+1}$ satisfying \eq{ln2eq7}, which determines the differential $\d$ in $A^\bu$ by \eq{ln2eq8}, and $\phi\in(\Om^1_{A^\bu})^k$ satisfying $\d\Phi=0$, $\dd\Phi+\d\phi=0$ and $\dd\phi=k\om^0$. As in \eq{ln2eq11} we write $\Phi=\Phi_++\sum_{i=-1}^d\sum_{j=1}^{m_i}\Phi_j^{i+1}y_j^{k-i}$, where $\Phi_+\in A_+^{k+1}$ and $\Phi_j^{i+1}\in A_+^{i+1}$ for all $i,j$ do not involve the $y^i_j$, and we define $\phi_+\in(\Om^1_{A^\bu})^k$ by~\eq{ln2eq15}.

Write $e=[k/2]$, so that if $k$ is even then $e=d$ and $k=2e=2d$, and if $k$ is odd then $e=d-1$ and $k=2e+1=2d-1$. Choose nonnegative integers $n_0,n_{-1},\ldots,n_e$. Choose a smooth $\K$-algebra $B^0$ of dimension $m_0+n_0$, and a smooth morphism $\al^0:A^0\ra B^0$. Localizing $B^0$ if necessary, we may assume there exist $u^0_1,\ldots,u^0_{n_0}\in B^0$ such that $\dd\ti x^0_1,\ldots,\dd\ti x^0_{m_0},\dd u^0_1,\ldots,\dd u^0_{n_0}$ form a basis of $\Om^1_{B^0}$ over $B^0$, where we write $\ti x^0_j=\al^0(x^0_j)\in B^0$.

Define $B^*$ as a commutative graded $\K$-algebra to be the free graded algebra over $B^0$
generated by variables
\e
\begin{aligned}
&\ti x_1^i,\ldots,\ti x^i_{m_i} &&\text{in degree $i$ for
$i=-1,-2,\ldots,d$, and} \\
&u_1^i,\ldots,u^i_{n_i} &&\text{in degree $i$ for
$i=-1,-2,\ldots,e$, and} \\
&v_1^{k-1-i},\ldots,v^{k-1-i}_{n_i} &&\text{in degree $k-1-i$ for
$i=0,-1,\ldots,e$.}
\end{aligned}
\label{ln3eq2}
\e
So the upper index $i$ in $\ti x^i_j,u^i_j,v^i_j$ always indicates the degree. 

Define a morphism $\al_+:A_+^*\ra B^*$ of commutative graded $\K$-algebras by $\al_+^0=\al^0$ in degree 0 and
\e
\al_+(x^i_j)=\ti x^i_j,\quad i=-1,-2,\ldots,d,\;\> j=1,\ldots,m_i. 
\label{ln3eq3}
\e
This is well-defined as $A_+^*$ is freely generated over $A^0$ by the~$x^i_j$.

Now choose a superpotential $\Psi$ in $B^k$, which we require to satisfy 
\e
\sum_{i=-1}^e\sum_{j=1}^{n_i}\frac{\pd\Psi}{\pd u^i_j}\,
\frac{\pd\Psi}{\pd v^{k-1-i}_j}+\al_+(\Phi_+)+\sum_{i=-1}^d\sum_{j=1}^{m_i}(-1)^{i+1}\al_+(\Phi_j^{i+1})\frac{\pd\Psi}{\pd\ti x^i_j}=0
\label{ln3eq4}
\e
in $B^{k+1}$. Extend $\al_+$ to a morphism $\al:A^*\ra B^*$ by $\al\vert_{A^*_+}=\al_+$ and 
\e
\al(y^{k-i}_j)=(-1)^{i+1}\frac{\pd\Psi}{\pd\ti x^i_j},\quad i=0,-1,\ldots,d,\;\> j=1,\ldots,m_i.
\label{ln3eq5}
\e
This is well-defined as $A^*$ is freely generated over $A^*_+$ by the $y^{k-i}_j$. Then from \eq{ln2eq11} and \eq{ln3eq5} we see that \eq{ln3eq4} may be rewritten
\begin{equation*}
\sum_{i=-1}^e\sum_{j=1}^{n_i}\frac{\pd\Psi}{\pd u^i_j}\,
\frac{\pd\Psi}{\pd v^{k-1-i}_j}+\al(\Phi)=0.
\end{equation*}

Define the differential $\d$ in the cdga $B^\bu=(B^*,\d)$ by $\d=0$ on $B^0$, and
\e
\begin{gathered}
\d\ti x^i_j =(-1)^{i+1}\al_+(\Phi_j^{i+1}),\quad
i=-1,-2,\ldots,d,\;\> j=1,\ldots,m_i, \\
\d u^i_j =(-1)^{(i+1)k}\frac{\pd\Psi}{\pd v^{k-1-i}_j}, \;\> \d
v^{k-1-i}_j=\frac{\pd\Psi}{\pd u^i_j},\;\>\begin{subarray}{l}\ts
i=0,-1,\ldots,e,\\[6pt] \ts j=1,\ldots,n_i.\end{subarray}
\end{gathered}
\label{ln3eq6}
\e
To prove that $\d\ci\d=0$, note that for $i'=-1,\ldots,d$ and $j'=1,\ldots,m_{i'}$ we have
\e
\begin{split}
\d\ci\d\ti x^{i'}_{j'}&=(-1)^{i'+1}\sum_{i=-1}^d\sum_{j=1}^{m_i}\d \ti x_j^i\cdot\frac{\pd}{\pd\ti x_j^i}\bigl[\al_+(\Phi_{j'}^{i'+1})\bigr]\\
&=(-1)^{i'+1}\al_+\Biggl[\,\sum_{i=-1}^d\sum_{j=1}^{m_i}(-1)^{i+1}\,\Phi_j^{i+1}\cdot\frac{\pd\Phi_{j'}^{i'+1}}{\pd x_j^i}\Biggr]=0,\\
\end{split}
\label{ln3eq7}
\e
where in the first step we use that $\pd/\pd u_{j'}^{i'},\pd/\pd v_{j'}^{k-1-i'}$ are zero on $\al_+(\Phi_j^{i+1})$ as this is a function of the $\ti x^i_j$ only, in the second \eq{ln3eq3} and the first line of \eq{ln3eq6}, and in the third the second line of~\eq{ln2eq12}. 

For $i'=-1,\ldots,e$ and $j'=1,\ldots,n_{i'}$ we have
\e
\begin{split}
&\d\ci\d u^{i'}_{j'}=(-1)^{(i'+1)k}\Biggl[\,\sum_{i=-1}^d\sum_{j=1}^{m_i}\d \ti x_j^i\cdot\frac{\pd^2\Psi}{\pd\ti x_j^i\pd v^{k-1-i'}_{j'}}\\
&\;\>+\sum_{i=-1}^e\sum_{j=1}^{n_i}\d  u_j^i\cdot\frac{\pd^2\Psi}{\pd u_j^i\pd v^{k-1-i'}_{j'}}+\sum_{i=-1}^e\sum_{j=1}^{n_i}\d  v_j^{k-1-i}\cdot\frac{\pd^2\Psi}{\pd v_j^{k-1-i}\pd v^{k-1-i'}_{j'}}\Biggr]\\
&=(-1)^{(i'+1)k}\Biggl[\,\sum_{i=-1}^d\sum_{j=1}^{m_i}
(-1)^{i+1}\al_+(\Phi_j^{i+1})\cdot\frac{\pd^2\Psi}{\pd\ti x_j^i\pd v^{k-1-i'}_{j'}}\\
&\qquad+\sum_{i=-1}^e\sum_{j=1}^{n_i}(-1)^{(i+1)k}\frac{\pd\Psi}{\pd v^{k-1-i}_j}\cdot\frac{\pd^2\Psi}{\pd u_j^i\pd v^{k-1-i'}_{j'}}\\
&\qquad+\sum_{i=-1}^e\sum_{j=1}^{n_i}\frac{\pd\Psi}{\pd u^i_j}\cdot\frac{\pd^2\Psi}{\pd v_j^{k-1-i}\pd v^{k-1-i'}_{j'}}\Biggr]=0,
\end{split}
\label{ln3eq8}
\e
where in the first and second steps we use \eq{ln3eq6}, and in the third we apply $\pd/\pd v^{k-1-i'}_{j'}$ to \eq{ln3eq4}, noting that $\pd/\pd v^{k-1-i'}_{j'}$ is zero on $\al_+(\Phi_+),\al_+(\Phi_j^{i+1})$ as these are functions of the $\ti x^i_j$ only, and dealing with signs appropriately. Similarly, for $i'=0,\ldots,e$ and $j'=1,\ldots,n_{i'}$ we have
\ea
&\d\ci\d v^{k-1-i'}_{j'}=\sum_{i=-1}^d\sum_{j=1}^{m_i}\d \ti x_j^i\cdot\frac{\pd^2\Psi}{\pd\ti x_j^i\pd u^{i'}_{j'}}
\nonumber\\
\begin{split}
&\;\>+\sum_{i=-1}^e\sum_{j=1}^{n_i}\d  u_j^i\cdot\frac{\pd^2\Psi}{\pd u_j^i\pd u^{i'}_{j'}}+\sum_{i=-1}^e\sum_{j=1}^{n_i}\d  v_j^{k-1-i}\cdot\frac{\pd^2\Psi}{\pd v_j^{k-1-i}\pd u^{i'}_{j'}}\\
&=\sum_{i=-1}^d\sum_{j=1}^{m_i}
(-1)^{i+1}\al_+(\Phi_j^{i+1})\cdot\frac{\pd^2\Psi}{\pd\ti x_j^i\pd u^{i'}_{j'}}
\end{split}
\label{ln3eq9}\\
&\;\>+\sum_{i=-1}^e\sum_{j=1}^{n_i}(-1)^{(i+1)k}\frac{\pd\Psi}{\pd v^{k-1-i}_j}\cdot\frac{\pd^2\Psi}{\pd u_j^i\pd u^{i'}_{j'}}\!+\!\sum_{i=-1}^e\sum_{j=1}^{n_i}\frac{\pd\Psi}{\pd u^i_j}\cdot\frac{\pd^2\Psi}{\pd v_j^{k-1-i}\pd u^{i'}_{j'}}\!=\!0,
\nonumber
\ea
where in the last step we apply $\pd/\pd u^{i'}_{j'}$ to \eq{ln3eq4}.

This proves that $\d\ci\d=0$, so $B^\bu$ is a standard form cdga over $\K$. Also
\e
\d\ci\al(x^{i'}_{j'})=\d\ti x^{i'}_{j'}=(-1)^{i'+1}\al_+(\Phi_{j'}^{i'+1})=\al[(-1)^{i'+1}\Phi_{j'}^{i'+1}]=\al\ci\d x^{i'}_{j'},
\label{ln3eq10}
\e
using equations \eq{ln2eq14}, \eq{ln3eq3} and \eq{ln3eq6}, and
\ea
\d&\ci\al(y^{k-i'}_{j'})=(-1)^{i'+1}\d\frac{\pd\Psi}{\pd\ti x^{i'}_{j'}}
=(-1)^{i'+1}\Biggl[\,\sum_{i=-1}^d\sum_{j=1}^{m_i}\d \ti x_j^i\cdot\frac{\pd^2\Psi}{\pd\ti x_j^i\pd\ti x^{i'}_{j'}}
\nonumber\\
&\;\>+\sum_{i=-1}^e\sum_{j=1}^{n_i}\d  u_j^i\cdot\frac{\pd^2\Psi}{\pd u_j^i\pd\ti x^{i'}_{j'}}+\sum_{i=-1}^e\sum_{j=1}^{n_i}\d  v_j^{k-1-i}\cdot\frac{\pd^2\Psi}{\pd v_j^{k-1-i}\pd\ti x^{i'}_{j'}}\Biggr]
\nonumber\\
&=(-1)^{i'+1}\Biggl[\,\sum_{i=-1}^d\sum_{j=1}^{m_i}(-1)^{i+1}\al_+(\Phi_j^{i+1})\cdot\frac{\pd^2\Psi}{\pd\ti x_j^i\pd\ti x^{i'}_{j'}}
\label{ln3eq11}\\
&\;\>+\sum_{i=-1}^e\sum_{j=1}^{n_i}(-1)^{(i+1)k}\frac{\pd\Psi}{\pd v^{k-1-i}_j}\cdot\frac{\pd^2\Psi}{\pd u_j^i\pd\ti x^{i'}_{j'}}+\sum_{i=-1}^e\sum_{j=1}^{n_i}\frac{\pd\Psi}{\pd u^i_j}\cdot\frac{\pd^2\Psi}{\pd v_j^{k-1-i}\pd\ti x^{i'}_{j'}}\Biggr]
\nonumber\\
&=\frac{\pd}{\pd\ti x^{i'}_{j'}}\al_+(\Phi_+)\!+\!\sum_{i=-1}^d\sum_{j=1}^{m_i}\frac{\pd}{\pd\ti x^{i'}_{j'}}\bigl[\al_+(\Phi_j^{i+1})\bigr]\cdot(-1)^{i+1}\frac{\pd\Psi}{\pd\ti x^i_j}\!=\!\al\!\ci\!\d (y^{k-i'}_{j'}),
\nonumber
\ea
where in the first step we use \eq{ln3eq5}, in the third \eq{ln3eq6}, in the fourth we apply $\pd/\pd\ti x^{i'}_{j'}$ to \eq{ln3eq4} and deal with signs, and in the fifth we use the second equation of \eq{ln2eq14} and \eq{ln3eq5}. Equations \eq{ln3eq10}--\eq{ln3eq11} imply that $\d\ci\al=\al\ci\d:A^*\ra B^*$, so $\al:A^\bu\ra B^\bu$ and hence $\al_+:A^\bu_+\ra B^\bu$ are morphisms in $\cdga_\K$. Note that $\al_+$ is a submersion, in the sense of Definition~\ref{ln3def1}.

Following \eq{ln2eq6}, define $h^0\in (\La^2\Om^1_{B^\bu})^{k-1}$ by 
\e
h^0=\sum_{i=0}^e\sum_{j=1}^{n_i}\dd u^i_j\,\dd v^{k-1-i}_j.
\label{ln3eq12}
\e
Then $\dd h^0=0$, and 
\ea
&\d h^0=\sum_{i=0}^e\sum_{j=1}^{n_i}\bigl[(\d\ci\dd u^i_j)\,\dd v^{k-1-i}_j\!+\!(-1)^{(i+1)k}(\d\ci\dd v^{k-1-i}_j)\,\dd u^i_j\bigr]
\nonumber\\
&=-\sum_{i=0}^e\sum_{j=1}^{n_i}\bigl[(\dd\ci\d u^i_j)\,\dd v^{k-1-i}_j+(-1)^{(i+1)k}(\dd\ci\d v^{k-1-i}_j)\,\dd u^i_j\bigr]
\nonumber\\
&=-\sum_{i=0}^e\sum_{j=1}^{n_i}(-1)^{(i+1)k}\biggl[
\dd\biggl(\frac{\pd\Psi}{\pd v^{k-1-i}_j}\biggr)\,\dd v^{k-1-i}_j
+\dd\biggl(\frac{\pd\Psi}{\pd u^i_j}\biggr)\,\dd u^i_j\biggr]
\nonumber\\
&=\dd\biggl[-\sum_{i=0}^e\sum_{j=1}^{n_i}\biggl[\dd v^{k-1-i}_j\,\frac{\pd\Psi}{\pd v^{k-1-i}_j}+\dd u^i_j\,\frac{\pd\Psi}{\pd u^i_j}\biggr]\biggr]
\label{ln3eq13}\\
&=\dd\biggl[-\dd\Psi+\sum_{i=0}^d\sum_{j=1}^{m_i}\dd\ti x^i_j\,\frac{\pd\Psi}{\pd\ti x^i_j}\biggr]
\nonumber\\
&=\sum_{i=0}^d\sum_{j=1}^{m_i}\dd\Bigl(\dd\bigl(\al(x^i_j)\bigr)\cdot(-1)^{i+1}\al(y_j^{k-i})\Bigr) 
\nonumber\\
&=\sum_{i=0}^d\sum_{j=1}^{m_i}\dd\bigl(\al(x^i_j)\bigr)\dd\bigl(\al(y_j^{k-i})\bigr)=\al_*\biggl[\sum_{i=0}^d\sum_{j=1}^{m_i}\dd x^i_j\,\dd y_j^{k-i}\biggr]=\al_*(\om^0),
\nonumber
\ea
using \eq{ln3eq12} in the first step, $\d\ci\dd+\dd\ci\d=0$ in the second, \eq{ln3eq6} in the third, $\dd\ci\dd=0$ in the fifth, \eq{ln3eq3} and \eq{ln3eq5} in the sixth, $\al_*\ci\dd=\dd\ci\al_*$ in the eighth, and \eq{ln2eq6} in the ninth.

Definition \ref{ln2def6} and equation \eq{ln3eq13} now imply that $h:=(h^0,0,0,\ldots)$ is an isotropic structure for $\bSpec\al:\bSpec B^\bu\ra\bSpec A^\bu$ and the $k$-shifted symplectic structure $\om=(\om^0,0,0,\ldots)$ on $\bSpec A^\bu$. We will now prove that this isotropic structure is nondegenerate, so that $\bSpec B^\bu$ is Lagrangian in $(\bSpec A^\bu,\om)$. To do this, we have to show that the morphism $\chi:\bT_{B^\bu/A^\bu}\ra\bL_{A^\bu}[k-1]$ of $B^\bu$-modules defined in \eq{ln2eq3} is a quasi-isomorphism. 

It is enough to apply $-\ot_{B^\bu}H^0(B^\bu)$, and show the corresponding morphism of complexes of $H^0(B^\bu)$-modules is an isomorphism. The analogue of \eq{ln2eq3} is
\ea
\xymatrix@C=60pt@R=15pt{ *+[r]{(\Om^1_{B^\bu})^\vee\ot_{B^\bu}H^0(B^\bu)} \drrtwocell_{}\omit^{^{h_{\bs L}^0\cdot\,\,\,\,\,\,\,\,\,\,\,\,\,\,\,\,\,}}\omit{} \ar[rr] \ar[d]^{(\Om^1_\al)^\vee} && *+[l]{0} \ar[d] \\
 *+[r]{(\Om^1_{A^\bu})^\vee\ot_{A^\bu}H^0(B^\bu)} \ar[r]^(0.61){\om^0\cdot} \ar[d] & \Om^1_{A^\bu}[k]\ot_{A^\bu}H^0(B^\bu) \ar[r]^(0.37){\Om^1_\al[k]} & *+[l]{\Om^1_{B^\bu}[k]\ot_{B^\bu}H^0(B^\bu)} \\
*+[r]{\bT_{B^\bu/A^\bu}[1]\ot_{B^\bu}H^0(B^\bu).} \ar@{.>}[urr]_{\chi[1]}
 }
\nonumber\\[-20pt]
\label{ln3eq14}
\ea
Here we have used $\Om^1_{A^\bu},\Om^1_{B^\bu},(\Om^1_{A^\bu})^\vee,(\Om^1_{B^\bu})^\vee$ as models for $\bL_{A^\bu},\bL_{B^\bu},\bT_{A^\bu},\bT_{B^\bu}$, since $A^\bu,B^\bu$ are standard form cdgas. As a model for $\bT_{B^\bu/A^\bu}\ot_{B^\bu}H^0(B^\bu)$ we will use the cone of $(\Om^1_\al)^\vee$ in \eq{ln3eq14}, so that
\begin{align*}
&\bigl((\bT_{B^\bu/A^\bu}\ot_{B^\bu}H^0(B^\bu))^*,\d\bigr)=\\
&\left(((\Om^1_{B^\bu})^\vee\!\ot_{B^\bu}\!H^0(B^\bu))^*\!\op\!((\Om^1_{A^\bu})^\vee\!\ot_{A^\bu}\!H^0(B^\bu)^{*-1},\begin{pmatrix} \d_{B^\bu}\!\! & 0 \\ (\Om^1_\al)^\vee\!\! & \d_{A^\bu} \end{pmatrix} \right).
\end{align*}

As $H^0(B^\bu)$-modules, the $i^{\rm th}$ graded pieces of $\Om^1_{A^\bu}\ot_{A^\bu}H^0(B^\bu)$, $\Om^1_{B^\bu}\ot_{B^\bu}H^0(B^\bu)$, $(\Om^1_{A^\bu})^\vee\ot_{A^\bu}H^0(B^\bu)$, and $(\Om^1_{B^\bu})^\vee\ot_{B^\bu}H^0(B^\bu)$ are
\ea
\begin{split}
\bigl(\Om^1_{A^\bu}\ot_{A^\bu}H^0(B^\bu)\bigr){}^i=\bigl\langle &\dd x^i_j,\; j=1,\ldots,m_i,\\
&\dd y^i_j,\; j=1,\ldots,m_{k-i}\bigr\rangle{}_{H^0(B^\bu)},
\end{split}
\label{ln3eq15}\\
\begin{split}
\bigl(\Om^1_{B^\bu}\ot_{B^\bu}H^0(B^\bu)\bigr){}^i=\bigl\langle &\dd\ti x^i_j,\; j=1,\ldots,m_i,\\
\dd u^i_j,\; j=1,\ldots,n_i,\; &\dd v^i_j,\; j=1,\ldots,n_{k-1-i}\bigr\rangle{}_{H^0(B^\bu)},
\end{split}
\label{ln3eq16}\\
\begin{split}
\bigl((\Om^1_{A^\bu})^\vee\ot_{A^\bu}H^0(B^\bu)\bigr){}^i=\bigl\langle &\ts\frac{\pd}{\pd x^{-i}_j},\; j=1,\ldots,m_{-i},\\
&\ts\frac{\pd}{\pd y^{-i}_j},\; j=1,\ldots,m_{i-k}\bigr\rangle{}_{H^0(B^\bu)},
\end{split}
\label{ln3eq17}\\
\begin{split}
\bigl((\Om^1_{B^\bu})^\vee\ot_{B^\bu}H^0(B^\bu)\bigr){}^i=\bigl\langle &\ts\frac{\pd}{\pd\ti x^{-i}_j},\; j=1,\ldots,m_{-i},\\
\ts\frac{\pd}{\pd u^{-i}_j},\; j=1,\ldots,n_{-i},\; &\ts\frac{\pd}{\pd v^{-i}_j},\; j=1,\ldots,n_{i+1-k}\bigr\rangle{}_{H^0(B^\bu)},
\end{split}
\label{ln3eq18}
\ea
where $\an{\cdots}_{H^0(B^\bu)}$ denotes the free $H^0(B^\bu)$-module with basis `$\cdots$'.

The next diagram shows $\chi:\bT_{B^\bu/A^\bu}\ot_{B^\bu}H^0(B^\bu)\ra \Om^1_{B^\bu}[k-1]\ot_{B^\bu}H^0(B^\bu)$ in degrees $i,i+1$, together with $\d$ in both complexes.
\ea
\nonumber\\[-77pt]
\begin{gathered}
\xymatrix@C=90pt@R=55pt{
*+[r]{\raisebox{-100pt}{$\ts\begin{subarray}{l} \ts (\bT_{B^\bu/A^\bu}\ot_{B^\bu}H^0(B^\bu))^i\\
\ts =\bigl\langle \frac{\pd}{\pd\ti x^{-i}_j},\;\forall j\bigr\rangle{}_{H^0(B^\bu)}\op
\\
\ts\bigl\langle\frac{\pd}{\pd u^{-i}_j},\frac{\pd}{\pd v^{-i}_j},\;\forall j\bigr\rangle{}_{H^0(B^\bu)}\\
\ts \op\bigl\langle \frac{\pd}{\pd x^{1-i}_j},\;\forall j\bigr\rangle{}_{H^0(B^\bu)}\\
\ts \op\bigl\langle\frac{\pd}{\pd y^{1-i}_j},\;\forall j\bigr\rangle{}_{H^0(B^\bu)}
\end{subarray}$}}
\ar@<-30pt>[r]_(0.4){\chi^i=\begin{pmatrix} \st 0 & \st 0 & \st * & \st \om^0\cdot \\
\st 0 & \st h^0\cdot & \st * & \st 0  \end{pmatrix}} 
\ar[d]^{\d=\begin{pmatrix} \st * & \st 0 & \st 0 & \st 0 \\
\st * & \st * & \st 0 & \st 0 \\
\st \al^* & \st 0 & \st * & \st 0 \\
\st * & \st  * & \st * & \st * \end{pmatrix}}
 & *+[l]{\raisebox{-40pt}{$\ts\begin{subarray}{l} \ts (\Om^1_{B^\bu}[k\!-\!1]\ot_{B^\bu}H^0(B^\bu))^i\\
\ts =\bigl\langle \dd\ti x^{k-1+i}_j,\;\forall j\bigr\rangle{}_{H^0(B^\bu)}\op
\\
\ts\bigl\langle\dd u^{k-1+i}_j,\dd v^{k-1+i}_j,\;\forall j\bigr\rangle{}_{H^0(B^\bu)}
\end{subarray}$}} 
\ar@<-40pt>[d]_{\d=\begin{pmatrix} \st * & \st 0  \\
\st * & \st *  \end{pmatrix}} \\
*+[r]{\raisebox{90pt}{$\ts\begin{subarray}{l} \ts (\bT_{B^\bu/A^\bu}\ot_{B^\bu}H^0(B^\bu))^{i+1}\\
\ts =\bigl\langle \frac{\pd}{\pd\ti x^{-i-1}_j},\;\forall j\bigr\rangle{}_{H^0(B^\bu)}\op
\\
\ts\bigl\langle\frac{\pd}{\pd u^{-i-1}_j},\frac{\pd}{\pd v^{-i-1}_j},\;\forall j\bigr\rangle{}_{H^0(B^\bu)}\\
\ts \op\bigl\langle \frac{\pd}{\pd x^{-i}_j},\;\forall j\bigr\rangle{}_{H^0(B^\bu)}\\
\ts \op\bigl\langle\frac{\pd}{\pd y^{-i}_j},\;\forall j\bigr\rangle{}_{H^0(B^\bu)}
\end{subarray}$}} 
\ar@<25pt>[r]^(0.45){\chi^{i+1}=\begin{pmatrix} \st 0 & \st 0 & \st * & \st \om^0\cdot \\
\st 0 & \st h^0\cdot & \st * & \st 0  \end{pmatrix}}  & *+[l]{\raisebox{30pt}{$\ts\begin{subarray}{l} \ts (\Om^1_{B^\bu}[k\!-\!1]\ot_{B^\bu}H^0(B^\bu))^{i+1}\\
\ts =\bigl\langle \dd\ti x^{k+i}_j,\;\forall j\bigr\rangle{}_{H^0(B^\bu)}\op
\\
\ts\bigl\langle\dd u^{k+i}_j,\dd v^{k+i}_j,\;\forall j\bigr\rangle{}_{H^0(B^\bu)}.
\end{subarray}$}} }\!\!\!\!\!\!\!\!\!\!\!\!\!\!\!\!\!\!\!\!\!\!\!{}
\end{gathered}
\label{ln3eq19}\\[-70pt]
\nonumber
\ea
We have divided $\bT_{B^\bu/A^\bu}\ot_{B^\bu}H^0(B^\bu)$ into the direct sum of four pieces, and $\Om^1_{B^\bu}[k-1]\ot_{B^\bu}H^0(B^\bu)$ into two. The morphisms $\d,\chi^i,\chi^{i+1}$ are written in matrix form, where `$*$' denotes some morphism. In the left hand $\d$, the `$\al^*$' maps $\frac{\pd}{\pd\ti x^{-i}_j}\mapsto \frac{\pd}{\pd x^{-i}_j}$, up to sign. In $\chi^i$, $\om^0\cdot$ maps $\frac{\pd}{\pd y^{1-i}_j}\mapsto \dd\ti x^{k-1+i}_j$, and $h^0\cdot$ maps $\frac{\pd}{\pd u^{-i}_j}\mapsto \dd v^{k-1+i}_j$ and $\frac{\pd}{\pd v^{-i}_j}\mapsto \dd u^{k-1+i}_j$, all up to sign, and similarly for $\chi^{i+1}$. The important thing is that these $\al^*,\om^0\cdot,h^0\cdot$ in $\d,\chi^i,\chi^{i+1}$ are all isomorphisms of $H^0(B^\bu)$-modules.

Now consider the graded $H^0(B^\bu)$-submodule 
\begin{align*}
C^*:&=\ts \{0\}\op \bigl\langle\frac{\pd}{\pd u^{-*}_j},\frac{\pd}{\pd v^{-*}_j},\;\forall j\bigr\rangle{}_{H^0(B^\bu)}\op \{0\} \op \bigl\langle\frac{\pd}{\pd y^{1-*}_j},\;\forall j\bigr\rangle{}_{H^0(B^\bu)}\\
&\subseteq (\bT_{B^\bu/A^\bu}\ot_{B^\bu}H^0(B^\bu))^*.
\end{align*}
The form of the left hand `$\d$' in \eq{ln3eq19} implies that $C^*$ is closed under $\d$, so $C^\bu=(C^*,\d)$ is a subcomplex of $\bT_{B^\bu/A^\bu}\ot_{B^\bu}H^0(B^\bu)$. The isomorphism `$\al^*$' plus two other zeroes in the left hand `$\d$' in \eq{ln3eq19} imply that the inclusion ${\rm inc}:C^\bu\hookra \bT_{B^\bu/A^\bu}\ot_{B^\bu}H^0(B^\bu)$ is a quasi-isomorphism. And the isomorphisms `$h^0\cdot$', `$\om^0\cdot$' plus two other zeroes in $\chi^i,\chi^{i+1}$ in \eq{ln3eq19} imply that $\chi\vert_{C^\bu}$ is a strict isomorphism of complexes. Thus we have a commutative diagram
\e
\begin{gathered}
\xymatrix@C=130pt@R=15pt{
*+[r]{ C^\bu} \ar[d]^{\rm inc}_\simeq \ar[dr]^(0.3){\chi\vert_{C^\bu}}_(0.3)\cong \\
*+[r]{ \bT_{B^\bu/A^\bu}\ot_{B^\bu}H^0(B^\bu)} \ar[r]^(0.45)\chi & *+[l]{\Om^1_{B^\bu}[k-1]\ot_{B^\bu}H^0(B^\bu),} }
\end{gathered}
\label{ln3eq20}
\e
so $\chi$ is a quasi-isomorphism. Therefore the isotropic structure $h$ is nondegenerate, and $\bSpec B^\bu$ is Lagrangian in $(\bSpec A^\bu,\om)$. We say that $A^\bu,\om,B^\bu,\al,h$ are in {\it Lagrangian Darboux form}.

Following \eq{ln2eq9}, define $\psi\in(\Om^1_{B^\bu})^{k-1}$ by 
\e
\psi=\sum_{i=0}^e\sum_{j=1}^{n_i}\bigl[i\,u^i_j\,\dd v^{k-1-i}_j +(-1)^{(i+1)k}(k-1-i)v^{k-1-i}_j\,\dd u^i_j \bigr].
\label{ln3eq21}
\e
As a relative version of \eq{ln2eq10}, we will prove that
\ea
\d\Psi&=-\al(\Phi+\Phi_+)&&\text{in $B^{k+1}$,}
\label{ln3eq22}\\
\dd\Psi+\d\psi&=-\al_*(\phi+\phi_+)&&\text{in $(\Om^1_{B^\bu})^k$, and}
\label{ln3eq23}\\
\dd\psi&=(k-1)h^0 &&\text{in $(\La^2\Om^1_{B^\bu})^{k-1}$.}
\label{ln3eq24}
\ea
Note that $\d h^0=\al_*(\om^0)$ in \eq{ln3eq13} also follows from
\begin{align*}
(k-1)\d h^0&=\d\ci\dd\psi=-\dd\ci\d\psi=-\dd\big[\dd\Psi+\d\psi\bigr]=\dd\ci\al_*(\phi+\phi_+)\\
&=\al_*(\dd\phi+\dd\phi_+)=\al_*(k\om^0-\om^0)=(k-1)\al_*(\om^0),
\end{align*}
using equations \eq{ln2eq10}, \eq{ln2eq16}, \eq{ln3eq23}, and \eq{ln3eq24}.

For equation \eq{ln3eq22}, we have
\ea
&\d\Psi=\sum_{i=-1}^d\sum_{j=1}^{m_i}\d \ti x_j^i\,\frac{\pd\Psi}{\pd\ti x_j^i}
+\sum_{i=-1}^e\sum_{j=1}^{n_i}\d u_j^i\,\frac{\pd\Psi}{\pd u_j^i}+\sum_{i=-1}^e\sum_{j=1}^{n_i}\d v_j^{k-1-i}\,\frac{\pd\Psi}{\pd v_j^{k-1-i}}
\nonumber\\
&=\sum_{i=-1}^d\sum_{j=1}^{m_i}(-1)^{i+1}\al_+(\Phi_j^{i+1})\frac{\pd\Psi}{\pd\ti x_j^i}
+\sum_{i=-1}^e\sum_{j=1}^{n_i}(-1)^{(i+1)k}\frac{\pd\Psi}{\pd v^{k-1-i}_j}\frac{\pd\Psi}{\pd u_j^i}
\nonumber\\
&\qquad+\sum_{i=-1}^e\sum_{j=1}^{n_i}\frac{\pd\Psi}{\pd u^i_j}\frac{\pd\Psi}{\pd v_j^{k-1-i}}
\label{ln3eq25}\\
&=\sum_{i=-1}^d\sum_{j=1}^{m_i}(-1)^{i+1}\al_+(\Phi_j^{i+1})\frac{\pd\Psi}{\pd\ti x_j^i}
+2\sum_{i=-1}^e\sum_{j=1}^{n_i}\frac{\pd\Psi}{\pd u^i_j}\frac{\pd\Psi}{\pd v_j^{k-1-i}}
\nonumber\\
&=\sum_{i=-1}^d\sum_{j=1}^{m_i}(-1)^{i+1}\al_+(\Phi_j^{i+1})\frac{\pd\Psi}{\pd\ti x_j^i}\!-\!2\biggl[\al_+(\Phi_+)\!+\!\sum_{i=-1}^d\sum_{j=1}^{m_i}(-1)^{i+1}\al_+(\Phi_j^{i+1})\frac{\pd\Psi}{\pd\ti x^i_j}\biggr]
\nonumber\\
&=-\al(\Phi+\Phi_+),
\nonumber
\ea
using \eq{ln3eq6} in the second step, \eq{ln3eq4} in the fourth, and \eq{ln2eq11} and \eq{ln3eq5} in the fifth. For equation \eq{ln3eq23}, we have\begin{small}
\ea
&\dd\Psi+\d\psi=\sum_{i=0}^d\sum_{j=1}^{m_i}\dd \ti x_j^i\,\frac{\pd\Psi}{\pd\ti x_j^i}
+\sum_{i=0}^e\sum_{j=1}^{n_i}\biggl[\dd u_j^i\,\frac{\pd\Psi}{\pd u_j^i}+\dd v_j^{k-1-i}\,\frac{\pd\Psi}{\pd v_j^{k-1-i}}\biggr]
\nonumber\\
&+\sum_{i=0}^e\sum_{j=1}^{n_i}\bigl[i\,\d u^i_j\,\dd v^{k-1-i}_j +(-1)^{(i+1)k}(k-1-i)\d v^{k-1-i}_j\,\dd u^i_j \bigr]
\nonumber\\
&+\sum_{i'=0}^e\sum_{j'=1}^{n_{i'}}\bigl[(-1)^{i'}i'u^{i'}_{j'}\,\d\ci\dd v^{k-1-i'}_{j'} -(-1)^{i'(k+1)}(k-1-i')v^{k-1-i'}_{j'}\,\d\ci\dd u^{i'}_{j'} \bigr]
\allowdisplaybreaks
\nonumber\\
&=\sum_{i=0}^d\sum_{j=1}^{m_i}(-1)^{(i+1)k}\frac{\pd\Psi}{\pd\ti x_j^i}\,\dd \ti x_j^i+
\nonumber\\
&+\sum_{i=0}^e\sum_{j=1}^{n_i}(-1)^{(i+1)k}\biggl[(i+1)\frac{\pd\Psi}{\pd v^{k-1-i}_j}\dd v^{k-1-i}_j+(k-i)\frac{\pd\Psi}{\pd u^i_j}\,\dd u^i_j \biggr]
\nonumber\\
&-\sum_{i'=0}^e\sum_{j'=1}^{n_{i'}}\biggl[(-1)^{i'}i'u^{i'}_{j'}\,\dd\biggl[\frac{\pd\Psi}{\pd u^{i'}_{j'}}\biggr] -(-1)^{i'+k}(k-1-i')v^{k-1-i'}_{j'}\,\dd\biggl[\frac{\pd\Psi}{\pd v^{k-1-i'}_{j'}}\biggr] \biggr]
\allowdisplaybreaks
\nonumber\\
&=\sum_{i=0}^d\sum_{j=1}^{m_i}\biggl[(-1)^{(i+1)k}\frac{\pd\Psi}{\pd\ti x_j^i}
-\sum_{i'=0}^e\sum_{j'=1}^{n_{i'}}(-1)^{(i+1)k}i'u^{i'}_{j'}\frac{\pd^2\Psi}{\pd u^{i'}_{j'}\pd\ti x_j^i}
\nonumber\\
&-\sum_{i'=0}^e\sum_{j'=1}^{n_{i'}}(-1)^{(i+1)k}(k-1-i')v^{k-1-i'}_{j'}\frac{\pd^2\Psi}{\pd v^{k-1-i'}_{j'}\pd\ti x_j^i}\biggr]\dd \ti x_j^i
\nonumber\\
&+\sum_{i=0}^e\sum_{j=1}^{n_i}\biggl[(-1)^{(i+1)k}(k-i)\frac{\pd\Psi}{\pd u^i_j}-\sum_{i'=0}^e\sum_{j'=1}^{n_{i'}}(-1)^{(i+1)k}i'u^{i'}_{j'}\frac{\pd^2\Psi}{\pd u^{i'}_{j'}\pd u_j^i}
\nonumber\\
&-\sum_{i'=0}^e\sum_{j'=1}^{n_{i'}}(-1)^{(i+1)k}(k-1-i')v^{k-1-i'}_{j'}\frac{\pd^2\Psi}{\pd v^{k-1-i'}_{j'}\pd u_j^i}  \biggr]\dd u_j^i
\nonumber\\
&+\sum_{i=0}^e\sum_{j=1}^{n_i}\biggl[(-1)^{(i+1)k}(i+1)\frac{\pd\Psi}{\pd v^{k-1-i}_j} -\sum_{i'=0}^e\sum_{j'=1}^{n_{i'}}(-1)^{(i+1)k}i'u^{i'}_{j'}\frac{\pd^2\Psi}{\pd u^{i'}_{j'}\pd v_j^{k-1-i}}
\nonumber\\
&-\sum_{i'=0}^e\sum_{j'=1}^{n_{i'}}(-1)^{(i+1)k}(k-1-i')v^{k-1-i'}_{j'}\frac{\pd^2\Psi}{\pd v^{k-1-i'}_{j'}\pd v^{k-1-i}_j}\biggr]\dd v_j^{k-1-i}
\allowdisplaybreaks
\nonumber\\
&=-\sum_{i=0}^d\sum_{j=1}^{m_i}\biggl[(-1)^{(i+1)k}(k-1-i)\frac{\pd\Psi}{\pd\ti x_j^i}
-\sum_{i'=0}^d\sum_{j'=1}^{m_{i'}}(-1)^{(i+1)k}i'\ti x^{i'}_{j'}\frac{\pd^2\Psi}{\pd\ti x^{i'}_{j'}\pd\ti x_j^i}\biggr]\dd \ti x_j^i
\nonumber\\
&+\sum_{i=0}^e\sum_{j=1}^{n_i}\sum_{i'=0}^d\sum_{j'=1}^{m_{i'}}(-1)^{(i+1)k}i'\ti x^{i'}_{j'}\biggl[\frac{\pd^2\Psi}{\pd\ti x^{i'}_{j'}\pd u_j^i}\dd u_j^i+
\frac{\pd^2\Psi}{\pd\ti x^{i'}_{j'}\pd v_j^{k-1-i}}\dd v_j^{k-1-i}\biggr]
\allowdisplaybreaks
\nonumber\\
&=-\sum_{i=0}^d\sum_{j=1}^{m_i}(-1)^{(i+1)(k+1)}(k-1-i)\al(y^{k-1-i}_j)\dd[\al(x^i_j)]
\nonumber\\
&-\sum_{i'=0}^d\sum_{j'=1}^{m_{i'}}i'\al(x^{i'}_{j'})\dd[\al(y^{k-i}_j)]
\allowdisplaybreaks
\nonumber\\
&=-\al_*\biggl[\sum_{i=0}^d\sum_{j=1}^{m_i}\bigl[(-1)^{(i+1)(k+1)}(k-1-i)y^{k-i}_j\,\dd x^i_j+i\,x^i_j\,\dd y^{k-i}_j\bigr]\biggr]
\nonumber\\
&=-\al_*(\phi+\phi_+),
\label{ln3eq26}
\ea\end{small}using \eq{ln3eq21} in the first step, \eq{ln3eq6} and $\d\ci\dd+\dd\ci\d=0$ in the second, in the fourth that
\begin{small}\begin{align*}
\biggl[\sum_{i'=0}^d\sum_{j'=1}^{m_{i'}}i'x^{i'}_{j'}\frac{\pd}{\pd\ti x^{i'}_{j'}}\!+\! \sum_{i'=0}^e\sum_{j'=1}^{n_{i'}}i'u^{i'}_{j'}\frac{\pd}{\pd u^{i'}_{j'}}\!+\!(k\!-\!1\!-\!i')v^{k-1-i'}_{j'}\frac{\pd}{\pd v^{k-1-i'}_{j'}}\biggr]\frac{\pd\Psi}{\pd\ti x^i_j}
\!=\!(k\!-\!i)\frac{\pd\Psi}{\pd\ti x^i_j},
\end{align*}\end{small}since $\frac{\pd\Psi}{\pd\ti x^i_j}$ has degree $k-i$, and similar equations for $\frac{\pd\Psi}{\pd u^i_j},\frac{\pd\Psi}{\pd v^{k-1-i}_j}$, equations \eq{ln3eq3} and \eq{ln3eq5} in the fifth, and \eq{ln2eq9} and \eq{ln2eq15} in the seventh. Equation \eq{ln3eq24} is immediate from \eq{ln3eq12} and~\eq{ln3eq21}. 

Let us summarize our progress:
\begin{itemize}
\setlength{\itemsep}{0pt}
\setlength{\parsep}{0pt}
\item Example \ref{ln2ex2} defined a cdga $A^\bu$ and $\om^0=\sum_{i=0}^d\sum_{j=1}^{m_i}\dd x^i_j\,\dd y^{k-i}_j$ in $(\La^2\Om^1_{A^\bu})^k$ such that $\om=(\om^0,0,\ldots)$ is $k$-shifted symplectic on $\bSpec A^\bu$, and $\Phi\in A^{k+1}$, $\phi\in(\Om^1_{A^\bu})^k$ with $\d\Phi=0$, $\dd\Phi+\d\phi=0$ and $\dd\phi=k\om^0$.
\item Remark \ref{ln2rem3} defined a sub-cdga $A^\bu_+\subseteq A^\bu$ and $\Phi_+\in A^{k+1}_+$, $\phi_+\in(\Om^1_{A^\bu})^k$ satisfying $\d\Phi_+=0$, $\dd\Phi_++\d\phi_+=0$ and $\dd\phi_+=-\om^0$.
\item We define a cdga $B^\bu$ and a morphism $\al:A^\bu\ra B^\bu$ such that $\al_+:=\al\vert_{A_+^\bu}:A^\bu_+\ra B^\bu$ is a submersion, a Lagrangian isotropic structure $h=(h^0,0,0,\ldots)$ with $h^0=\sum_{i=0}^e\sum_{j=1}^{n_i}\dd u^i_j\,\dd v^{k-1-i}_j$  for $\al$, and $\Psi\in B^k$, $\psi\in(\Om^1_{B^\bu})^{k-1}$ with $\d\Psi=-\al(\Phi+\Phi_+)$, $\dd\Psi+\d\psi=-\al_*(\phi+\phi_+)$ and $\dd\psi=(k-1)h^0$. We say $A^\bu,\om,B^\bu,\al,h$ are in {\it Lagrangian Darboux form}.
\end{itemize}
This finally concludes Example~\ref{ln3ex1}.
\end{ex}

\begin{rem} 
\label{ln3rem1}
Bouaziz and Grojnowski \cite{BoGr} proved their own $k$-shifted symplectic Darboux Theorem independently of \cite{BBJ}, showing that any $k$-shifted symplectic derived $\K$-scheme $(\bX,\om_\bX)$ for $k<0$ with $k\not\equiv 2\mod 4$ is \'etale locally equivalent to a {\it twisted\/ $k$-shifted cotangent bundle\/} $T^*_t[k]\bY$, where $\bY$ is an affine derived $\K$-scheme, and $t\in\O_\bY^{k+1}$ with $\d t=0$ is used to `twist' the $k$-shifted cotangent bundle $T^*[k]\bY$. Remark \ref{ln2rem3} related their picture to Theorem~\ref{ln2thm5}.

We can explain our $k$-shifted Lagrangian Neighbourhood Theorem \ref{ln3thm2}(i) below in the style of Bouaziz and Grojnowski \cite{BoGr}, by saying that a $k$-shifted Lagrangian $\bs f:\bs L\ra\bX$ for $k\not\equiv 3\mod 4$ is \'etale locally equivalent to the {\it twisted\/ $(k-1)$-shifted relative cotangent bundle\/} $T^*_{u/t}[k-1](\bZ/\bY)\ra T^*_t[k]\bY$ of a morphism of affine derived $\K$-schemes $\bs g:\bZ\ra\bY$, with `twisting' $u\in\O_\bZ^k$ satisfying $\d u+\bs g^*(t)=0$, or equivalently, to the {\it twisted\/ $k$-shifted conormal bundle\/} $N^*_{u/t}[k](\bZ/\bY)\ra T^*_t[k]\bY$, since~$N^*(\bZ/\bY)=T^*[-1](\bZ/\bY)$.

Explicitly, in the situation of Example \ref{ln3ex1} and following Remark \ref{ln2rem3}, define $B^\bu_+$ to be the sub-cdga of $B^\bu$ generated by $B^0$ and the variables $\ti x^i_j,u^i_j$ for all $i<0$ and $j$. Then $B^*$ is freely generated over $B^*_+$ by the variables $v^{k-1-i}_j$ for all $i,j$. Also $\al_+$ maps $A_+^\bu\ra B_+^\bu\subseteq B^\bu$. For degree reasons $\Psi$ can be at most linear in the $v_j^{k-1-i}$, so as in \eq{ln2eq11} we may write
\e
\Psi=\Psi_++\sum_{i=-1}^e\sum_{j=1}^{n_i}\Psi_j^{i+1}v_j^{k-1-i},
\label{ln3eq27}
\e
where $\Psi_+\in B_+^k$ and $\Psi_j^{i+1}\in B_+^{i+1}$ for all $i,j$ do not involve the variables $v^{k-1-i}_j$. 

Then as in \eq{ln2eq12}--\eq{ln2eq13}, equation \eq{ln3eq4} is equivalent to the equations
\ea
\!\!\sum_{i=-1}^d\!\sum_{j=1}^{m_i}(-\!1)^{i+1}\Psi_j^{i+1}\frac{\pd\Psi_+}{\pd u^i_j}\!+\!\al_+(\Phi_+)\!+\!\!\sum_{i=-1}^d\!\sum_{j=1}^{m_i}(-\!1)^{i+1}\al_+(\Phi_j^{i+1})\frac{\pd\Psi_+}{\pd\ti x^i_j}\!&=0,\!\!
\label{ln3eq28}\\
\!\!\sum_{i=-1}^{i'+1}\sum_{j=1}^{m_i}(-1)^{i+1}\Psi_j^{i+1}\,\frac{\pd\Psi_{j'}^{i'+1}}{\pd u^i_j}+\sum_{i=-1}^d\sum_{j=1}^{m_i}(-1)^{i+1}\al_+(\Phi_j^{i+1})\frac{\pd\Psi_{j'}^{i'+1}}{\pd\ti x^i_j}&=0,\!\!
\label{ln3eq29}
\ea
where \eq{ln3eq28} holds in $B_+^{k+1}$, and \eq{ln3eq29} holds in $B_+^{i'+2}$ for all $i'=-1,\ldots,e$ and $j'=1,\ldots,n_{i'}$. Also equations \eq{ln3eq5} and \eq{ln3eq6} may be rewritten
\ea
\al(y^{k-i}_j)&=(-1)^{i+1}\biggl[\frac{\pd\Psi_+}{\pd\ti x^i_j}
+\sum_{i'=i-1}^e\sum_{j'=1}^{n_{i'}}\frac{\pd\Psi_{j'}^{i'+1}}{\pd\ti x^i_j}\,v_{j'}^{k-1-i'}\biggr],
\label{ln3eq30}\\
\begin{split}
\d\ti x^i_j&=(-1)^{i+1}\al_+(\Phi_j^{i+1}), \qquad
\d u^i_j =(-1)^{i+1}\Psi^{i+1}_j, \\ 
\d v^{k-1-i}_j&=\frac{\pd\Psi_+}{\pd u^i_j}+\sum_{i'=i-1}^e\sum_{j'=1}^{n_{i'}}\frac{\pd\Psi_{j'}^{i'+1}}{\pd u^i_j}\,v_{j'}^{k-1-i'}.
\end{split}
\label{ln3eq31}
\ea
From these we see that \eq{ln3eq28} is equivalent to
\e
\d\Psi_++\al_+(\Phi_+)=0.
\label{ln3eq32}
\e

Now write $\bY=\bSpec A_+^\bu$ and $\bZ=\bSpec B^\bu_+$, as affine derived $\K$-schemes, and $\bs g=\bSpec\al_+:\bZ\ra\bY$. As in Remark \ref{ln2rem3}, we interpret $\bSpec A^\bu$ as a twisted $k$-shifted cotangent bundle $T^*_t[k]\bY$ with projection $\bSpec\io:T^*_t[k]\bY\ra\bY$, where $\io:A^\bu_+\hookra A^\bu$ is the inclusion, and the `twist' $t\in\O_\bY^{k+1}$ is $t=\Phi_+$. Similarly, we interpret $\bSpec B^\bu$ as a twisted $k$-shifted conormal bundle $N^*_{u/t}[k](\bZ/\bY)$ of $\bs g:\bZ\ra\bY$, with projection $\bSpec\jmath:N^*_{u/t}[k](\bZ/\bY)\ra\bZ$, where $\jmath:B^\bu_+\hookra B^\bu$ is the inclusion, and we interpret $\bSpec\al$ as the twisted inclusion morphism~$N^*_{u/t}[k](\bZ/\bY)\ra T^*_t[k]\bY$.

Here $\Psi_+$ is the `twist' $u$ of the $k$-shifted conormal bundle $N^*[k](\bZ/\bY)$, and \eq{ln3eq32} is the compatibility condition $\d u+\bs g^*(t)=0$ with the `twist' $t=\Phi_+$ of $T^*[k]\bY$. The data $\Psi_j^{i+1}$ in $\Psi$ in \eq{ln3eq27} defines the differential $\d$ in $B^\bu_+=(B^*_+,\d)$, via $\d u^i_j =(-1)^{(i+1)k}\Psi^{i+1}_j$ in \eq{ln3eq31}. Equation \eq{ln3eq29} means that $\d\ci\d=0$ in $B^\bu_+=(B^*_+,\d)$, necessary for $B_+^\bu$ to be a cdga and $\bZ$ a derived scheme.
\end{rem}

The next example, defining `weak Lagrangian Darboux form' and `strong Lagrangian Darboux form' when $k\equiv 3\mod 4$, is related to Example \ref{ln3ex1} in the same way that Example \ref{ln2ex3} is related to Example \ref{ln2ex2} in~\S\ref{ln25}.

\begin{ex} 
\label{ln3ex2}
Let $k<0$ with $k\equiv 3\mod 4$, suppose $A^\bu,\om$ are in $k$-shifted Darboux form, as in Example \ref{ln2ex2}, and use the notation of Remark \ref{ln2rem3}. These define a standard form cdga $A^\bu$ over $\K$, a sub-cdga $A_+^\bu\subseteq A^\bu$, coordinates $x^i_j$ in $A_+^i\subseteq A^i$ and $y^{k-i}_j$ in $A^{k-i}$ for $i=0,-1,\ldots,d=[(k+1)/2]$ and $j=1,\ldots,m_i$, and a $k$-shifted 2-form $\om^0=\sum_{i=0}^d\sum_{j=1}^{m_i}\dd x^i_j\,\dd y^{k-i}_j$. They also define $\Phi\in A^{k+1}$ satisfying \eq{ln2eq7}, which determines the differential $\d$ in $A^\bu$ by \eq{ln2eq8}, and $\phi\in(\Om^1_{A^\bu})^k$ satisfying $\d\Phi=0$, $\dd\Phi+\d\phi=0$ and $\dd\phi=k\om^0$. As in \eq{ln2eq11} we write $\Phi=\Phi_++\sum_{i=-1}^d\sum_{j=1}^{m_i}\Phi_j^{i+1}y_j^{k-i}$, where $\Phi_+\in A_+^{k+1}$ and $\Phi_j^{i+1}\in A_+^{i+1}$ for all $i,j$ do not involve the $y^i_j$, and we define $\phi_+\in(\Om^1_{A^\bu})^k$ as in~\eq{ln2eq15}.

Write $e=[k/2]$, so that $e=d-1$ and $k=2e+1=2d-1$. Choose nonnegative integers $n_0,n_{-1},\ldots,n_e$. Choose a smooth $\K$-algebra $B^0$ of dimension $m_0+n_0$, and a smooth morphism $\al^0:A^0\ra B^0$. Localizing $B^0$ if necessary, we assume there exist $u^0_1,\ldots,u^0_{n_0}\in B^0$ such that $\dd\ti x^0_1,\ldots,\dd\ti x^0_{m_0},\dd u^0_1,\ldots,\dd u^0_{n_0}$ form a basis of $\Om^1_{B^0}$ over $B^0$, where we write $\ti x^0_j=\al^0(x^0_j)\in B^0$.

As in \eq{ln3eq2}, define $B^*$ as a commutative graded $\K$-algebra to be the free graded algebra over $B^0$
generated by variables
\begin{align*}
&\ti x_1^i,\ldots,\ti x^i_{m_i} &&\text{in degree $i$ for
$i=-1,-2,\ldots,d$, and} \\
&u_1^i,\ldots,u^i_{n_i} &&\text{in degree $i$ for
$i=-1,-2,\ldots,d$, and} \\
&w_1^e,\ldots,w^e_{n_e} &&\text{in degree $e$, and} \\
&v_1^{k-1-i},\ldots,v^{k-1-i}_{n_i} &&\text{in degree $k-1-i$ for
$i=0,-1,\ldots,d$.}
\end{align*}
So the upper index $i$ in $\ti x^i_j,u^i_j,v^i_j,w^i_j$ always indicates the degree. 

As in \eq{ln3eq3}, define a morphism $\al_+:A_+^*\ra B^*$ of commutative graded $\K$-algebras by $\al_+^0=\al^0$ in degree 0 and
\begin{equation*}
\al_+(x^i_j)=\ti x^i_j,\quad i=-1,-2,\ldots,d,\;\> j=1,\ldots,m_i. 
\end{equation*}
This is well-defined as $A_+^*$ is freely generated over $A^0$ by the~$x^i_j$.

Let $q_1,\ldots,q_{n_e}$ be invertible elements of $B^0$. Choose a superpotential $\Psi$ in $B^k$, which as in \eq{ln2eq18} and \eq{ln3eq4} we require to satisfy 
\e
\begin{split}
\sum_{i=-1}^d&\sum_{j=1}^{n_i}\frac{\pd\Psi}{\pd u^i_j}\,
\frac{\pd\Psi}{\pd v^{k-1-i}_j}+\frac{1}{4}\sum_{j=1}^{n_e}
\frac{1}{q_j}\,\biggl(\frac{\pd\Psi}{\pd w^e_j}\biggr)^2\\
&+\al_+(\Phi_+)+\sum_{i=-1}^d\sum_{j=1}^{m_i}(-1)^{i+1}\al_+(\Phi_j^{i+1})\frac{\pd\Psi}{\pd\ti x^i_j}=0
\end{split}
\label{ln3eq33}
\e
in $B^{k+1}$. As in \eq{ln3eq5}, extend $\al_+$ to $\al:A^*\ra B^*$ by $\al\vert_{A^*_+}=\al_+$ and 
\e
\begin{split}
\al(y^{k-i}_j)&=(-1)^{i+1}\frac{\pd\Psi}{\pd\ti x^i_j},\quad i=-1,\ldots,d,\;\> j=1,\ldots,m_i, \\
\al(y^k_j)&=-\frac{\pd\Psi}{\pd\ti x^0_j}+\sum_{j'=1}^{n_e} \frac{w_{j'}^e}{2q_{j'}}\,\frac{\pd q_{j'}}{\pd\ti x^0_j}\,\frac{\pd\Psi}{\pd w^e_{j'}},\quad j=1,\ldots,m_0.
\end{split}
\label{ln3eq34}
\e
This is well-defined as $A^*$ is freely generated over $A^*_+$ by the $y^{k-i}_j$. Then from \eq{ln2eq11} and \eq{ln3eq34} we see that \eq{ln3eq33} may be rewritten
\begin{equation*}
\sum_{i=-1}^d\sum_{j=1}^{n_i}\frac{\pd\Psi}{\pd u^i_j}\,
\frac{\pd\Psi}{\pd v^{k-1-i}_j}+\frac{1}{4}\sum_{j=1}^{n_e}
\frac{1}{q_j}\,\biggl(\frac{\pd\Psi}{\pd w^e_j}\biggr)^2+\al(\Phi)=0.
\end{equation*}

As in \eq{ln2eq19} and \eq{ln3eq6}, define the differential $\d$ in the cdga $B^\bu=(B^*,\d)$ by $\d=0$ on $B^0$, and
\e
\begin{aligned}
\d\ti x^i_j&=(-1)^{i+1}\al_+(\Phi_j^{i+1}),&&
\begin{subarray}{l}\ts i=-1,-2,\ldots,d,\\[3pt] \ts j=1,\ldots,m_i, \end{subarray} \\[3pt]
\d u^i_j&=(-1)^{i+1}\frac{\pd\Psi}{\pd v^{k-1-i}_j},&&\begin{subarray}{l}\ts i=-1,-2,\ldots,d,\\[3pt] \ts j=1,\ldots,n_i,\end{subarray} \\
\d v^{k-1-i}_j&=\frac{\pd\Psi}{\pd u^i_j},&&\begin{subarray}{l}\ts i=-1,-2,\ldots,d,\\[3pt] \ts j=1,\ldots,n_i,\end{subarray}\\
\d v^{k-1}_j&=\frac{\pd\Psi}{\pd u^0_j}-\sum_{j'=1}^{n_e} \frac{w_{j'}^e}{2q_{j'}}\,\frac{\pd q_{j'}}{\pd u^0_j}\,\frac{\pd\Psi}{\pd w^e_{j'}},&& j=1,\ldots,n_0,\\
\d w^e_j&=\frac{1}{2q_j}\,\frac{\pd\Psi}{\pd
w^e_j}, && j=1,\ldots,n_e.
\end{aligned}
\label{ln3eq35}
\e
We prove that $\d\ci\d=0$ as in \eq{ln3eq7}--\eq{ln3eq9}, applying $\pd/\pd u^{i'}_{j'},\pd/\pd v^{k-1-i'}_{j'},\pd/\pd w^e_{j'}$ to \eq{ln3eq33}. Thus $B^\bu$ is a standard form cdga over~$\K$.

As in \eq{ln3eq10}--\eq{ln3eq11} we can check that $\d\ci\al(x^i_j)=\al\ci\d x^i_j$ and $\d\ci\al(y^{k-i}_j)=\al\ci\d y^{k-i}_j$, so that $\d\ci\al=\al\ci\d$, and $\al:A^\bu\ra B^\bu$ is a cdga morphism.

Following \eq{ln2eq17} and \eq{ln3eq12}, define $h^0\in (\La^2\Om^1_{B^\bu})^{k-1}$ by 
\e
h^0=\sum_{i=0}^d\sum_{j=1}^{n_i}\dd u^i_j\,\dd v^{k-1-i}_j+\sum_{j=1}^{n_e}\dd\bigl(q_j w^e_j\bigr)\,\dd
w^e_j.
\label{ln3eq36}
\e
Then $\dd h^0=0$, and as in \eq{ln3eq13} we can show that $\d h^0=\al_*(\om^0)$. Therefore $h:=(h^0,0,0,\ldots)$ is an isotropic structure for $\bSpec\al:\bSpec B^\bu\ra\bSpec A^\bu$ and the $k$-shifted symplectic structure $\om=(\om^0,0,0,\ldots)$ on $\bSpec A^\bu$. Following the argument of \eq{ln3eq14}--\eq{ln3eq20} we can prove this isotropic structure is nondegenerate, so that $\bSpec B^\bu$ is Lagrangian in~$(\bSpec A^\bu,\om)$. 

Following \eq{ln2eq20} and \eq{ln3eq21}, define $\psi\in(\Om^1_{B^\bu})^{k-1}$ by 
\begin{align*}
\psi=\sum_{i=0}^d&\sum_{j=1}^{n_i}\bigl[i\,u^i_j\,\dd v^{k-1-i}_j +(-1)^{i+1}(k-1-i)v^{k-1-i}_j\,\dd u^i_j\bigr]\\[-6pt]
&+(k-1)\sum_{j=1}^{n_e}q_j\,w^e_j\,\dd w^e_j.
\end{align*}
As in \eq{ln3eq25}--\eq{ln3eq26}, we can show that equations \eq{ln3eq22}--\eq{ln3eq24} hold.

Following the notation of weak and strong Darboux form in Example \ref{ln2ex3}, we say that $A^\bu,\om,B^\bu,\al,h$ are in {\it weak Lagrangian Darboux form}. If all the above holds with $q_j=1$ for $j=1,\ldots,n_e$, we say that $A^\bu,\om,B^\bu,\al,h$ are in {\it strong Lagrangian Darboux form}. This concludes Example~\ref{ln3ex2}. 
\end{ex}

The next example, similar to Example \ref{ln2ex5}, discusses Example \ref{ln3ex2} in more detail when~$k=-1$.

\begin{ex} 
\label{ln3ex3}
Consider `weak Lagrangian Darboux form' in Example \ref{ln3ex2} when $k=-1$. Example \ref{ln2ex2} gives $A^\bu,\om$, where $A^0$ is a smooth $\K$-algebra, elements $x_1^0,\ldots,x_{m_0}^0\in A^0$ such that $\dd x^0_1,\ldots,\dd x^0_{m_0}$ form a basis of $\Om^1_{A^0}$ over $A^0$, and a Hamiltonian $\Phi\in A^0$. The classical master equation \eq{ln2eq7} is trivial in this case, so $\Phi$ is arbitrary. We have $A^\bu=A^0[y^{-1}_1,\ldots,y^{-1}_{m_0}]$, where $y^{-1}_1,\ab\ldots,\ab y^{-1}_{m_0}$ have degree $-1$, with differential
\begin{equation*}
\d x^0_j=0,\quad \d y^{-1}_j=\frac{\pd \Phi}{\pd x^0_j}, \quad
j=1,\ldots,m_0,
\end{equation*}
and $-1$-shifted 2-form
\begin{equation*}
\om^0=\dd x^0_1\,\dd y^{-1}_1+\cdots+\dd
x^0_{m_0}\,\dd y^{-1}_{m_0}.
\end{equation*}
Then $\om=(\om^0,0,0,\ldots)$ is a $-1$-shifted symplectic structure on $\bX=\bSpec A^\bu$. Note that~$H^0(A^\bu)=A^0/(\frac{\pd\Phi}{\pd x^0_1},\ldots,\frac{\pd\Phi}{\pd x^0_{m_0}})=A^0/(\dd \Phi)$.

Geometrically, $U=\Spec A^0$ is a smooth classical $\K$-scheme with \'etale coordinates $(x^0_1,\ldots,x^0_{m_0}):U\ra\bA^{m_0}$, and $\Phi:U\ra\bA^1$ is regular, and $\bX=\bs\Crit(\Phi)$ is the derived critical locus of $\Phi$, with $X=t_0(\bX)$ the classical critical locus $\Crit(\Phi)$. As in Example \ref{ln2ex2} and \cite[Prop.~5.7(b)]{BBJ}, the restriction $\Phi\vert_{X^\red}:X^\red\ra\bA^1$ of $\Phi$ to the reduced $\K$-subscheme $X^\red$ of $X$ is locally constant. By adding a constant to $\Phi$, we suppose that~$\Phi\vert_{X^\red}=0$.

Example \ref{ln3ex2} now chooses a smooth $\K$-algebra $B^0$, a smooth morphism $\al^0:A^0\ra B^0$, and elements $u^0_1,\ldots,u^0_{n_0}$ such that $\dd\ti x^0_1,\ldots,\dd \ti x^0_{m_0},\dd u^0_1,\ldots,\ab\dd u^0_{n_0}$ form a basis of $\Om^1_{B^0}$ over $B^0$, where $\ti x^0_j=\al^0(x^0_j)$. As a graded $\K$-algebra we have $B^*=B^0[v^{-2}_1,\ldots,v^{-2}_{n_0},w^{-1}_1,\ldots,w^{-1}_{n_{-1}}]$ for some $n_{-1}\ge 0$, with $v^{-2}_j$ in degree $-2$ and $w^{-1}_j$ in degree~$-1$.

We choose invertible elements $q_1,\ldots,q_{n_{-1}}$ in $B^0$, and a superpotential $\Psi\in B^{-1}$, which we write in the form
\begin{equation*}
\Psi=s_1w^{-1}_1+\cdots+s_{n_{-1}}w^{-1}_{n_{-1}}
\end{equation*}
for $s_1,\ldots,s_{n_{-1}}\in B^0$. The p.d.e.\ \eq{ln3eq33} which $\Psi$ must satisfy reduces to
\e
\frac{1}{4}\sum_{j=1}^{n_{-1}} \frac{(s_j)^2}{q_j}+\al^0(\Phi)=0.
\label{ln3eq37}
\e

By \eq{ln3eq34}, the morphism $\al:A^\bu\ra B^\bu$ is determined by $\al\vert_{A^0}=\al^0$ and 
\begin{equation*}
\al(y^{-1}_j)=-\sum_{j'=1}^{n_e}\biggl[\frac{\pd s_{j'}}{\pd\ti x^0_j}-\frac{s_{j'}}{2q_{j'}}\,\frac{\pd q_{j'}}{\pd\ti x^0_j}\biggr]w_{j'}^{-1},
\qquad j=1,\ldots,m_0.
\end{equation*}
By \eq{ln3eq35}, the differential $\d$ in $B^\bu=(B^*,\d)$ is given by $\d=0$ on $B^0$ and
\begin{align*}
\d v^{-2}_j&=\sum_{j'=1}^{n_e}\biggl[\frac{\pd s_{j'}}{\pd u^0_j}-\frac{s_{j'}}{2q_{j'}}\,\frac{\pd q_{j'}}{\pd u^0_j}\biggr]w_{j'}^{-1}, && j=1,\ldots,n_0,\\
\d w^{-1}_j&=\frac{s_j}{2q_j}, && j=1,\ldots,n_{-1}.
\end{align*}
Then $\d\ci\d v^{-2}_j=0$ follows by applying $\frac{\pd}{\pd u^0_j}$ to \eq{ln3eq37}. The Lagrangian structure is $h=(h^0,0,0,\ldots)$, where $h^0\in (\La^2\Om^1_{B^\bu})^{-2}$ is given by
\begin{equation*}
h^0=\sum_{j=1}^{n_0}\dd u^0_j\,\dd v^{-2}_j+\sum_{j=1}^{n_{-1}}\dd\bigl(q_j w^{-1}_j\bigr)\,\dd w^{-1}_j.
\end{equation*}

Geometrically, we have a smooth classical $\K$-scheme $V=\Spec B^0$ with \'etale coordinates $(\ti x^0_1,\ldots,\ti x^0_{m_0},u^0_1,\ldots,u^0_{n_0}):V\ra \bA^{m_0+n_0}$, a smooth morphism $\pi=\Spec\al^0:V\ra U$ acting in coordinates by $\pi:(\ti x^0_1,\ldots,\ti x^0_{m_0},u^0_1,\ldots,u^0_{n_0})\mapsto (\ti x^0_1,\ldots,\ti x^0_{m_0})$, a trivial vector bundle $E\ra V$ with fibre $\K^{n_{-1}}$, a nondegenerate quadratic form $Q$ on $E$ given by
\begin{equation*}
Q(e_1,\ldots,e_{n_{-1}})=\frac{(e_1)^2}{q_1}+\cdots+\frac{(e_{n_{-1}})^2}{q_{n_{-1}}}
\end{equation*}
for all regular functions $e_1,\ldots,e_{n_{-1}}:V\ra\bA^1$, and a
section $s=(s_1,\ldots,s_{n_{-1}})$ in $H^0(E)$ which by \eq{ln3eq37} satisfies
\begin{equation*}
Q(s,s)+4\pi^*(\Phi)=0.
\end{equation*}

To summarize:
\begin{itemize}
\setlength{\itemsep}{0pt}
\setlength{\parsep}{0pt}
\item The important geometric data in writing a $-1$-shifted symplectic derived $\K$-scheme $(\bX,\om)$ in `Darboux form' is a smooth $\K$-scheme $U$ and a regular function $\Phi:U\ra\bA^1$ with $\Phi\vert_{\Crit(\Phi)^\red}=0$, and then $\bX$ is the derived critical locus~$\bs\Crit(\Phi)$.
\item The important geometric data in writing a Lagrangian $\bs f:\bs L\ra\bX$ in $(\bX,\om)$ in `weak Lagrangian Darboux form' is a smooth $\K$-scheme $V$, a smooth morphism $\pi:V\ra U$, a vector bundle $E\ra V$, a nondegenerate quadratic form $Q$ on $E$, and a section $s\in H^0(E)$ with $Q(s,s)+4\pi^*(\Phi)=0$. Then $t_0(\bs L)$ is the $\K$-subscheme $s^{-1}(0)$ in $V$, and $t_0(\bs f)$ is~$\pi\vert_{s^{-1}(0)}$.
\end{itemize}
The remaining data is choices of \'etale coordinates $(x^0_1,\ldots,x^0_{m_0})$ on $U$ and $(\ti x^0_1,\ldots,\ti x^0_{m_0},u^0_1,\ldots,u^0_{n_0})$ on $V$, and a trivialization $E\cong V\t\bA^{n_{-1}}$, but these are not very interesting geometrically.
\end{ex}

\subsection[\texorpdfstring{A `$k$-shifted Lagrangian Neighbourhood Theorem' for $k<0$}{A \textquoteleft k-shifted Lagrangian Neighbourhood Theorem\textquoteright\ for k<0}]{A `$k$-shifted Lagrangian Neighbourhood Theorem'}
\label{ln33}

Here is the main result of this paper, proved in~\S\ref{ln42}--\S\ref{ln47}.

\begin{thm} 
\label{ln3thm2}
Let\/ $(\bX,\om_\bX)$ be a $k$-shifted symplectic derived\/ $\K$-scheme for $k<0,$ and\/ $\bs f:\bs L\ra\bX$ be a Lagrangian derived\/ $\K$-scheme in $(\bX,\om_\bX),$ with isotropic structure $h_{\bs L}:0\,{\buildrel\sim\over\longra}\,\bs f^*(\om_\bX)$. Let\/ $y\in\bs L$ with\/~$\bs f(y)=x\in\bX$.

Suppose we are given a standard form cdga $A^\bu$ over $\K,$ a $k$-shifted symplectic form\/ $\om$ on $\bSpec A^\bu$ with\/ $A^\bu,\om$ in Darboux form (as in Example\/ {\rm\ref{ln2ex2}} and Remark\/ {\rm\ref{ln2rem3},} which also define a sub-cdga $A^\bu_+\subseteq A^\bu$), a point\/ $p\in\bSpec A^\bu,$ and a morphism $\bs i:\bSpec A^\bu\ra\bX$ which is either a Zariski open inclusion or \'etale, with\/ $\bs i(p)=x$ and\/ $\om\sim\bs i^*(\om_\bX)$. We do not assume $A^\bu$ is minimal at\/~$p$.

(As an aside, we note that Theorem\/ {\rm\ref{ln2thm5}(i),(iv)} guarantee such\/ $A^\bu,\om,p,\bs i$ exist, where we may take $A^\bu$ to be minimal at\/ $p,$ and\/ $\bs i$ to be a Zariski open inclusion for $k\not\equiv 2\mod 4,$ and \'etale for $k\equiv 2\mod 4,$ since if\/ $k\equiv 2\mod 4$ then $\vdim\bX=2\vdim\bs L$ is even near\/~$x$.)

Then there exist a standard form cdga $B^\bu$ over $\K,$ a point\/ $q\in\bSpec B^\bu,$ a morphism $\al:A^\bu\ra B^\bu$ in $\cdga_\K$ with\/ $\bSpec\al(q)=p$ such that\/ $\al_+:=\al\vert_{A^\bu_+}:A^\bu_+\ra B^\bu$ is a submersion minimal at\/ $q$ in the sense of Definition\/ {\rm\ref{ln3def1},} a morphism $\bs j:\bSpec B^\bu\hookra\bs L$ which is either a Zariski open inclusion or \'etale, with\/ $\bs j(q)=y,$ in a homotopy commutative diagram
\e
\begin{gathered}
\xymatrix@C=90pt@R=15pt{ *+[r]{\bSpec B^\bu\,} \ar[d]^{\bSpec\al} \ar[r]_{\bs j} & *+[l]{\bs L} \ar[d]_{\bs f}  \\ *+[r]{\bSpec A^\bu\,} \ar[r]^{\bs i}  & *+[l]{\bX,\!} }
\end{gathered}
\label{ln3eq38}
\e
and a Lagrangian structure $h:0\,{\buildrel\sim\over\longra}\,\al_*(\om)$ on\/ $\bSpec B^\bu$ which is compatible with\/ $h_{\bs L}$ in the sense that the following diagram homotopy commutes
\e
\begin{gathered}
\xymatrix@C=110pt@R=15pt{ *+[r]{0=\bs j^*(0)} \ar[d]^h \ar[r]_(0.53){\bs j^*(h_{\bs L})} & *+[l]{\bs j^*\ci\bs f^*(\om_\bX)} \ar[d]_\sim  \\ *+[r]{(\bSpec\al)^*(\om)=\al_*(\om)} \ar[r]^(0.53)\sim  & *+[l]{(\bSpec\al)^*\ci\bs i^*(\om_\bX),\!} }
\end{gathered}
\label{ln3eq39}
\e
where the bottom equivalence comes from the homotopy $\om\sim\bs i^*(\om_\bX),$ and the right equivalence from the homotopy across \eq{ln3eq38}. Furthermore:
\begin{itemize}
\setlength{\itemsep}{0pt}
\setlength{\parsep}{0pt}
\item[{\bf(i)}] If\/ $k\not\equiv 3\mod 4$ and\/ $\bs i$ is a Zariski open inclusion, then we may take $\bs j$ to be a Zariski open inclusion, and\/ $A^\bu,\om,B^\bu,\al,h$ to be in Lagrangian Darboux form, as in Example\/ {\rm\ref{ln3ex1}}. 

If instead\/ $\bs i$ is \'etale, the same holds with $\bs j$ \'etale.
\item[{\bf(ii)}] If\/ $k\equiv 3\mod 4$ and\/ $\bs i$ is a Zariski open inclusion, then we may take $\bs j$ to be a Zariski open inclusion, and\/ $A^\bu,\om,B^\bu,\al,h$ to be in weak Lagrangian Darboux form, as in Example\/~{\rm\ref{ln3ex2}}.
\item[{\bf(iii)}] If\/ $k\equiv 3\mod 4$ then we may take $\bs j$ to be \'etale, and\/ $A^\bu,\om,B^\bu,\al,h$ to be in strong Lagrangian Darboux form, as in Example\/~{\rm\ref{ln3ex2}}.
\end{itemize}
\end{thm}

\begin{rem} 
\label{ln3rem2}
Let $(\bX,\om_\bX)$ be a $k$-shifted symplectic derived $\K$-scheme for $k<0$, and $\bs f:\bs L\ra\bX$, $h_{\bs L}$ a Lagrangian in $(\bX,\om_\bX)$, and $y\in\bs L$ with $\bs f(y)=x\in\bX$. For clarity, we spell out what Theorems \ref{ln2thm5} and \ref{ln3thm2} together tell us about joint local models for $\bX,\bs L$ near $x,y$, for various $k=-1,-2,\ldots:$ 
\begin{itemize}
\setlength{\itemsep}{0pt}
\setlength{\parsep}{0pt}
\item[(a)] If $k<0$ with $k\equiv 0$ or $1\mod 4$, then Theorem \ref{ln2thm5} gives a `Darboux form' Zariski local model $A^\bu,\om$ for $(\bX,\om_\bX)$, and Theorem \ref{ln3thm2} gives a `Lagrangian Darboux form' Zariski local model $B^\bu,\al,h$ for $\bs f:\bs L\ra\bX$.
\item[(b)] If $k<0$ with $k\equiv 2\mod 4$, then Theorem \ref{ln2thm5} gives a `Darboux form' local model $A^\bu,\om$ for $(\bX,\om_\bX)$ only in the \'etale topology, so Theorem \ref{ln3thm2} gives a corresponding `Lagrangian Darboux form' local model $B^\bu,\al,h$ for $\bs f:\bs L\ra\bX$ only in the \'etale topology.
\item[(c)] If $k<0$ with $k\equiv 3\mod 4$, Theorem \ref{ln2thm5} gives a `Darboux form' Zariski local model $A^\bu,\om$ for $(\bX,\om_\bX)$. Then Theorem \ref{ln3thm2} gives a `weak Lagrangian Darboux form' Zariski local model $B^\bu,\al,h$, and also a `strong Lagrangian Darboux form' \'etale local model $B^\bu,\al,h$, for $\bs f:\bs L\ra\bX$.
\end{itemize}

Our theorems thus do not provide local models in the Zariski topology for Lagrangians in general $k$-shifted symplectic derived schemes when $k\equiv 2\mod 4$. This is due to the laziness of the authors. One should find a `weak Lagrangian Darboux form' adapted to the `weak Darboux form' of Example~\ref{ln2ex3}.
\end{rem}

\subsection{\texorpdfstring{The case $k=0$}{The case k=0}}
\label{ln34}

When $k=0$, a 0-shifted symplectic derived $\K$-scheme $(\bX,\om_\bX)$ is simply a smooth classical $\K$-scheme $X$ with a classical symplectic form $\om_X$. However, Lagrangians $\bs f:\bs L\ra\bX$ in them, in the sense of \S\ref{ln24}, need {\it not\/} be smooth classical Lagrangians; they can be truly derived objects, singular at the classical level, as Example \ref{ln2ex1} shows for Lagrangians in the point. So a `$k$-shifted derived Lagrangian Neighbourhood Theorem' is still of interest when $k=0$, and is more closely related to classical symplectic geometry than the $k<0$ case.

If we try to extend Theorem \ref{ln3thm2} to the case $k=0$, two things go wrong:
\begin{itemize}
\setlength{\itemsep}{0pt}
\setlength{\parsep}{0pt}
\item[(a)] As is well known, although the classical Darboux Theorem holds for real $C^\iy$ and complex symplectic manifolds, it is {\it false\/} for algebraic symplectic manifolds (symplectic schemes). So given a 0-shifted derived Lagrangian $\bs f:\bs L\ra X$, $h$ in a general classical symplectic scheme $(X,\om_X)$, we do not have `Darboux form' local models $A^\bu,\om$ for $(X,\om_X)$ near $x\in X$.
\item[(b)] Even if we assume that $(X,\om_X)$ has a very nice local model near $x$ (e.g.\ if $X=\bA^{2n}$ and $\om_X=\sum_{j=1}^n\dd x_j\dd y_j$), in the proof of Theorem \ref{ln3thm2} in \S\ref{ln42}--\S\ref{ln47}, Proposition \ref{ln4prop1} fails when $k=0$, as there is an obstruction in ${\rm H}^1_{\rm inf}(H^0(B^\bu))$ which need not vanish Zariski or \'etale locally.
\end{itemize}

The next two examples define notions of `Darboux form' for 0-shifted symplectic schemes, and `Lagrangian Darboux form' for Lagrangians in them. Because of (a),(b) they are not local models for general symplectic schemes and their Lagrangians, but at least they are local models for especially nice symplectic schemes and nice 0-shifted Lagrangians.

\begin{ex} 
\label{ln3ex4}
Suppose $A^0_+$ is a smooth $\K$-algebra of dimension $m_0$, and $x^0_1,\ab\ldots,\ab x^0_{m_0}\in A^0_+$ such that $\dd x^0_1,\ldots,\dd x^0_{m_0}$ form a basis of $\Om^1_{A^0_+}$ over $A^0_+$. Let $A^0=A^0_+[y^0_1,\ldots,y^0_{m_0}]$ be the $\K$-algebra freely generated over $A^0_+$ by variables $y^0_1,\ldots,y^0_{m_0}$ in degree 0, and write $\io:A^0_+\hookra A^+$ for the inclusion. Regard $A_+^0,A^0$ as cdgas $A_+^\bu,A^\bu$ concentrated in degree 0. Define
\begin{equation*}
\om^0=\ts\sum_{j=1}^{m_0}\dd x_j^0\dd y_j^0\qquad\text{in $\La^2\Om^1_{A^0}$.}
\end{equation*}
Then $\om^0\in (\La^2\Om^1_{A^\bu})^0$ with $\d\om^0=\dd\om^0=0$, and $\om:=(\om^0,0,0,\ldots)$ is a 0-shifted symplectic structure on $\bSpec A^\bu$. We say that $A^\bu,\om$ (and $A^\bu_+\subseteq A^\bu$) are in {\it Darboux form}. Following \eq{ln2eq15} we also define 
\begin{equation*}
\phi_+=\ts\sum_{j=1}^{m_0}y_j^0\dd x_j^0\qquad\text{in $\La^1\Om^1_{A^0}$,}
\end{equation*}
and then $\dd\phi_+=-\om^0$ and $\d\phi_+=0$.

Geometrically, $U=\Spec A^0_+$ is a smooth $\K$-scheme of dimension $m_0$, and $(x^0_1,\ldots,x^0_{m_0}):U\ra\bA^{m_0}$ are \'etale coordinates on $U$, and $\Spec A^0=T^*U$ is its cotangent bundle with projection $\pi=\Spec\io:T^*U\ra U$, and $\om^0$ is the canonical symplectic form on $T^*U$.
\end{ex}

The next example, similar to Example \ref{ln2ex4}, is basically Example \ref{ln3ex1} for~$k\!=\!0$.

\begin{ex} 
\label{ln3ex5}
Use the notation of Example \ref{ln3ex4}. Choose a smooth $\K$-algebra $B^0$ of dimension $m_0+n_0$, and a smooth morphism $\al^0_+:A^0_+\ra B^0$. Localizing $B^0$ if necessary, we may assume there exist $u^0_1,\ldots,u^0_{n_0}\in B^0$ such that $\dd\ti x^0_1,\ldots,\dd\ti x^0_{m_0},\dd u^0_1,\ldots,\dd u^0_{n_0}$ form a basis of $\Om^1_{B^0}$ over $B^0$, where we write $\ti x^0_j=\al^0_+(x^0_j)\in B^0$. As in \eq{ln3eq2}, define $B^*=B^0[v_1^{-1},\ldots,v_{n_0}^{-1}]$ to be the free graded algebra over $B^0$ generated by variables $v_1^{-1},\ldots,v_{n_0}^{-1}$ in degree $-1$. Choose a superpotential $\Psi$ in $B^0$. The p.d.e.\ \eq{ln3eq4} is trivial in this case, so $\Psi$ is arbitrary. As in \eq{ln3eq5}, extend $\al_+$ to $\al:A^\bu\ra B^\bu$ by $\al\vert_{A^0_+}=\al_+$ and 
\begin{equation*}
\al(y^0_j)=-\frac{\pd\Psi}{\pd\ti x^0_j},\qquad j=1,\ldots,m_0.
\end{equation*}

Define the differential $\d$ in the cdga $B^\bu=(B^*,\d)$ by $\d=0$ on $B^0$, and
\begin{equation*}
\d v^{-1}_j=\frac{\pd\Psi}{\pd u^i_j},\qquad j=1,\ldots,n_0,\
\end{equation*}
as in \eq{ln3eq6}. Then $\d\ci\d=0$ trivially. As in \eq{ln3eq12}, define $h^0\in (\La^2\Om^1_{B^\bu})^{-1}$ by 
\begin{equation*}
h^0=\ts\sum_{j=1}^{n_0}\dd u^0_j\,\dd v^{-1}_j.
\end{equation*}
Then $\dd h^0=0$, and \eq{ln3eq13} implies that $\d h^0=\al_*(\om^0)$. Hence $h:=(h^0,0,0,\ldots)$ is an isotropic structure for $\bSpec\al:\bSpec B^\bu\ra\bSpec A^\bu$ and the 0-shifted symplectic structure $\om=(\om^0,0,0,\ldots)$ on $\bSpec A^\bu$. Following \eq{ln3eq14}--\eq{ln3eq20} we prove that $h$ is nondegenerate, so that $\bSpec B^\bu$ is Lagrangian in $(\bSpec A^\bu,\om)$. We say that $A^\bu,\om,B^\bu,\al,h$ are in {\it Lagrangian Darboux form}.

Following \eq{ln3eq21}, define $\psi\in(\Om^1_{B^\bu})^{-1}$ by 
\begin{equation*}
\psi=\ts-\sum_{j=1}^{n_0}v^{-1}_j\,\dd u^0_j.
\end{equation*}
As for \eq{ln3eq23}--\eq{ln3eq24} we have
\begin{align*}
\dd\Psi+\d\psi&=-\al_*(\phi_+)&&\text{in $(\Om^1_{B^\bu})^0$, and}\\
\dd\psi&=-h^0 &&\text{in $(\La^2\Om^1_{B^\bu})^{-1}$.}
\end{align*}

Geometrically, $V=\Spec B^0$ is a smooth $\K$-scheme, $\pi:=\Spec\al^0_+:V\ra U=\Spec A^0_+$ is a smooth morphism of $\K$-schemes, and $\Psi:V\ra\bA^1$ is a regular function. We should interpret $\bs L:=\bSpec B^\bu$ as the {\it derived relative critical locus\/} $\bs\Crit(\Psi/U)$ of $\Psi:V\ra\bA^1$ relative to $\pi:V\ra U$. Heuristically, $\bs L$ is the total space of a family of derived critical loci over the base $U$:
\begin{equation*}
\bs L=\bs\Crit(\Psi/U)\approx\ts\coprod_{u\in U}\bs\Crit\bigl(\Psi\vert_{V_u}:V_u\ra\bA^1\bigr),
\end{equation*}
where $V_u=\pi^{-1}(u)$ is the (smooth) fibre of $\pi:V\ra U$ over $u\in U$.

The morphism $\bSpec\al:\bs L\ra T^*U$ can now be understood as follows. We have a commutative diagram of vector bundles on $\bs L\subseteq V$:
\begin{equation*}
\xymatrix@C=30pt@R=15pt{ && \O_{\bs L} \ar@/_.4pc/@{..>}[dl]_(0.6)\la \ar[d]^(0.55){\dd\Psi\vert_{\bs L}} \ar@/^.4pc/[dr]^(0.6)0
\\
0 \ar[r] & \pi^*(T^*U)\vert_{\bs L} \ar[r] & T^*V\vert_{\bs L} \ar[r] & T^*(V/U)\vert_{\bs L} \ar[r] & 0, }
\end{equation*}
with the bottom row exact. The section $\dd\Psi\vert_{\bs L}$ of $T^*V\vert_{\bs L}$ projects to 0 in $T^*(V/U)\vert_{\bs L}$, since $\bs L$ is the derived zero locus of $\dd\Psi$ in $T^*(V/U)$. Hence by exactness $\dd\Psi\vert_{\bs L}$ lifts to a section $\la$ of $\pi^*(T^*U)\vert_{\bs L}$, and $\bSpec\al:\bs L\ra T^*U$ can be interpreted as the graph of~$-\la$.
\end{ex}

We can now prove the following somewhat weak and unsatisfactory `0-shifted Lagrangian Neighbourhood Theorem':

\begin{thm} 
\label{ln3thm3}
Let\/ $(\bX,\om_\bX)$ be a $0$-shifted symplectic derived\/ $\K$-scheme, and\/ $\bs f:\bs L\ra\bX$ be a Lagrangian derived\/ $\K$-scheme in $(\bX,\om_\bX),$ with isotropic structure $h_{\bs L}:0\,{\buildrel\sim\over\longra}\,\bs f^*(\om_\bX)$. Let\/ $y\in\bs L$ with\/~$\bs f(y)=x\in\bX$.

Suppose we are given a standard form cdga $A^\bu$ over\/ $\K,$ a\/ $0$-shifted symplectic form\/ $\om$ on $\bSpec A^\bu$ with\/ $A^\bu,\om$ in Darboux form as in Example\/ {\rm\ref{ln3ex4},} which also defines $A^\bu_+\subseteq A^\bu,$ a point\/ $p\in\bSpec A^\bu,$ and a Zariski open inclusion $\bs i:\bSpec A^\bu\ra\bX$ with\/ $\bs i(p)=x$ and\/~$\om\sim\bs i^*(\om_\bX)$. 

Then we can define an \begin{bfseries}obstruction class\end{bfseries} $[\ga]$ in
\e
{\rm H}^1_{\rm inf}(t_0(\bs L))_y=\underrightarrow{\lim}\,_{y\in U\subseteq t_0(\bs L)}{\rm H}^1_{\rm inf}(U),
\label{ln3eq40}
\e
where the direct limit is over Zariski open neighbourhoods\/ $U$ of\/ $y$ in the classical\/ $\K$-scheme $t_0(\bs L),$ and\/ ${\rm H}^1_{\rm inf}(-)$ is algebraic de Rham cohomology.

If this obstruction class $[\ga]$ is zero then there exist a standard form cdga $B^\bu$ over $\K,$ a point\/ $q\in\bSpec B^\bu,$ a morphism\/ $\al:A^\bu\ra B^\bu$ in\/ $\cdga_\K$ with\/ $\bSpec\al(q)=p$ such that\/ $\al_+:=\al\vert_{A^\bu_+}:A^\bu_+\ra B^\bu$ is a submersion minimal at\/ $q$ in the sense of Definition\/ {\rm\ref{ln3def1},} a Zariski open inclusion $\bs j:\bSpec B^\bu\hookra\bs L$ with\/ $\bs j(q)=y$ in a homotopy commutative diagram {\rm\eq{ln3eq38},} and a Lagrangian structure $h:0\,{\buildrel\sim\over\longra}\,\al_*(\om)$ on $\bSpec B^\bu$ for which\/ \eq{ln3eq39} homotopy commutes, such that\/ $A^\bu,\om,B^\bu,\al,h$ are in Lagrangian Darboux form, as in Example\/~{\rm\ref{ln3ex5}}. 

If instead\/ $\bs i$ is \'etale, the same holds with $\bs j$ \'etale.
\end{thm}

\begin{proof} We follow the proof of Theorem \ref{ln3thm2} for $k<0$ even in \S\ref{ln42}--\S\ref{ln45}, setting $k=0$ and $\Phi_+=\Phi=\phi=0$. The only place where taking $k=0$ causes problems is in the proof of Proposition \ref{ln4prop1}. Then in place of \eq{ln4eq9} we have that
\begin{equation*}
\ga=\bigl(-\al_*(\phi_+),-h^0,-h^1,\ldots\bigr)
\end{equation*}
is a 0-shifted closed 1-form on $\bSpec B^\bu$, that is, a closed element of degree 0 in the complex $\bigl(\prod_{i\ge 0}(\La^{i+1}\bL_{B^\bu})[i],\dd+\d\bigr)$. The analogue of \eq{ln4eq10} is
\begin{align*}
H^0\bigl(\ts\prod_{i\ge 0}&(\La^{i+1}\bL_{B^\bu})[i],\dd+\d\bigr)\cong\HN^0(B^\bu)(1)\cong \HP^0(B^\bu)(1)\\
&\cong \HP^0(H^0(B^\bu))(1)\cong {\rm H}^1_{\rm inf}(H^0(B^\bu))\cong {\rm H}^1_{\rm inf}(U),
\end{align*}
where $U=t_0(\bs j)(\Spec H^0(B^\bu))$ is a Zariski open neighbourhood of $y$ in $t_0(\bs L)$.

The vanishing theorems used in Proposition \ref{ln4prop1} fail when $k=0$, so we may have ${\rm H}^1_{\rm inf}(U)\ne 0$. If the cohomology class $[\ga]$ of $\ga$ in ${\rm H}^1_{\rm inf}(U)$ is nonzero then $\Xi,\psi$ in Proposition \ref{ln4prop1} do not exist, so we cannot continue the proof.

If there is some Zariski open neighbourhood $V$ of $q$ in $\Spec H^0(B^\bu)$ such that $[\ga]$ becomes zero when restricted to ${\rm H}^1_{\rm inf}(V)$, then by localizing $B^\bu$ we can make $[\ga]=0$, so $\Xi,\psi$ in Proposition \ref{ln4prop1} do exist, and the rest of \S\ref{ln42}--\S\ref{ln45} works without a hitch. The condition for there to exist some such $V$ is that the image of $[\ga]$ should be zero in the direct limit \eq{ln3eq40}. This completes the proof.
\end{proof}

\begin{rem} 
\label{ln3rem3}
It seems likely that there exists a good theory of {\it Derived Complex Analytic Geometry}, a complex analytic version of Derived Algebraic Geometry, including {\it derived complex analytic spaces}, built using complex manifolds and holomorphic functions. Within this there should exist good notions of $k$-{\it shifted symplectic derived complex analytic space}, and {\it Lagrangians\/} in these. So far as the authors know, neither theory is yet available in the literature.

If such theories were constructed, the authors expect that the obvious complex analytic generalizations of the $k$-shifted symplectic Darboux Theorem \ref{ln2thm5} and Lagrangian Neighbourhood Theorem \ref{ln3thm2} will hold.

Observe that the two problems (a),(b) in the case $k=0$ above will not occur in the complex analytic case. For (a) the classical Darboux Theorem holds for complex symplectic manifolds, and for (b), the complex analytic analogue of \eq{ln3eq40} is zero, since $H^i$ cohomology classes for $i>0$ on a complex analytic space are zero locally in the complex analytic topology. So the complex analytic versions of Theorems \ref{ln2thm5} and \ref{ln3thm2} should also work when $k=0$. The authors hope in future work to use these ideas to define a notion of `derived Lagrangian' in complex symplectic manifolds, generalizing the `d-critical loci' of Joyce~\cite{Joyc}.
\end{rem}

\subsection{\texorpdfstring{$k$-shifted Poisson structures and coisotropics}{k-shifted Poisson structures and coisotropics}}
\label{ln35}

Recently, Calaque, Pantev, To\"en, Vaqui\'e and Vezzosi \cite{CPTVV} defined $k$-{\it shifted Poisson structures\/} $\pi_\bX$ on a derived scheme or stack $\bX$, for $k\in\Z$, and {\it coisotropics\/} $\bs f:\bs C\ra\bX$ in $(\bX,\pi_\bX)$. They prove \cite[Th.~3.2.4]{CPTVV} that the spaces of $k$-shifted symplectic structures $\om_\bX$ and nondegenerate $k$-shifted Poisson structures $\pi_\bX$ on $\bX$ are equivalent, and for fixed equivalent $\om_\bX,\pi_\bX$ they conjecture \cite[Conj.~3.4.5]{CPTVV} that the spaces of Lagrangian structures on $\bs f:\bs L\ra\bX$ in $(\bX,\om_\bX)$ and nondegenerate coisotropic structures on $\bs f:\bs L\ra\bX$ in $(\bX,\pi_\bX)$ are equivalent. Recently this conjecture has been proved in \cite[Th.~4.22]{MS2}.

The purpose of this section is to observe that for our `Darboux form' local models for $k$-shifted symplectic derived schemes in \S\ref{ln25}, we can write down simple, explicit (strict) $k$-shifted Poisson structures, and for our `Lagrangian Darboux form' local models for Lagrangians in $k$-shifted symplectic derived schemes in \S\ref{ln32}, we can write down simple, explicit (strict) coisotropic structures.

\begin{dfn} 
\label{ln3def2}
A $\bP_{k+1}$-{\it algebra\/} is a cdga $A^\bu$ equipped with the data of a Lie bracket $\{,\}: A^\bu\ot A^\bu\ra A^\bu[-k]$ satisfying the following equations:
\begin{itemize}
\setlength{\itemsep}{0pt}
\setlength{\parsep}{0pt}
\item[(i)] $\{f, g\} = -(-1)^{(|f|+k)(|g|+k)}\{g, f\}$,
\item[(ii)] $\{f, \{g, h\}\} = \{\{f, g\}, h\} + (-1)^{(|f|+k)(|g|+k)}\{g,\{f,h\}\}$,
\item[(iii)] $\d\{f, g\} = \{\d f, g\} + (-1)^{|f|+k}\{f, \d g\}$,
\item[(iv)] $\{f, gh\} = \{f, g\}h + (-1)^{|g|(|f|+k)}g\{f, h\}$,
\end{itemize}
for all elements $f,g,h\in A^\bu$.
\end{dfn}

Note that if $A^\bu$ is a $\bP_{k+1}$-algebra, then, forgetting the multiplication, $A^\bu[k]$ is a dg Lie algebra.

\begin{dfn}
Let $A^\bu$ be a cdga. A {\it $k$-shifted Poisson structure} on $A^\bu$ is a $\bP_{k+1}$-algebra $\widetilde{A}^\bu$ equipped with a quasi-isomorphism of cdgas $\widetilde{A}^\bu\rightarrow A^\bu$.
\end{dfn}

We say that a $k$-shifted Poisson structure on $A^\bu$ is \emph{strict} if $\widetilde{A}^\bu\rightarrow A^\bu$ is an isomorphism; that is, $A^\bu$ itself is a $\bP_{k+1}$-algebra.

Let $\bX=\bSpec A^\bu$ be an affine derived scheme and assume that $A$ is a cofibrant cdga. Then one can define the complex of $k$-shifted polyvector fields to be
\begin{equation*}
\Pol(\bX, k) = \Hom^\bu_{A^\bu}(\Sym(\Omega^1_{A^\bu}[k+1]), A^\bu).
\end{equation*}

This is a graded $\bP_{k+2}$-algebra, where the grading comes from the symmetric algebra, and the Lie bracket is the Schouten bracket of polyvector fields that we denote by $[\,,\,]$. We write $\widehat{\Pol}(\bX,k)$ for its completion with respect to the grading, and $\widehat{\Pol}{}^{\geq 2}(\bX,k)$ for the part in degrees at least~2.

If $A^\bu$ is a $\bP_{k+1}$-algebra, we get a bivector $\pi^2_{A^\bu}\in \Pol(\bX, k)$ by the formula
\begin{equation*}
\{f, g\}=(-1)^{|f| + k + 1}\iota_{\pi_{A^\bu}^2}(\dd f\, \dd g).
\end{equation*}
satisfying the equations
\[\d \pi^2_{A^\bu}=0,\qquad [\pi^2_{A^\bu}, \pi^2_{A^\bu}] = 0,\]
i.e. $\pi^2_{A^\bu}\in\widehat{\Pol}{}^{\geq 2}(\bX, k)$ is a Maurer--Cartan element. More generally, Melani \cite[Th.~3.2]{Mela} proves:

\begin{thm}[Melani] 
\label{ln3thm4}
Let\/ $\bX=\bSpec A^\bu$ be an affine derived scheme. Then the space of\/ $k$-shifted Poisson structures on $\bX$ is equivalent to the space of Maurer--Cartan elements in $\widehat{\Pol}{}^{\geq 2}(\bX,k)$.
\end{thm}

A $k$-shifted Poisson structure $\pi_\bX$ on $\bX$ defines a map $\pi_\bX^2\cdot:\bL_{\bX}\ra \bT_{\bX}[-k]$. We say that $\pi_\bX$ is {\it nondegenerate\/} if this map is a quasi-isomorphism. Calaque et al.\ \cite[Th.~3.2.4]{CPTVV} prove the following theorem (see also Costello--Rozenblyum, to appear, and Pridham \cite{Prid}, for related results).

\begin{thm}[Calaque--Pantev--To\"en--Vaqui\'e--Vezzosi] 
\label{ln3thm5}
The space of nondegenerate $k$-shifted Poisson structures on a derived stack $\bX$ is equivalent to the space of\/ $k$-shifted symplectic structures on\/~$\bX$.
\end{thm}

It is difficult in general to explicitly invert a $k$-shifted symplectic structure to obtain a $k$-shifted Poisson structure, but this can be easily done in the `Darboux form' models of Bussi, Brav and Joyce \cite{BBJ} from~\S\ref{ln25} where, in fact, we obtain \textit{strict} $k$-shifted Poisson structures.

\begin{ex} 
\label{ln3ex6}
Let $k<0$ and suppose $A^\bu,\om$ are in $k$-shifted Darboux form, as in Example \ref{ln2ex2}. The differential of vector fields is given by
\begin{align*}
\d\frac{\pd}{\pd x^i_j} &= \sum_{i'=0}^d \sum_{j'=1}^{m_{i'}}(-1)^{i+i'}\frac{\pd \Phi_{j'}^{i'+1}}{\pd x^i_j}\frac{\pd}{\pd x_{j'}^{i'}} + \sum_{i'=0}^d \sum_{j'=1}^{m_{i'}}(-1)^{i+1} \frac{\pd^2 \Phi_+}{\pd x^i_j\pd x^{i'}_{j'}}\frac{\pd}{\pd y^{k-i'}_{j'}}\\
&\qquad + \sum_{i',i''=0}^d\sum_{j'=1}^{m_{i'}}\sum_{j''=1}^{m_{i''}}
(-1)^{i+1}\frac{\pd^2\Phi_{j''}^{i''+1}}{\pd x^i_j\pd x^{i'}_{j'}}\,y_{j''}^{k-i''}\frac{\pd}{\pd y^{k-i'}_{j'}}, \\
\d\frac{\pd}{\pd y^{k-i}_j} &= \sum_{i'=0}^d\sum_{j'=1}^{m_{i'}} (-1)^{(k+i)(i+i')+1}\frac{\pd \Phi^{i+1}_j}{\pd x^{i'}_{j'}}\frac{\pd}{\pd y^{k-i'}_{j'}}.
\end{align*}

The morphism $\om^0\cdot:\bT_{A^\bu}\ra \bL_{A^\bu}[k]$ is a strict isomorphism given by
\begin{equation*}
\frac{\pd}{\pd x^i_j}\longmapsto \dd y^{k-i}_j,\qquad \frac{\pd}{\pd y^{k-i}_j}\longmapsto (-1)^{(i+1)(k+1)} \dd x^i_j.\end{equation*}
Its inverse $\bL_{A^\bu}\ra \bT_{A^\bu}[-k]$ is given by
\begin{equation*}
\dd y^{k-i}_j\longmapsto (-1)^k \frac{\pd}{\pd x^i_j},\qquad \dd x^i_j\longmapsto (-1)^{ik+i+1} \frac{\pd}{\pd y^{k-i}_j}.
\end{equation*}
This gives a degree $-k$ bivector
\e
\pi^2_{A^\bu} = \sum_{i=0}^d\sum_{j=1}^{m_i} \frac{\pd}{\pd x^i_j}\,\frac{\pd}{\pd y^{k-i}_j}.
\label{ln3eq41}
\e
Its differential is
\begin{align*}
\d \pi^2_{A^\bu} &=\sum_{i,i'=0}^d \sum_{j=1}^{m_i}\sum_{j'=1}^{m_{i'}}(-1)^{i+i'}\biggl[\frac{\pd \Phi_{j'}^{i'+1}}{\pd x^i_j}\frac{\pd}{\pd x_{j'}^{i'}}\frac{\pd}{\pd y^{k-i}_j}-\frac{\pd \Phi_j^{i+1}}{\pd x^{i'}_{j'}}\frac{\pd}{\pd x_j^i}\frac{\pd}{\pd y^{k-i'}_{j'}}\biggr] \\
&\quad+ \sum_{i,i'=0}^d \sum_{j=1}^{m_i}\sum_{j'=1}^{m_{i'}} (-1)^{i+1}\frac{\pd^2 \Phi_+}{\pd x^i_j\pd x^{i'}_{j'}}\frac{\pd}{\pd y^{k-i'}_{j'}}\frac{\pd}{\pd y^{k-i}_j} \\
&\quad+\sum_{i,i',i''=0}^d\sum_{j=1}^{m_i}\sum_{j'=1}^{m_{i'}}\sum_{j''=1}^{m_{i''}} (-1)^{i+1}\frac{\pd^2\Phi_{j''}^{i''+1}}{\pd x^i_j\pd x^{i'}_{j'}}\,y_{j''}^{k-i''}\frac{\pd}{\pd y^{k-i'}_{j'}}\frac{\pd}{\pd y^{k-i}_j}= 0,
\end{align*}
where each line vanishes separately after exchanging the indices $i$ and~$i'$.

Clearly, $[\pi^2_{A^\bu},\pi^2_{A^\bu}]=0$ as the bivector has constant coefficients. Therefore, it defines a strict $k$-shifted Poisson structure on~$\bX=\bSpec A^\bu$.

The same formulae also work trivially when $k=0$, defining $A^\bu,\om$ in 0-shifted Darboux form as in Example \ref{ln3ex4}, with coordinates $x^0_j,y^0_j$, with $\Phi_+=\Phi^{i+1}_j=0$ and $\d\frac{\pd}{\pd x^i_j}=\d\frac{\pd}{\pd y^{k-i}_j}=\d \pi^2_{A^\bu}=0$ for degree reasons, and  $\pi^2_{A^\bu}$ is a classical Poisson structure on~$A^\bu=A^0$.
\end{ex}

\begin{ex}
\label{ln3ex7}
Let $k<0$ with $k\equiv 2\mod 4$ and suppose $A^\bu,\om$ are in $k$-shifted weak Darboux form, as in Example \ref{ln2ex3}. Example \ref{ln3ex6} generalizes to this case, where instead of \eq{ln3eq41} we have
\e
\pi^2_{A^\bu}\!=\!\sum_{i=0}^{d+1} \sum_{j=1}^{m_i} \frac{\pd}{\pd x^i_j}\,\frac{\pd}{\pd y^{k-i}_j}\!+\!\sum_{j=1}^{m_d} \frac{1}{4q_j} \frac{\pd}{\pd z_j^d} \frac{\pd}{\pd z_j^d} \!-\!\sum_{j=1}^{m_d} \sum_{j'=1}^{m_0} \frac{\pd q_j}{\pd x^0_{j'}}\frac{z_j^d}{2q_j} \frac{\pd}{\pd z_j^d}\frac{\pd}{\pd y_{j'}^k}.
\label{ln3eq42}
\e
Then $\d\pi_{A^\bu}^2=0$ and $[\pi^2_{A^\bu},\pi^2_{A^\bu}]=0$, so $\pi^2_{A^\bu}$ defines a strict $k$-shifted Poisson structure on~$\bX=\bSpec A^\bu$.
\end{ex}

Combining Theorems \ref{ln2thm5} and \ref{ln3thm5} gives a $k$-shifted Poisson version of the $k$-shifted Darboux Theorem~\ref{ln3thm2}:

\begin{thm}
\label{ln3thm6}
Let\/ $(\bX,\pi_\bX)$ be a nondegenerate $k$-shifted Poisson derived\/ $\K$-scheme for $k<0,$ and\/ $x\in\bX$. Then $(\bX,\pi_\bX)$ is Zariski locally modelled near $x$ up to equivalence on a strict\/ $k$-shifted Poisson affine derived\/ $\K$-scheme $(\bSpec A^\bu,\pi_{A^\bu}^2)$ in Example\/ {\rm\ref{ln3ex6}} if\/ $k\not\equiv 2\mod 4,$ and in Example\/ {\rm\ref{ln3ex7}} if\/ $k\equiv 2\mod 4$. Also, when $k\equiv 2\mod 4,$ by instead taking $(\bX,\pi_\bX)$ to be \'etale locally modelled on $(\bSpec A^\bu,\pi_{A^\bu}^2)$ we may set\/ $q_1=\cdots=q_{m_d}=1$ in Example\/~{\rm\ref{ln3ex7}}.
\end{thm}

Now let us turn to $k$-shifted Lagrangians and coisotropics. 

\begin{dfn}
A $\bP_{[k+1, k]}$-{\it algebra\/} is a triple of a $\bP_{k+1}$-algebra $A^\bu$, a $\bP_k$-algebra $B^\bu$ and a morphism of $\bP_{k+1}$-algebras
\[A^\bu\longra \bigl(\Hom^\bu_{B^\bu}(\Sym(\Omega^1_{B^\bu}[k]), B),\d+[\pi^2_{B^\bu},-]\bigr).\]
\end{dfn}

Note that if $(A^\bu, B^\bu)$ is a $\bP_{[k+1, k]}$-algebra, the composite
\[A^\bu \longra \bigl(\Hom^\bu_{B^\bu}(\Sym(\Omega^1_{B^\bu}[k]), B),\d+[\pi^2_{B^\bu},-]\bigr) \longra B^\bu,\]
where the latter morphism is given by projection to the weight zero part, is a morphism of cdgas.

\begin{dfn} 
\label{ln3def3}
Let $\al\colon A^\bu\rightarrow B^\bu$ be a morphism of cdgas. A {\it $k$-shifted coisotropic structure\/} on $\al$ is a $\bP_{[k+1, k]}$-algebra $(\widetilde{A}^\bu, \widetilde{B}^\bu)$ together with quasi-isomorphisms of cdgas $\widetilde{A}^\bu\rightarrow A^\bu$ and $\widetilde{B}^\bu\rightarrow B^\bu$ making the diagram of cdgas
\[
\xymatrix{
\widetilde{A}^\bu \ar[r] \ar[d] & \widetilde{B}^\bu \ar[d] \\
A^\bu \ar^{\al}[r] & B^\bu
}
\]
commutative.
\end{dfn}

This definition in fact is equivalent to $k$-shifted coisotropic structures of \cite{CPTVV} as shown in \cite{Saf}. Moreover, one can define a nondegeneracy condition on a $k$-shifted coisotropic structure on $A^\bu\rightarrow B^\bu$ which in particular implies that the $k$-shifted Poisson strucutre on $A^\bu$ is nondegenerate. The following is \cite[Th.~4.22]{MS2}.

\begin{thm}
The space of nondegenerate $k$-shifted Poisson structures on a morphism of derived stacks $\bs L\rightarrow \bX$ is equivalent to the space of pairs of a $k$-shifted symplectic structure on $\bX$ and a Lagrangian structure on $\bs L\rightarrow \bX$.
\label{ln3thm7}
\end{thm}

As before, we say a $k$-shifted coisotorpic structure is \emph{strict} if the maps $\widetilde{A}^\bu\rightarrow A^\bu$ and $\widetilde{B}^\bu\rightarrow B^\bu$ are isomorphisms, so that $(A^\bu, B^\bu)$ is a $\bP_{[k+1, k]}$-algebra. In general, it is difficult to construct a $k$-shifted coisotropic structure corresponding to a given $k$-shifted Lagrangian structure, but we will now explain how to perform this construction for our local models for Lagrangians which will moreover give strict coisotropic structures.

\begin{ex} 
\label{ln3ex8}
Let $k<0$ with $k\not\equiv 3\mod 4$, and consider data $A^\bu,\om,B^\bu,\al,h$ in Lagrangian Darboux form as in Example \ref{ln3ex1}, so that $A^\bu,\om$ is in Darboux form as in Example \ref{ln2ex2}. Then Example \ref{ln3ex6} defines a strict $k$-shifted Poisson structure $\pi^2_{A^\bu}$ on $A^\bu$. We will define a strict $k$-shifted coisotropic structure on $\al:A^\bu\ra B^\bu$. We need to construct a $(k-1)$-shifted Poisson structure $\pi_{B^\bu}$ on $B^\bu$ and provide a $\bP_{k+1}$-morphism $\ti\al: A^\bu\ra \bigl(\widehat{\Sym}(\bT_{B^\bu}[-k]), \d + [\pi^2_{B^\bu}, -]\bigr)$.

The morphism $\chi: \bT_{B^\bu/A^\bu}\ra \bL_{B^\bu}[k-1]$ is given by
\begin{align*}
&\frac{\pd}{\pd\ti x^i_j}\longmapsto 0,\qquad \frac{\pd}{\pd u^i_j}\longmapsto \dd v^{k-1-i}_j,\qquad \frac{\pd}{\pd v^{k-1-i}_j}\longmapsto (-1)^{k(i+1)}\dd u^i_j,\\
&\frac{\pd}{\pd x^i_j}\longmapsto (-1)^i \dd \frac{\pd \Psi}{\pd \ti x^i_j},\qquad \frac{\pd}{\pd y^{k-i}_j}\longmapsto (-1)^{ik+k+i} \dd\ti x^i_j.
\end{align*}
We can find its one-sided inverse $\chi^{-1}$ so that the composite 
\begin{equation*}
\xymatrix@C=30pt{ \bL_{B^\bu} \ar[r]^(0.4){\chi^{-1}} & \bT_{B^\bu/A^\bu}[1-k] \ar[r] & \bT_{B^\bu}[1-k] }
\end{equation*}
is given by
\begin{equation*}
\dd v^{k-1-i}_j\longmapsto (-1)^{k+1} \frac{\pd}{\pd u^i_j},\quad \dd u^i_j\longmapsto (-1)^{ik+1}\frac{\pd}{\pd v^{k-1-i}_j},\quad \dd \ti x^i_j\longmapsto 0.
\end{equation*}
This gives a degree $1-k$ bivector
\e
\pi^2_{B^\bu} = \sum_{i=0}^e \sum_{j=1}^{n_i} \frac{\pd}{\pd u^i_j}\,\frac{\pd}{\pd v^{k-1-i}_j}.
\label{ln3eq43}
\e
Clearly, $[\pi^2_{B^\bu},\pi^2_{B^\bu}]=0$ as the bivector has constant coefficients.

As a graded commutative algebra, $A^*$ is freely generated over $A_+^*$ by the variables $y^{k-i}_j$ for $i=0,-1,\ldots,d$ and $j=1,\ldots,m_i$. We define the morphism $\ti\al$ to be $\al$ on $A_+^*$ and
\e
\ti\al(y^{k-i}_j) = (-1)^{i+1}\frac{\pd \Psi}{\pd\ti x^i_j} + (-1)^{i+1}\frac{\pd}{\pd \ti x^i_j}.
\label{ln3eq44}
\e

Let us check that this $\ti\al$ is compatible with the differential and the brackets:
\begin{equation*}
\d\circ \ti\al(y^{k-i}_j) = (-1)^{i+1}\biggl(\d\frac{\pd \Psi}{\pd\ti x^i_j} + \d \frac{\pd}{\pd \ti x^i_j} + \sum_{i'=0}^e \sum_{j'=1}^{n_{i'}} \biggl[\frac{\pd}{\pd u^{i'}_{j'}}\,\frac{\pd}{\pd v^{k-1-i'}_{j'}}, \frac{\pd \Psi}{\pd\ti x^i_j}\biggr]\biggr).
\end{equation*}
Equation \eq{ln3eq11} implies that
\begin{equation*}
(-1)^{i+1}\d\frac{\pd \Psi}{\pd\ti x^i_j} = \al(\d y^{k-i}_j) = \ti\al(\d y^{k-i}_j) - \sum_{i'=0}^d\sum_{j'=1}^{m_{i'}}
(-1)^{i'+1}\al_+\left(\frac{\pd \Phi_{j'}^{i'+1}}{\pd x^i_j}\right)\,\frac{\pd}{\pd \ti x_{j'}^{i'}}.
\end{equation*}
Therefore,
\ea
\d& \circ\ti\al(y^{k-i}_j) = \ti\al(\d y^{k-i}_j) - \sum_{i'=0}^d\sum_{j'=1}^{m_{i'}}
(-1)^{i'+1}\al_+\left(\frac{\pd \Phi_{j'}^{i'+1}}{\pd x^i_j}\right)\,\frac{\pd}{\pd \ti x_{j'}^{k-i'}}
\nonumber\\
&+ \sum_{i'=0}^d \sum_{j'=1}^{m_{i'}} (-1)^{i'+1}\al_+\left(\frac{\pd \Phi^{i'+1}_{j'}}{\pd x^i_j}\right) \frac{\pd}{\pd \ti x^{i'}_{j'}} + \sum_{i'=0}^e \sum_{j'=1}^{n_{i'}} (-1)^{k(i'+1)} \frac{\pd^2\Psi}{\pd \ti x^i_j \pd v^{k-1-i'}_{j'}} \frac{\pd}{\pd u^{i'}_{j'}}
\nonumber\\
&+ \sum_{i'=0}^e \sum_{j'=1}^{n_{i'}} \frac{\pd^2\Psi}{\pd \ti x^i_j \pd u^{i'}_{j'}} \frac{\pd}{\pd v^{k-1-i'}_{j'}} + \sum_{i'=0}^e \sum_{j'=1}^{n_{i'}} (-1)^{(k+i)(k+i')+i+1} \frac{\pd^2 \Psi}{\pd v^{k-1-i'}_{j'} \pd \ti x^i_j} \frac{\pd}{\pd u^{i'}_{j'}} 
\nonumber\\
&+ \sum_{i'=0}^e \sum_{j'=1}^{n_{i'}} (-1)^{ii' + 1}\frac{\pd^2\Psi}{\pd u^{i'}_{j'} \pd \ti x^i_j} \frac{\pd}{\pd v^{k-1-i'}_{j'}}= \ti\al(\d y^{k-i}_j),
\label{ln3eq45}
\ea
which shows that $\ti\al$ commutes with the differential $\d$.

The Poisson bracket on $A^\bu$ is given by
\e
\bigl\{x^i_j, x^{i'}_{j'}\bigr\} = 0,\quad \bigl\{y^i_j, y^{i'}_{j'}\bigr\} = 0,\quad 
\bigl\{y^i_j, x^{i'}_{j'}\bigr\} = (-1)^{i+1}\de_{jj'}\de^{ii'}.
\label{ln3eq46}
\e
Under $\ti\al$ these are sent to
\ea
\bigl[\ti\al(x^i_j),\ti\al(x^{i'}_{j'})\bigr] &= \bigl[\ti x^i_j, \ti x^{i'}_{j'}\bigr] = 0,
\nonumber\\
\begin{split}
\bigl[\ti\al(y^i_j),\ti\al(y^{i'}_{j'})\bigr] &= (-1)^{i+i'}\biggl[\frac{\pd \Psi}{\pd \ti x^i_j} + \frac{\pd}{\pd \ti x^i_j}, \frac{\pd \Psi}{\pd \ti x^{i'}_{j'}} + \frac{\pd}{\pd \ti x^{i'}_{j'}}\biggr] \\
&= (-1)^{i+i'}\frac{\pd^2\Psi}{\pd\ti x^i_j \pd\ti x^{i'}_{j'}} + (-1)^{(i+1)(i'+1)} \frac{\pd^2\Psi}{\pd\ti x^{i'}_{j'} \pd\ti x^i_j} = 0,
\end{split}
\label{ln3eq47}\\
\bigl[\ti\al(y^i_j), \ti\al(x^{i'}_{j'})\bigr] &= (-1)^{i+1}\biggl[\frac{\pd \Psi}{\pd \ti x^i_j} + \frac{\pd}{\pd \ti x^i_j}, \ti x^{i'}_{j'}\biggr]= (-1)^{i+1}\de_{jj'}\de^{ii'} = \ti\al\bigl(\{y^i_j, x^{i'}_{j'}\}\bigr).
\nonumber
\ea
Therefore, $\ti\al$ is a morphism of $\bP_{k+1}$-algebras, and so \eq{ln3eq43} and \eq{ln3eq44} define a (strict) coisotropic structure on $\al: A^\bu\ra B^\bu$.

The same formulae also work when $k=0$, defining $A^\bu,\om$ in 0-shifted Darboux form as in Example \ref{ln3ex4}, with $\Phi^{i+1}_j=0$,  and $A^\bu,\om,B^\bu,\al,h$ in Lagrangian Darboux form as in Example~\ref{ln3ex5}.
\end{ex}

\begin{ex} 
\label{ln3ex9}
Now let $k<0$ with $k\equiv 3\mod 4$, and consider data $A^\bu,\ab\om,\ab B^\bu,\ab\al,\ab h$ in weak Lagrangian Darboux form as in Example \ref{ln3ex2}, so that $A^\bu,\om$ is in Darboux form as in Example \ref{ln2ex2}, and we have coordinates $\ti x^i_j,u^i_j,v^{k-1-i}_j,w^e_j$ in $B^\bu$, with the $w^e_j$ associated to invertible $q_1,\ldots,q_{n_e}\in B^0$, where $A^\bu,\om,B^\bu,\al,h$ are in strong Lagrangian Darboux form if~$q_1=\cdots=q_{n_e}=1$.

All of Example \ref{ln3ex8} generalizes to this case, so we just give the definitions, leaving most verifications to the reader. As in \eq{ln3eq42}--\eq{ln3eq43}, the bivector $\pi^2_{B^\bu}$ is
\begin{equation*}
\pi^2_{B^\bu} = \sum_{i=0}^d \sum_{j=1}^{n_i} \frac{\pd}{\pd u^i_j}\,\frac{\pd}{\pd v^{k-1-i}_j} + \sum_{j=1}^{n_e} \frac{1}{4q_j} \frac{\pd}{\pd w_j^e} \frac{\pd}{\pd w_j^e} - \sum_{j=1}^{n_e} \sum_{j'=1}^{n_0} \frac{\pd q_j}{\pd u^0_{j'}}\frac{w_j^e}{2q_j} \frac{\pd}{\pd w_j^e} \frac{\pd}{\pd v_{j'}^{k-1}}.
\end{equation*}
As this does not have constant coefficients, we check that $[\pi^2_{B^\bu},\pi^2_{B^\bu}]=0$ by
\begin{align*}
[\pi^2_{B^\bu}, \pi^2_{B^\bu}] &= \sum_{j=1}^{n_0}\sum_{j'=1}^{n_e}\biggl[\frac{\pd q_{j'}}{\pd u_j^0}\frac{1}{4q_{j'}^2}\frac{\pd}{\pd v_j^{k-1}} \frac{\pd}{\pd w_{j'}^e} \frac{\pd}{\pd w_{j'}^e} - \frac{\pd q_{j'}}{\pd u_j^0}\frac{1}{4q_{j'}^2}\frac{\pd}{\pd v_j^{k-1}} \frac{\pd}{\pd w_{j'}^e} \frac{\pd}{\pd w_{j'}^e}\biggr] \\
&\qquad- \sum_{j,j'=1}^{n_0}\sum_{j''=1}^{n_e} \frac{\pd^2 q_{j''}}{\pd u_j^0 \pd u_{j'}^0}\frac{w_{j''}^e}{2q_{j''}}\frac{\pd}{\pd w_{j''}^e} \frac{\pd}{\pd v_j^{k-1}} \frac{\pd}{\pd v_{j'}^{k-1}} = 0,
\end{align*}
which vanishes as the $\smash{\frac{\pd}{\pd w_j^e}}$ are symmetric under multiplication, as $e$ is odd, and the $\frac{\pd}{\pd v_j^{k-1}}$ are antisymmetric, as $k-1$ is even, and $\frac{\pd^2 q_{j''}}{\pd u_j^0\pd u_{j'}^0}$ is symmetric in~$j,j'$.

The $\bP_{k+1}$-morphism $\ti\al: A^\bu\ra \bigl(\widehat{\Sym}(\bT_{B^\bu}[-k]),\d+[\pi^2_{B^\bu},-]\bigr)$ is given by $\ti\al\vert_{A_+^*}=\al\vert_{A_+^*}$ and, generalizing \eq{ln4eq43},
\begin{align*}
\ti\al(y^{k-i}_j)&=\al(y^{k-i}_j) + (-1)^{i+1}\frac{\pd}{\pd \ti x^i_j},\quad i=-1,\ldots,d,\;\> j=1,\ldots,m_i, \\
\ti\al(y^k_j)&=\al(y^k_j)-\frac{\pd}{\pd \ti x^0_j} + \sum_{j'=1}^{n_e} \frac{1}{2q_{j'}} \frac{\pd q_{j'}}{\pd \ti x^0_j} w_{j'}^e \frac{\pd}{\pd w_{j'}^e},\quad j=1,\ldots,m_0.
\end{align*}
We can show that $\d\ci\ti\al=\ti\al\ci\d$ as in \eq{ln3eq45}, and that $\ti\al$ preserves $\{\,,\,\},[\,,\,]$ as in \eq{ln3eq46}--\eq{ln3eq47}. Thus $\pi^2_{B^\bu},\ti\al$ define a (strict) coisotropic structure on~$\al:A^\bu\!\ra\! B^\bu$.

\end{ex}

Examples \ref{ln3ex8} and \ref{ln3ex9} show that all of our (ordinary/weak/strong) `Lagrangian Darboux form' local models in $k$-shifted symplectic derived $\K$-schemes for $k\le 0$ can be promoted to explicit (strict) coisotropic structures in $k$-shifted Poisson derived $\K$-schemes.

Theorems \ref{ln3thm2} and \ref{ln3thm7} imply a `Coisotropic Neighbourhood Theorem', saying that a nondegenerate coisotropic $\bs f:\bs C\ra\bX$ in a nondegenerate $k$-shifted Poisson derived $\K$-scheme $\bX$ for $k<0$ is Zariski or \'etale locally modelled on $\bSpec\al:\bSpec B^\bu\ra \bSpec A^\bu$ in $\bSpec A^\bu$ in Examples \ref{ln3ex8} or \ref{ln3ex9}, in a similar way to Theorem~\ref{ln3thm6}.

\section{Proofs of the main results}
\label{ln4}

Sections \ref{ln41} and \ref{ln42}--\ref{ln47} will prove Theorems \ref{ln3thm1} and~\ref{ln3thm2}, respectively.

\subsection{Proof of Theorem \ref{ln3thm1}}
\label{ln41}

The proof is modelled on that of Theorem \ref{ln2thm1} in \cite[\S 4.1]{BBJ}. As in Theorem \ref{ln3thm1}, let $\bs f:\bY\ra\bX$ be a morphism of derived $\K$-schemes, $A^\bu$ be a standard form cdga over $\K$, $\bs i:\bSpec A^\bu\hookra\bX$ be a Zariski open inclusion (or \'etale morphism), and $y\in\bY$, $p\in\Spec H^0(A^\bu)$ with $\bs f(y)=\bs i(p)=x\in\bX$.

First, consider a homotopy pullback diagram in $\dSch_\K$:
\begin{equation*}
\xymatrix@C=90pt@R=15pt{
*+[r]{\bZ} \ar[r]_{\bs h} \ar[d]^{\bs g} & *+[l]{\bSpec A^\bu} \ar[d]_{\bs i} \\
*+[r]{\bY} \ar[r]^{\bs f} & *+[l]{\bX,\!} }
\end{equation*}
where $\bZ=\bY\t_\bX\bSpec A^\bu$. The map $\bs g:\bZ\ra\bY$ is a Zariski open immersion (or \'etale if $\bs i$ is \'etale). Also there is a unique point $z\in\bZ$ with $\bs g(z)=y$ and $\bs h(z)=p$. Let $\bs k:\bSpec C^\bu\hookra\bZ$ be an affine Zariski neighbourhood of $z$ for some finitely presented cdga $C^\bu$, so that $r\in\bSpec C^\bu$ with~$\bs k(r)=z$.

Recall the distinction between the ordinary category $\cdga_\K$ and the $\iy$-category $\cdga_\K^\iy$ of cdgas over $\K$, discussed in Remark \ref{ln2rem1}. The morphism $\bs h\ci\bs k:\bSpec C^\bu\ra\bSpec A^\bu$ is equivalent to $\bSpec\ga^\iy$ for some morphism $\ga^\iy:A^\bu\ra C^\bu$ in $\cdga_\K^\iy$, unique up to equivalence. Later (after modifying $C^\bu$) we will show that $\ga^\iy$ descends to a morphism $\ga$ in~$\cdga_\K$.

Possibly after localizing $C^\bu$, we will inductively construct a standard form cdga $B^\bu$ with a quasi-isomorphism $\be:B^\bu\,{\buildrel\sim\over\longra}\,C^\bu$ in $\cdga_\K$, such that $\bs h\ci\bs k\ci(\bSpec\be)^{-1}:\bSpec B^\bu\ra\bSpec A^\bu$ is equivalent to $\bSpec\al$, for $\al:A^\bu\ra B^\bu$ in $\cdga_\K$ a submersion of cdgas minimal at $q=\bSpec\be(r)$, with $(\bSpec\be)^{-1}$ a quasi-inverse for $\bSpec\be$. Then setting $\bs j=\bs g\ci\bs k\ci(\bSpec\be)^{-1}$, so that $\bSpec\al(q)=p$, $\bs j(q)=y$, we have a homotopy commutative diagram
\e
\begin{gathered}
\xymatrix@C=33pt@R=15pt{
*+[r]{\bSpec B^\bu} \ar@<1ex>[drr]^(0.75){(\bSpec\be)^{-1}} \ar@/_1.6pc/[ddrrr]_(0.4){\bs j} \ar@/^.8pc/[drrrrr]^(0.6){\bSpec\al} \\
&& \bSpec C^\bu \ar@/^1.7ex/[rrr]^(0.25){\bSpec\ga} \ar@<0ex>[ull]^{\bSpec\be} \ar[r]_(0.55){\bs k} & \bZ \ar[rr]_(0.4){\bs h} \ar[d]^{\bs g} && *+[l]{\bSpec A^\bu} \ar[d]_{\bs i} \\
&&& \bY \ar[rr]^(0.4){\bs f} && *+[l]{\bX,\!} }
\end{gathered}
\label{ln4eq1}
\e
which gives \eq{ln3eq1}. Also $\bs j$ is a Zariski open inclusion (or \'etale if $\bs i$ is \'etale), as $\bs g,\bs k,(\bSpec\be)^{-1}$ are, so this will prove Theorem~\ref{ln3thm1}.

As $\bSpec C^\bu$ is affine, we can choose an embedding $\Spec H^0(C^\bu)\hookra\bA^N$ for $N\gg 0$. We also have a composition of morphisms 
\begin{equation*}
\xymatrix@C=30pt{ \Spec H^0(C^\bu) \ar[r]^(0.6){t_0(\bs k)} & t_0(\bZ) \ar[r]^(0.4){t_0(\bs h)}  & \Spec H^0(A^\bu)\,\, \ar@{^{(}->}[r] & \Spec A^0=:U, }
\end{equation*}
so the direct product is an embedding $e:\Spec H^0(C^\bu)\hookra\bA^N\t U$. Choose a smooth, affine, locally closed $\K$-subscheme $V$ in $\bA^N\t U$, such that $V$ contains an open neighbourhood of $e(r)$ in $e[\Spec H^0(C^\bu)]$ as a closed $\K$-subscheme, and the projection $V\ra U$ is smooth, and $\dim V$ is minimal under these conditions. Localizing $V$ if necessary, we can assume $T^*V$ is a trivial vector bundle.

As $V$ is affine we have $V=\Spec B^0$ for a smooth $\K$-algebra $B^0$, with $\Om^1_{B^0}$ a free $B^0$-module. The (smooth) projection $V\ra U$ is $\Spec\al^0$ for $\al^0:B^0\ra A^0$ a smooth morphism of $\K$-algebras. Since $V$ contains an open neighbourhood of $e(p)$ in $e[\Spec H^0(C^\bu)]$, localizing $C^\bu$ if necessary we can suppose $e[\Spec H^0(C^\bu)]$ is a closed $\K$-subscheme in $V$. Then the closed embedding $\Spec H^0(C^\bu)\hookra V$ is $\Spec\be'$ for some~$\be':B^0\ra H^0(C^\bu)$. 

Since $C^\bu$ is the homotopy limit of its Postnikov tower $\cdots\ra\tau_{\geq -1}C^\bu\ra \tau_{\geq 0}C^\bu\simeq H^0(C^\bu)$ in which each map is a square-zero extension of cdgas \cite[Prop.~7.1.3.19]{Luri}, and as $B^0$ is smooth and hence maps out of it can be lifted along square-zero extensions, after replacing $C^\bu$ by an equivalent cdga we can lift $\be':B^0\ra H^0(C^\bu)$ along the canonical map $C^\bu\ra H^0(C^\bu)$ to obtain a map~$\be^0:B^0\ra C^0\subseteq C^\bu$.

Set $\ga^0=\be^0\ci\al^0:A^0\ra C^\bu$, as a morphism in $\cdga_\K$. Then we have a homotopy commutative diagram in $\cdga_\K^\iy$, with $\al^0,\be^0,\ga^0$ morphisms in $\cdga_\K$:
\e
\begin{gathered}
\xymatrix@C=90pt@R=15pt{
*+[r]{A^0_{\phantom{k}}} \ar[r]_{\al^0} \ar@/^.8pc/[rr]^{\ga^0} \ar@{^{(}->}[d] & B^0_{\phantom{k}} \ar[r]_{\be^0}  & *+[l]{C^0_{\phantom{k}}} \ar@{^{(}->}[d] \\
*+[r]{A^\bu} \ar[rr]^{\ga^\iy} && *+[l]{C^\bu.\!} }
\end{gathered}
\label{ln4eq2}
\e
Since $A^\bu$ is cofibrant over $A^0$, and $A^0\ra C^\bu$ is represented by a morphism $\ga^0$ in $\cdga_\K$, up to equivalence $\ga^\iy$ descends to a morphism $\ga:A^\bu\ra C^\bu$ in~$\cdga_\K$.

As $A^\bu$ is a standard form cdga, it is freely generated over $A^0$ by finitely many generators $x^i_1,\ldots,x_{m_i}^i$ in degree $i$ for $i=-1,-2,\ldots,$ where $m_i=0$ for $i\ll 0$. Write $A^\bu(k)$ for $k=0,-1,\ldots$ for the sub-cdga of $A^\bu$ generated over $A^0$ by the generators $x^i_j$ in degrees $i=-1,-2,\ldots,k$ only, so that $A^\bu(0)\subseteq A^\bu(-1)\subseteq A^\bu(-2)\subseteq\cdots,$ and $A^\bu(0)=A^0$, and $A^\bu(k)=A^\bu$ for $k\ll 0$.

Next we inductively construct a sequence of standard form cdgas $B^\bu(0)\subseteq B^\bu(-1)\subseteq B^\bu(-2)\subseteq\cdots,$ and submersions $\al(k):A^\bu(k)\ra B^\bu(k)$ in $\cdga_\K$, and morphisms $\be(k):B^\bu(k)\ra C^\bu$ in $\cdga_\K$, such that $\al(k-1)\vert_{A^\bu(k)}=\al(k)$, $\be(k-1)\vert_{B^\bu(k)}=\be(k)$, and the following diagram commutes in $\cdga_\K$:
\e
\begin{gathered}
\xymatrix@C=90pt@R=15pt{
*+[r]{A^\bu(k)\,\,} \ar@{^{(}->}[r] \ar[d]^{\al(k)} & *+[r]{A^\bu(k-1)\,\,} \ar@{^{(}->}[r] \ar[d]^{\al(k-1)} & *+[l]{A^\bu} \ar[d]_{\ga} \\
*+[r]{B^\bu(k)\,\,} \ar@{^{(}->}[r] \ar@/_2ex/[rr]_{\be(k)} & *+[r]{B^\bu(k-1)} \ar[r]^{\be(k-1)} & *+[l]{C^\bu,\!} }
\end{gathered}
\label{ln4eq3}
\e
and $\bL_{C^\bu/B^\bu(k)}$ is concentrated in degrees $(-\iy,k-1]$, and $B^*(k-1)$ is freely generated over $B^*(k)$ by finitely many generators in degree $k-1$, where this number of generators is minimal such that the previous conditions hold near~$r$.

For the first step, set $A^\bu(0)=A^0$ and $B^\bu(0)=B^0$, regarded as cdgas concentrated in degree 0, and $\al(0)=\al^0:A^\bu(0)\ra B^\bu(0)$, which is a submersion, and $\be(0)=\be^0:B^\bu(0)\ra C^\bu$. Then \eq{ln4eq2} implies that the outer rectangle of \eq{ln4eq3} commutes for~$k=0$.

For the inductive step, suppose that for some $k\le 0$ we have chosen $B^\bu(0),\ab B^\bu(-1),\ab\ldots,\ab B^\bu(k)$ and $\al(0),\al(-1),\ldots,\al(k)$ and $\be(0),\be(-1),\ldots,\be(k)$ with the desired properties. Now $H^{k-1}(\bL_{C^\bu/B^\bu(k)}\vert_r)$ is spanned by elements $(\dd y,\dd z)$ for $y\in C^{k-1}$ and $z\in B^k(k)$ with $\d y=\be(k)(z)\in C^k$ and $\d z=0\in B^{k+1}(k)$. We have generators $x_j^{k-1}$ of $A^\bu$ in degree $k-1$ with $\d x_j^{k-1}\in A^\bu(k)\subseteq A^\bu$, so that $y=\ga(x_j^{k-1})$, $z=\al(k)(\d x_j^{k-1})$ satisfy $\d y=\be(k)(z)$ and $\d z=0$, and $(\dd\ga(x_j^{k-1}),\dd\al(k)(\d x_j^{k-1}))$ gives an element of $H^{k-1}(\bL_{C^\bu/B^\bu(k)}\vert_r)$. Choose a minimal number of additional pairs $(y_1^{k-1},z_1^k),\ldots,(y_{n_{k-1}}^{k-1},z_{n_{k-1}}^k)$ in $C^{k-1}\t B^k(k)$ with $\d y^{k-1}_j=\be(k)(z_j^k)$ and $\d z_j^k=0$, such that $H^{k-1}(\bL_{C^\bu/B^\bu(k)}\vert_r)$ is spanned by $(\dd\ga(x_j^{k-1}),\dd\al(k)(\d x_j^{k-1}))$ for $j=1,\ldots,m_{k-1}$ and $(\dd y_j^{k-1},\ab\dd z_j^k)$ for~$j=1,\ldots,n_{k-1}$.

Define $B^*(k-1)$ to be the commutative graded algebra freely generated over $B^*(k)$ by generators $\ti x_1^{k-1},\ab\ldots,\ab\ti x_{m_{k-1}}^{k-1},\ti y_1^{k-1},\ab\ldots,\ti y_{n_{k-1}}^{k-1}$ in degree $k-1$. Define the differential $\d$ in $B^\bu(k-1)=(B^*(k-1),\d)$ by $\d\vert_{B^*(k)}=\d_{B^\bu(k)}$, and $\d(\ti x_j^{k-1})=\al(k)(\d x_j^{k-1})$, $\d\ti y_j^{k-1}=z_j^k$ for all $j$. Then $\d\ci\d=0$ as $\d\ci\al(k)(\d x_j^{k-1})=0$ and $\d z_j^k=0$ in $B^\bu(k)$. Define $\al(k-1):A^*(k-1)\ra B^*(k-1)$ to be the unique graded algebra morphism with $\al(k-1)\vert_{A^*(k-1)}=\al(k)$ and $\al(k-1)(x_j^{k-1})=\ti x_j^{k-1}$ for $j=1,\ldots,m_{k-1}$. Then $\al(k-1)$ is a cdga morphism as $\d\al(k-1)(x_j^{k-1})=\d\ti x_j^{k-1}=\al(k)(\d x_j^{k-1})=\al(k-1)(\d x_j^{k-1})$.
 
Define $\be(k-1):B^*(k-1)\ra C^*$ to be the unique graded algebra morphism with $\be(k-1)\vert_{B^*(k)}=\be(k)$ and $\be(k-1)(\ti x_j^{k-1})=\ga(x_j^{k-1})$ and $\be(k-1)(\ti y_j^{k-1})=y_j^{k-1}$ for all $j$. Then $\be(k-1)$ is a cdga morphism $B^\bu(k-1)\ra C^\bu$ as
\begin{align*}
\d\be(k-1)(\ti x_j^{k-1})&=\d\ga(x_j^{k-1})=\ga(\d x_j^{k-1})=\be(k)\ci\al(k)(\d x_j^{k-1})\\
&=\be(k-1)\ci\d(\ti x_j^{k-1}),\\
\d\be(k-1)(\ti y_j^{k-1})&=\d y_j^{k-1}=\be(k)(z_j^k)=\be(k-1)\ci\d(\ti y_j^{k-1}).
\end{align*}
Also \eq{ln4eq3} commutes as~$\be(k-1)\ci\al(k-1)(x_j^{k-1})=\ga(x_j^{k-1})$.

As $\bL_{C^\bu/B^\bu(k)}$ is concentrated in degrees $(-\iy,k-1]$, and the new generators of $B^\bu(k-1)$ span $H^{k-1}(\bL_{C^\bu/B^\bu(k)}\vert_r)$, we see that $\bL_{C^\bu/B^\bu(k-1)}$ is concentrated in degrees $(-\iy,k-2]$ near $r\in\bSpec C^\bu$. Thus, localizing $C^\bu$ and $B^\bu(0),\ldots,B^\bu(k-1)$, we can suppose that $\bL_{C^\bu/B^\bu(k-1)}$ is concentrated in degrees $(-\iy,k-2]$. This completes the inductive step. Hence, by induction we have defined standard form cdgas $B^\bu(0)\subseteq B^\bu(-1)\subseteq B^\bu(-2)\subseteq\cdots,$ submersions $\al(k):A^\bu(k)\ra B^\bu(k)$, and morphisms $\be(k):B^\bu(k)\ra C^\bu$ in $\cdga_\K$ for~$k=0,-1,-2,\ldots.$ 

Since $A^\bu$ is of standard form and $C^\bu$ is finitely presented, for $k\ll 0$ $A^\bu$ has no generators in degrees $<k$, and $\bL_{C^\bu}$ is concentrated in degrees $[k,0]$. Then we add no further generators in degrees $<k$, so that $A^\bu(k)=A^\bu(k-1)=\cdots=A^\bu$ and $B^\bu(k)=B^\bu(k-1)=\cdots.$ Set $B^\bu=B^\bu(k)$, $\al=\al(k)$ and $\be=\be(k)$ for such $k$. Then $\be$ is a quasi-isomorphism, since $\be(i)$ is a quasi-isomorphism in degrees $>i$ and $\be=\be(k)=\be(k-1)=\cdots$, and $\ga=\be\ci\al$, so \eq{ln4eq1} commutes up to homotopy. Also $\al:A^\bu\ra B^\bu$ is a submersion, as $\al(k):A^\bu(k)\ra B^\bu(k)$ is, and the minimality conditions we imposed on $\dim V$ and the number of $(y_j^{k-1},z_j^k)$ imply that $\al$ is minimal at $q$. This completes the proof of Theorem~\ref{ln3thm1}.

\subsection{\texorpdfstring{Beginning the proof of Theorem \ref{ln3thm2} for $k<0$}{Beginning the proof of Theorem \ref{ln3thm2} for k<0}}
\label{ln42}

Sections \ref{ln42}--\ref{ln47} will prove our main result, Theorem \ref{ln3thm2}. This section begins by showing that when $k<0$ we can choose $B^\bu,q,\al,\bs i,\bs j$ and $h=(h^0,0,0,\ldots)$ as in the first part of Theorem \ref{ln3thm2}, and $\Xi\in B^k$, $\psi\in(\Om^1_{B^\bu})^{k-1}$ with $\d\Xi=-\al(\Phi+\Phi_+)$, $\dd\Xi+\d\psi=-\al_*(\phi+\phi_+)$ and~$h^0=\frac{1}{k-1}\dd\psi$. 

Sections \ref{ln43}--\ref{ln45} continue the proof for even $k<0$. In \S\ref{ln43} we choose coordinates $\ti x^i_j,u^i_j,v^{k-1-i}_j$ on $B^\bu$ with $h^0=\sum_{i,j}\dd a^{k-1-i}_j\,\dd\ti x^i_j+\dd u^i_j\,\dd v^{k-1-i}_j$ for functions $a^{k-1-i}_j\in B^{k-1-i}$. Section \ref{ln44} defines $\Psi=\Xi-\sum_{i,j}ia_j^{k-1-i}\d\ti x^i_j$, and computes the p.d.e.\ satisfied by $\Psi$, and expressions for $\d$ in $B^\bu=(B^*,\d)$, and for~$\al:A^\bu\ra B^\bu$. 

Section \ref{ln45} explains how to replace $\al,h$ by an equivalent morphism $\hat\al:A^\bu\ra B^\bu$ and Lagrangian structure $\hat h=(\hat h^0,0,0,\ldots)$, such that for $\hat\al,\hat h$ the functions $a^{k-1-i}_j$ are zero, giving $\hat h^0=\dd u^i_j\,\dd v^{k-1-i}_j$. This will complete the proof of Theorem \ref{ln3thm2} for even $k<0$. Sections \ref{ln46}--\ref{ln47} discuss how to modify \S\ref{ln43}--\S\ref{ln45} to prove Theorem \ref{ln3thm2} for the cases $k<0$ with $k\equiv 1\mod 4$, and $k<0$ with $k\equiv 3\mod 4$, respectively.

Let $(\bX,\om_\bX)$ be a $k$-shifted symplectic derived $\K$-scheme for $k<0$, and $\bs f:\bs L\ra\bX$ be a Lagrangian derived $\K$-scheme in $(\bX,\om_\bX),$ with isotropic structure $h_{\bs L}:0\,{\buildrel\sim\over\longra}\,\bs f^*(\om_\bX)$. Let $y\in\bs L$ with $\bs f(y)=x\in\bX$. Suppose $A^\bu$ is a standard form cdga over $\K$, $\om$ is a $k$-shifted symplectic form on $\bSpec A^\bu$ with $A^\bu,\om$ in Darboux form, $p\in\Spec H^0(A^\bu)$, and $\bs i:\bSpec A^\bu\ra\bX$ is either a Zariski open inclusion or \'etale, with $\bs i(p)=x$ and~$\om\sim\bs i^*(\om_\bX)$. 

As in Example \ref{ln2ex2}, we have coordinates $x^i_j,y^{k-i}_j$ in $A^\bu$ for $i=0,-1,\ldots,d$ and $j=1,\ldots,m_d$, and $\om=(\om^0,0,0,\ldots)$ with $\om^0=\sum_{i=0}^d\sum_{j=1}^{m_i}\dd x^i_j\,\dd y^{k-i}_j$ in $(\La^2\Om^1_{A^\bu})^k$, and $\Phi\in A^{k+1}$, $\phi\in(\Om^1_{A^\bu})^k$ with $\d\Phi=0$, $\dd\Phi+\d\phi=0$ and $\dd\phi=k\om^0$. We also use the notation of Remark \ref{ln2rem3}, which defined a sub-cdga $\io:A^\bu_+\hookra A^\bu$ containing the $x^i_j$ but not the $y^{k-i}_j$, and $\Phi_+\in A^{k+1}_+$, $\phi_+\in(\Om^1_{A^\bu})^k$ satisfying $\d\Phi_+=0$, $\dd\Phi_++\d\phi_+=0$ and~$\dd\phi_+=-\om^0$.

Form a homotopy commutative diagram
\e
\begin{gathered}
\xymatrix@C=90pt@R=15pt{ *+[r]{\bZ} \ar[d]^{\bs g} \ar[r]_{\bs h} &  \bSpec A^\bu\, \ar[d]_{\bs i} \ar[r]_{\bSpec\io} & *+[l]{\bSpec A_+^\bu} \\ *+[r]{\bs L} \ar[r]^{\bs f} & \bX, }
\end{gathered}
\label{ln4eq4}
\e
with $\bZ=\bs L\t_{\bs f,\bX,\bs i}\bSpec A^\bu$. Then $\bs i$ a Zariski open inclusion, or \'etale, implies that $\bs g$ is a Zariski open inclusion, or \'etale, respectively. There is a unique point $z\in\bZ$ with $\bs g(z)=y\in\bs L$ and~$\bs h(z)=p\in\bSpec A^\bu$. 

Apply Theorem \ref{ln3thm1} with $\bSpec\io\ci\bs h:\bZ\ra\bSpec A_+^\bu$, $\bs\id:\bSpec A_+^\bu\ra \bSpec A_+^\bu$, $z\in\bZ$, $p\in\bSpec A_+^\bu$ in place of $\bs f:\bY\ra\bX$, $\bs i:\bSpec A^\bu\ra\bX$, $y\in\bY$, $p\in\bSpec A^\bu$. This gives a standard form cdga $B^\bu$, a point $q\in\bSpec B^\bu$, a submersion $\al_+:A_+^\bu\ra B^\bu$ minimal at $q$ with $\bSpec\al_+(q)=p$, and a Zariski open inclusion $\bs e:\bSpec B^\bu\ra\bZ$ with $\bs e(q)=z$ such that the following diagram is homotopy commutative
\e
\begin{gathered}
\xymatrix@C=100pt@R=15pt{ *+[r]{\bSpec B^\bu} \ar[d]^{\bs e} \ar[r]_{\bSpec\al_+} &  *+[l]{\bSpec A_+^\bu} \ar[d]_{\bs\id}  \\ *+[r]{\bZ} \ar[r]^{\bSpec\io\ci\bs h} & *+[l]{\bSpec A_+^\bu.\!} }
\end{gathered}
\label{ln4eq5}
\e

The morphism $\bs h\ci\bs e:\bSpec B^\bu\ra\bSpec A^\bu$ is equivalent to $\bSpec\al^\iy$ for some morphism $\al^\iy:A^\bu\ra B^\bu$ in $\cdga_\K^\iy$, where $\al^\iy\ci\io\simeq\al_+$. Since $A^\bu$ is freely generated over $A^\bu_+$ and so cofibrant over $A_+^\bu$, up to equivalence $\al^\iy$ descends to a morphism $\al:A^\bu\ra B^\bu$ in $\cdga_\K$ with $\al\ci\io=\al_+$. Thus, combining \eq{ln4eq4}--\eq{ln4eq5} gives a homotopy commutative diagram
\e
\begin{gathered}
\xymatrix@C=90pt@R=15pt{ *+[r]{\bSpec B^\bu} \ar[d]^{\bs e} \ar@/^.8pc/[drr]^(0.7){\bSpec \al_+ } \ar@/^.5pc/[dr]^(0.7){\bSpec\al } \ar@<-.1pc>@/_.5pc/[dd]_{\bs j}
\\
*+[r]{\bZ} \ar[d]^{\bs g} \ar[r]_{\bs h} &  \bSpec A^\bu\, \ar[d]_{\bs i} \ar[r]_{\bSpec\io} & *+[l]{\bSpec A_+^\bu} \\ *+[r]{\bs L} \ar[r]^{\bs f} & \bX, }
\end{gathered}
\label{ln4eq6}
\e
where we write $\bs j=\bs g\ci\bs e$, so that $\bs j$ is a Zariski open inclusion, or \'etale, if $\bs i$ is a Zariski open inclusion, or \'etale, respectively. This proves the first part of Theorem \ref{ln3thm2}: we have constructed a standard form cdga $B^\bu$, a point $q\in\bSpec B^\bu$, a morphism $\al:A^\bu\ra B^\bu$ in $\cdga_\K$ with $\bSpec\al(q)=p$ such that $\al_+=\al\ci\io:A^\bu_+\ra B^\bu$ is a submersion minimal at $q$, and a morphism $\bs j:\bSpec B^\bu\hookra\bs L$ which is either a Zariski open inclusion or \'etale with $\bs j(q)=y$, such that \eq{ln3eq38} homotopy commutes by~\eq{ln4eq6}.

Next we discuss the isotropic structure. As we have the homotopy commutative diagram \eq{ln3eq38} with $\bs i,\bs j$ \'etale, and $\om\sim\bs i^*(\om_\bX)$, Definition \ref{ln2def6} implies that the Lagrangian structure $h_{\bs L}$ for $\bs f:\bs L\ra(\bX,\om_\bX)$ lifts to a Lagrangian structure $h$ for $\bSpec\al:\bSpec B^\bu\ra(\bSpec A^\bu,\om)$, where $h=(h^0,h^1,h^2,\ldots)$ with $h^i\in(\La^{2+i}\bL_{B^\bu})^{k-1-i}$, which by \eq{ln2eq4} satisfy
\e
\d h^0=\al_*(\om^0)\quad\text{and}\quad \dd h^i+\d h^{i+1}=0\quad\text{for $i=0,1,\ldots.$}
\label{ln4eq7}
\e

First, we need a vanishing result for the isotropic structure. Note that `Lagrangian Darboux form' in Example \ref{ln3ex1} involves $\Psi\in B^k$ and $\psi\in(\Om^1_{B^\bu})^{k-1}$, and \eq{ln4eq8} is \eq{ln3eq22}--\eq{ln3eq23} with $\Xi$ in place of $\Psi$, and $\ti h^0=\frac{1}{k-1}\dd\psi$ is~\eq{ln3eq24}.

\begin{prop} 
\label{ln4prop1}
There exist exist\/ $\Xi\in B^k$ and\/ $\psi\in(\Om^1_{B^\bu})^{k-1}$ satisfying
\e
\d\Xi=-\al(\Phi+\Phi_+)\quad\text{and}\quad \dd\Xi+\d\psi=-\al_*(\phi+\phi_+),
\label{ln4eq8}
\e
such that the isotropic structure $h=(h^0,h^1,\ldots)$ is homotopic to $(\ti h^0,0,0,\ldots)$ where $\ti h^0=\frac{1}{k-1}\dd\psi$.
\end{prop}

\begin{proof} Combining equations \eq{ln2eq10}, \eq{ln2eq16} and \eq{ln4eq7} gives the equations
\begin{align*}
\d\bigl[-\al(\Phi)-\al(\Phi_+)\bigr]&=0\quad\text{in $B^{k+2}$,}\\
\dd\bigl[-\al(\Phi)\!-\!\al(\Phi_+)\bigr]\!+\!\d\bigl[-\al_*(\phi)\!-\!\al_*(\phi_+)\bigr]&=0\quad\text{in $(\Om^1_{B^\bu})^{k+1}$,}\\
\dd\bigl[-\al_*(\phi)-\al_*(\phi_+)\bigr]+\d\bigl[(k-1)h^0\bigr]&=0\quad\text{in $(\La^2\Om^1_{B^\bu})^k$,}\\
\dd\bigl[(k-1)h^i\bigr]+\d\bigl[(k-1)h^{i+1}\bigr]&=0\quad\text{in $(\La^{3+i}\Om^1_{B^\bu})^{k-1-i}$, $i\ge 0.$}
\end{align*}
Therefore
\e
\ga=\bigl(-\al(\Phi)-\al(\Phi_+),-\al_*(\phi)-\al_*(\phi_+),(k-1)h^0,(k-1)h^1,\ldots\bigr)
\label{ln4eq9}
\e
is a $(k+1)$-shifted closed 0-form on $\bSpec B^\bu$, that is, $\ga$ is a closed element of degree $k+1$ in the complex $\bigl(\prod_{i\ge 0}(\La^i\bL_{B^\bu})[i],\dd+\d\bigr)$.

As in Bussi, Brav and Joyce \cite[\S 5.2]{BBJ} we have isomorphisms
\e
\begin{split}
H^{k+1}\bigl(\ts\prod_{i\ge 0}(\La^i\bL_{B^\bu})[i]&,\dd+\d\bigr)\cong\HN^{k+1}(B^\bu)(0)\cong \HP^{k+1}(B^\bu)(0)\\
&\cong \HP^{k+1}(H^0(B^\bu))(0)\cong {\rm H}^{k+1}_{\rm inf}(H^0(B^\bu)),
\end{split}
\label{ln4eq10}
\e
where $\HN^*(-)$ is negative cyclic homology, $\HP^*(-)$ is periodic cyclic homology, and ${\rm H}^*_{\rm inf}(-)$ is algebraic de Rham cohomology. 

If $k<-1$ then ${\rm H}^{k+1}_{\rm inf}(H^0(B^\bu))=0$. If $k=-1$ then ${\rm H}^0_{\rm inf}(H^0(B^\bu))$ is isomorphic to the locally constant functions $(\Spec H^0(B^\bu))^\red\ra\bA^1$. This identifies the cohomology class of $\ga$ with $-\al_*(\Phi\vert_{(\Spec H^0(A^\bu))^\red})$. But as in Example \ref{ln2ex2}, when $k=-1$ we impose the additional condition that $\Phi\vert_{(\Spec H^0(A^\bu))^\red}=0$. Hence $\ga$ is exact in $\bigl(\ts\prod_{i\ge 0}(\La^i\bL_{B^\bu})[i],\dd+\d\bigr)$ for all $k<0$. So we may write
\e
\ga=(\dd+\d)\bigl(\Xi,\psi,(k-1)\de^0,(k-1)\de^1,(k-1)\de^2,\ldots\bigr),
\label{ln4eq11}
\e
where $\Xi\in B^k$, $\psi\in(\Om^1_{B^\bu})^{k-1}$ and $\de^i\in(\La^{2+i}\Om^1_{B^\bu})^{k-2-i}$ for $i=0,1,\ldots.$ Set $\ti h^0=\frac{1}{k-1}\dd\psi$. Then combining \eq{ln4eq9}--\eq{ln4eq11} proves \eq{ln4eq8} and the equations
\begin{equation*}
\d\de^0=h^0-\ti h^0,\quad \dd\de^i+\d\de^{i+1}=h^{i+1},\quad i=0,1,\ldots.
\end{equation*}
Thus $(\de^0,\de^1,\ldots)$ is a homotopy from $(\ti h^0,0,0,\ldots)$ to $h=(h^0,h^1,\ldots)$ in the complex $\bigl(\ts\prod_{i\ge 0}(\La^i\bL_{B^\bu})[i],\dd+\d\bigr)$. This completes the proof. 
\end{proof}

We replace $h=(h^0,h^1,\ldots)$ by $(\ti h^0,0,0,\ldots)$, so from now on we suppose that $h=(h^0,0,0,\ldots)$ with $h^0=\frac{1}{k-1}\dd\psi$. We continue the proof in the cases $k$ even, $k\equiv 1\mod 4$, and $k\equiv 3\mod 4$, in \S\ref{ln43}, \S\ref{ln46} and \S\ref{ln47}, respectively.

\subsection{\texorpdfstring{Choosing coordinates $\ti x^i_j,u^i_j,v^{k-1-i}_j$ on $B^\bu$ for even $k<0$}{Choosing coordinates x,u,v on B for even k<0}}
\label{ln43}

We carry on from \S\ref{ln42}, but now also assume that $k$ is even. In the notation of Example \ref{ln3ex1} we have $d=[(k+1)/2]$, $e=[k/2]$, so that $e=d$ and $k=2d=2e$. Supposing $k$ is even simplifies the proof as there is no `middle degree' $(k-1)/2$ in $B^\bu$, which would require special treatment. 

Sections \ref{ln43}--\ref{ln45} will complete the proof of Theorem \ref{ln3thm2} for even $k<0$. To save work in \S\ref{ln46}--\S\ref{ln47}, in \S\ref{ln43}--\S\ref{ln45} we give the correct signs in formulae for general $k$, so we include factors such as $(-1)^{(i+1)k}$ in \eq{ln4eq27} although $k$ is even.

From \S\ref{ln42}, we have a submersion of standard form cdgas $\al_+:A_+^\bu\ra B^\bu$ minimal at $q\in\bSpec B^\bu$, and coordinates $x^i_j\in A_+^i$ for $i=0,-1,\ldots,d$ and $j=1,\ldots,m_i$, where $(x^0_1,\ldots,x^0_{m_0})$ are \'etale coordinates on $U=\Spec A^0$, and $A^*_+$ is freely generated over $A^0$ by the $x^i_j$ in degree $i$ for $i<0$. As $\al_+^0:A^0\ra B^0$ is smooth, localizing $B^\bu$ if necessary, we may assume there exist $u^0_1,\ldots,u^0_{n_0}\in B^0$ such that $\dd\ti x^0_1,\ldots,\dd\ti x^0_{m_0},\dd u^0_1,\ldots,\dd u^0_{n_0}$ form a basis of $\Om^1_{B^0}$ over $B^0$, where we write $\ti x^0_j=\al_+^0(x^0_j)\in B^0$. Geometrically, $(x_1^0,\ldots,x_{m_0}^0)$ are \'etale coordinates on $U:=\Spec A^0$, and $(\ti x_1^0,\ldots,\ti x_{m_0}^0,u_1^0,\ldots,u_{n_0}^0)$ are \'etale coordinates on~$V:=\Spec B^0$.

Since $\al_+:A_+^\bu\ra B^\bu$ is a submersion of standard form cdgas and $A_+^*$ is freely generated over $A^0$ by the $x^i_j$ for $i=-1,-2,\ldots,d$ and $j=1,\ldots,m_i$, we see that $B^*$ is freely generated over $B^0$ by the $\ti x^i_j=\al_+(x^i_j)$ for $i=-1,-2,\ldots,d$ and $j=1,\ldots,m_i$ plus some additional variables, which lie in degrees $-1,-2,\ldots,k-1$ since $\bL_{B^\bu}$ is concentrated in degrees $[k-1,0]$ and $\al_+$ is minimal at $q$. Thus, as in \eq{ln3eq2}, we take $B^*$ to be freely generated over $B^0$ by
\e
\begin{aligned}
&\ti x_1^i,\ldots,\ti x^i_{m_i} &&\text{in degree $i$ for
$i=-1,-2,\ldots,d$, and} \\
&u_1^i,\ldots,u^i_{n_i} &&\text{in degree $i$ for
$i=-1,-2,\ldots,e$, and} \\
&v_1^{k-1-i},\ldots,v^{k-1-i}_{n_i'} &&\text{in degree $k-1-i$ for
$i=0,-1,\ldots,e$,}
\end{aligned}
\label{ln4eq12}
\e
where $\ti x^i_j=\al_+(x^i_j)$, and later we will show that $n_i'=n_i$. Note that by \eq{ln2eq8} and \eq{ln2eq14} this implies that in $B^\bu=(B^*,\d)$ we have
\e
\d\ti x^i_j=\al_+(\d x^i_j)=(-1)^{(i+1)(k+1)}\al\biggl(\frac{\pd\Phi}{\pd y^{k-i}_j}\biggr)=(-1)^{i+1}\al_+\bigl(\Phi_j^{i+1}\bigr).
\label{ln4eq13}
\e

Since $\psi\in(\bL_{B^\bu})^{k-1}$, we may write
\e
\begin{split}
\frac{1}{k-1}\psi=\ts\sum_{i=0}^d\sum_{j=1}^{m_i}a^{k-1-i}_j\,\dd\ti x^i_j&\ts+\sum_{i=0}^e\sum_{j=1}^{n_i}b^{k-1-i}_j\,\dd u^i_j\\
&\ts+\sum_{i=0}^e\sum_{j=1}^{n_i'}c^i_j\,\dd v^{k-1-i}_j,
\end{split}
\label{ln4eq14}
\e
where $a^l_j,b^l_j,c^l_j\in B^l$. For degree reasons, the $c^i_j$ depend on $B^0$ and the $\ti x^{i'}_{j'},u^{i'}_{j'}$, but do not involve the $v^{k-1-i'}_{j'}$. By leaving $h^0$ unchanged but replacing $\Xi,\psi$ by
\e
\begin{split}
\ti\Xi&=\Xi-(k-1)\d\bigl[\ts\sum_{i=0}^e\sum_{j=1}^{n_i'}(-1)^ic^i_jv^{k-1-i}_j\bigr],\\\ti\psi&=\psi-(k-1)\dd\bigl[\ts\sum_{i=0}^e\sum_{j=1}^{n_i'}(-1)^ic^i_jv^{k-1-i}_j\bigr],
\end{split}
\label{ln4eq15}
\e
we may assume that $c^i_j=0$ for all $i,j$, as $\dd c^i_j$ involves no terms in $\dd v^{k-1-i'}_{j'}$. Thus we have 
\e
\frac{1}{k-1}\psi=\ts\sum_{i=0}^d\sum_{j=1}^{m_i}a^{k-1-i}_j\,\dd\ti x^i_j+\sum_{i=0}^e\sum_{j=1}^{n_i}b^{k-1-i}_j\,\dd u^i_j,
\label{ln4eq16}
\e
so that $h^0=\frac{1}{k-1}\dd\psi$ yields
\e
h^0=\ts\sum_{i=0}^d\sum_{j=1}^{m_i}\dd a^{k-1-i}_j\,\dd\ti x^i_j+\sum_{i=0}^e\sum_{j=1}^{n_i}\dd b^{k-1-i}_j\,\dd u^i_j.
\label{ln4eq17}
\e

The next part of the argument follows \eq{ln3eq14}--\eq{ln3eq20} in Example \ref{ln3ex1}. 
By Definition \ref{ln2def6}, $h$ being a nondegenerate isotropic structure means that a certain morphism $\chi:\bT_{B^\bu/A^\bu}\ra\Om^1_{B^\bu}[k-1]$ is a quasi-isomorphism of $B^\bu$-modules, so $\chi\vert_q:\bT_{B^\bu/A^\bu}\vert_q\ra\Om^1_{B^\bu}[k-1]\vert_q$ is also a quasi-isomorphism of complexes of $\K$-vector spaces. As for \eq{ln3eq15}--\eq{ln3eq18}, as $\K$-vector spaces, the $i^{\rm th}$ graded pieces of $\Om^1_{A^\bu}\vert_p$, $\Om^1_{B^\bu}\vert_q$, $(\Om^1_{A^\bu})^\vee\vert_p$, and $(\Om^1_{A^\bu})^\vee\vert_p$ are
\ea
\begin{split}
\bigl(\Om^1_{A^\bu}\vert_p\bigr){}^i=\bigl\langle &\dd x^i_j,\; j=1,\ldots,m_i,\;\dd y^i_j,\; j=1,\ldots,m_{k-i}\bigr\rangle{}_\K,
\end{split}
\label{ln4eq18}\\
\begin{split}
\bigl(\Om^1_{B^\bu}\vert_q\bigr){}^i=\bigl\langle &\dd\ti x^i_j,\; j=1,\ldots,m_i,\;
\dd u^i_j,\; j=1,\ldots,n_i,\\ 
&\qquad \dd v^i_j,\; j=1,\ldots,n'_{k-1-i}\bigr\rangle{}_\K,
\end{split}
\label{ln4eq19}\\
\begin{split}
\bigl((\Om^1_{A^\bu})^\vee\vert_p\bigr){}^i=\bigl\langle &\ts\frac{\pd}{\pd x^{-i}_j},\; j=1,\ldots,m_{-i},\;\frac{\pd}{\pd y^{-i}_j},\; j=1,\ldots,m_{i-k}\bigr\rangle{}_\K,
\end{split}
\label{ln4eq20}\\
\begin{split}
\bigl((\Om^1_{A^\bu})^\vee\vert_p\bigr){}^i=\bigl\langle &\ts\frac{\pd}{\pd\ti x^{-i}_j},\; j=1,\ldots,m_{-i},\; \frac{\pd}{\pd u^{-i}_j},\; j=1,\ldots,n_{-i},\\ 
&\qquad \ts\frac{\pd}{\pd v^{-i}_j},\; j=1,\ldots,n'_{i+1-k}\bigr\rangle{}_\K.
\end{split}
\label{ln4eq21}
\ea

As for \eq{ln3eq19}, the next diagram shows $\chi\vert_q:\bT_{B^\bu/A^\bu}\vert_q\ra\Om^1_{B^\bu}[k-1]\vert_q$ in degrees $i,i+1$, together with $\d$ in both complexes:
\ea
\nonumber\\[-77pt]
\begin{gathered}
\xymatrix@C=90pt@R=55pt{
*+[r]{\raisebox{-100pt}{$\ts\begin{subarray}{l} \ts (\bT_{B^\bu/A^\bu}\vert_q)^i=\\
\ts \bigl\langle \frac{\pd}{\pd\ti x^{-i}_j},\;\forall j\bigr\rangle{}_\K\op
\\
\ts\bigl\langle\frac{\pd}{\pd u^{-i}_j},\frac{\pd}{\pd v^{-i}_j},\;\forall j\bigr\rangle{}_\K\\
\ts \op\bigl\langle \frac{\pd}{\pd x^{1-i}_j},\;\forall j\bigr\rangle{}_\K\\
\ts \op\bigl\langle\frac{\pd}{\pd y^{1-i}_j},\;\forall j\bigr\rangle{}_\K
\end{subarray}$}}
\ar@<-30pt>[r]_(0.4){\chi\vert_q^i=\begin{pmatrix} \st h^0\cdot & \st h^0\cdot & \st * & \st \om^0\cdot \\
\st h^0\cdot & \st h^0\cdot & \st * & \st 0  \end{pmatrix}} 
\ar[d]^{\d=\begin{pmatrix} \st * & \st 0 & \st 0 & \st 0 \\
\st * & \st 0 & \st 0 & \st 0 \\
\st \al^* & \st 0 & \st * & \st 0 \\
\st * & \st  * & \st * & \st * \end{pmatrix}}
 & *+[l]{\raisebox{-40pt}{$\ts\begin{subarray}{l} \ts (\Om^1_{B^\bu}[k-1]\vert_q)^i=\\
\ts \bigl\langle \dd\ti x^{k-1+i}_j,\;\forall j\bigr\rangle{}_\K\op
\\
\ts\bigl\langle\dd u^{k-1+i}_j,\dd v^{k-1+i}_j,\;\forall j\bigr\rangle{}_\K
\end{subarray}$}} 
\ar@<-40pt>[d]_{\d=\begin{pmatrix} \st * & \st 0  \\
\st * & \st 0  \end{pmatrix}} \\
*+[r]{\raisebox{90pt}{$\ts\begin{subarray}{l} \ts (\bT_{B^\bu/A^\bu}\vert_q)^{i+1}=\\
\ts \bigl\langle \frac{\pd}{\pd\ti x^{-i-1}_j},\;\forall j\bigr\rangle{}_\K\op
\\
\ts\bigl\langle\frac{\pd}{\pd u^{-i-1}_j},\frac{\pd}{\pd v^{-i-1}_j},\;\forall j\bigr\rangle{}_\K\\
\ts \op\bigl\langle \frac{\pd}{\pd x^{-i}_j},\;\forall j\bigr\rangle{}_\K\\
\ts \op\bigl\langle\frac{\pd}{\pd y^{-i}_j},\;\forall j\bigr\rangle{}_\K
\end{subarray}$}} 
\ar@<25pt>[r]^(0.45){\chi\vert_q^{i+1}=\begin{pmatrix} \st h^0\cdot & \st h^0\cdot & \st * & \st \om^0\cdot \\
\st h^0\cdot & \st h^0\cdot & \st * & \st 0  \end{pmatrix}}  & *+[l]{\raisebox{30pt}{$\ts\begin{subarray}{l} \ts (\Om^1_{B^\bu}[k-1]\vert_q)^{i+1}=\\
\ts \bigl\langle \dd\ti x^{k+i}_j,\;\forall j\bigr\rangle{}_\K\op
\\
\ts\bigl\langle\dd u^{k+i}_j,\dd v^{k+i}_j,\;\forall j\bigr\rangle{}_\K.
\end{subarray}$}} }\!\!\!\!\!\!\!\!\!\!\!\!\!\!\!\!\!\!\!\!\!\!\!{}
\end{gathered}
\label{ln4eq22}
\\[-70pt]
\nonumber
\ea

Here in the left hand column of \eq{ln4eq22}, the component of `$\d$' mapping $\bigl\langle\frac{\pd}{\pd u^{-i}_j},\frac{\pd}{\pd v^{-i}_j}\bigr\rangle{}_\K\ra \bigl\langle\frac{\pd}{\pd u^{-i-1}_j},\frac{\pd}{\pd v^{-i-1}_j}\bigr\rangle{}_\K$ is zero. This is because $\al_+:A_+^\bu\ra B^\bu$ is a submersion minimal at $q$, so in the additional variables $u^i_j,v^{k-1-i}_j$ in $B^\bu$ the differential vanishes at $q$. Similarly, in the right hand column, the component of `$\d$' mapping $\bigl\langle\dd u^{k-1+i}_j,\dd v^{k-1+i}_j\bigr\rangle{}_\K\ra \bigl\langle\dd u^{k+i}_j,\dd v^{k+i}_j,\;\forall j\bigr\rangle{}_\K$ is zero.

The argument in Example \ref{ln3ex1} involving the subcomplex $C^\bu$ works again in this case. Thus in \eq{ln4eq22} we may replace the right hand column by the subcomplex $C^\bu$, giving a simpler diagram in which the rows are a quasi-isomorphism:
\ea
\begin{gathered}
\xymatrix@C=135pt@R=30pt{
*+[r]{\begin{subarray}{l} \ts\bigl\langle\frac{\pd}{\pd u^{-i}_j},\frac{\pd}{\pd v^{-i}_j},\;\forall j\bigr\rangle{}_\K\\
\ts \op\bigl\langle\frac{\pd}{\pd y^{1-i}_j},\;\forall j\bigr\rangle{}_\K
\end{subarray}}
\ar[r]_(0.35){\chi\vert_q^i=\begin{pmatrix} \st h^0\cdot & \st \om^0\cdot \\
\st h^0\cdot & \st 0  \end{pmatrix}} 
\ar[d]^{\d=\begin{pmatrix} \st 0 & \st 0 \\ \st  * &  \st * \end{pmatrix}}
 & *+[l]{\begin{subarray}{l} 
\ts \bigl\langle \dd\ti x^{k-1+i}_j,\;\forall j\bigr\rangle{}_\K\op \\
\ts\bigl\langle\dd u^{k-1+i}_j,\dd v^{k-1+i}_j,\;\forall j\bigr\rangle{}_\K
\end{subarray}} 
\ar[d]_{\d=\begin{pmatrix} \st * & \st 0  \\
\st * & \st 0  \end{pmatrix}} \\
*+[r]{\begin{subarray}{l} 
\ts\bigl\langle\frac{\pd}{\pd u^{-i-1}_j},\frac{\pd}{\pd v^{-i-1}_j},\;\forall j\bigr\rangle{}_\K\\
\ts \op\bigl\langle\frac{\pd}{\pd y^{-i}_j},\;\forall j\bigr\rangle{}_\K
\end{subarray}} 
\ar[r]^(0.46){\chi\vert_q^{i+1}=\begin{pmatrix} \st h^0\cdot & \st \om^0\cdot \\
\st h^0\cdot & \st 0  \end{pmatrix}}  & *+[l]{\begin{subarray}{l} 
\ts \bigl\langle \dd\ti x^{k+i}_j,\;\forall j\bigr\rangle{}_\K\op
\\
\ts\bigl\langle\dd u^{k+i}_j,\dd v^{k+i}_j,\;\forall j\bigr\rangle{}_\K.
\end{subarray}} }\!\!\!\!\!\!{}
\end{gathered}
\label{ln4eq23}
\ea
In the top row of \eq{ln4eq23}, the morphism $\om^0\cdot:\bigl\langle\frac{\pd}{\pd y^{1-i}_j}\bigr\rangle{}_\K\ra \bigl\langle \dd\ti x^{k-1+i}_j\bigr\rangle{}_\K$ is an isomorphism. Using this, we see that $\chi\vert_q$ is a quasi-isomorphism if and only if the morphism $h^0\cdot:\bigl\langle\frac{\pd}{\pd u^{-i}_j},\frac{\pd}{\pd v^{-i}_j}\bigr\rangle{}_\K\ra \bigl\langle\dd u^{k-1+i}_j,\dd v^{k-1+i}_j\bigr\rangle{}_\K$ in $\chi\vert_q^i$ is an isomorphism for all $i\in\Z$. 

From \eq{ln4eq17}, $h^0\cdot:\bigl\langle\frac{\pd}{\pd u^{-i}_j},\frac{\pd}{\pd v^{-i}_j}\bigr\rangle{}_\K\ra \bigl\langle\dd u^{k-1+i}_j,\dd v^{k-1+i}_j\bigr\rangle{}_\K$ acts by
\begin{align*}
&h^0\cdot:\frac{\pd}{\pd u^{-i}_j}\longmapsto (-1)^{k(i+1)}\sum_{j'=1}^{n'_{-i}}\frac{\pd b_j^{k-1+i}}{\pd v^{k-1+i}_{j'}}\bigg\vert_q\,\dd v^{k-1+i}_{j'},\\
&h^0\cdot:\frac{\pd}{\pd v^{-i}_j}\longmapsto \sum_{j'=1}^{n_{k-1+i}}\frac{\pd b_{j'}^{-i}}{\pd v^{-i}_j}\bigg\vert_q\,\dd u^{k-1+i}_{j'}.
\end{align*}
Therefore we see that $h$ nondegenerate at $q$ is equivalent to $n_i'=n_i$ and the following being an invertible matrix over $\K$ for all~$i=0,-1,\ldots,e$:
\e
\biggl(\frac{\pd b_j^{k-1-i}}{\pd v^{k-1-i}_{j'}}\bigg\vert_q\biggr)_{j,j'=1}^{n_i}.
\label{ln4eq24}
\e

Set $\hat v_j^{k-1-i}=(-1)^{(i+1)k}b^{k-i-i}_j$ for all $i=0,-1,\ldots,e$ and $j=1,\ldots,n_i$. The invertibility of \eq{ln4eq24} implies that near $q\in V$ we can invert this to write the $v^{k-1-i'}_{j'}$ as linear functions of the $\hat v_j^{k-1-i}$ with coefficients in $B^0[\ti x^{i''}_{j''},u^{i''}_{j''}]$. Localizing $B^\bu$ we can suppose this invertibility holds globally on $V=\Spec B^0$, so that $\ti x^i_j,u^i_j,\hat v^i_j$ is an alternative set of coordinates for $B^\bu$. Thus, replacing the $v^i_j$ by the $\hat v^i_j$, by \eq{ln4eq16}--\eq{ln4eq17} we can take $\psi$ and $h^0$ to be of the form
\ea
\psi&=\ts(k-1)\biggl[\sum\limits_{i=0}^d\sum\limits_{j=1}^{m_i}a^{k-1-i}_j\,\dd\ti x^i_j+\sum\limits_{i=0}^e\sum\limits_{j=1}^{n_i}(-1)^{(i+1)k}v^{k-1-i}_j\,\dd u^i_j\biggr],
\label{ln4eq25}\\
h^0&=\ts\sum\limits_{i=0}^d\sum\limits_{j=1}^{m_i}\dd a^{k-1-i}_j\,\dd\ti x^i_j+\sum\limits_{i=0}^e\sum\limits_{j=1}^{n_i}\dd u^i_j\,\dd v^{k-1-i}_j.
\label{ln4eq26}
\ea

Leaving $h^0$ unchanged, but by replacing $\Xi,\psi$ by
\begin{align*}
\ti\Xi&=\Xi+\d\biggl[\ts\sum\limits_{i=0}^d\sum\limits_{j=1}^{m_i}(-1)^{k-i}ia^{k-1-i}_j\ti x^i_j+\sum\limits_{i=0}^e\sum\limits_{j=1}^{n_i}(-1)^iiu^i_jv^{k-1-i}_j\biggr],\\
\ti\psi&=\psi+\dd\biggl[\ts\sum\limits_{i=0}^d\sum\limits_{j=1}^{m_i}(-1)^{k-i}ia^{k-1-i}_j\ti x^i_j+\sum\limits_{i=0}^e\sum\limits_{j=1}^{n_i}(-1)^iiu^i_jv^{k-1-i}_j\biggr],
\end{align*}
we have 
\e
\begin{split}
\psi&=\sum_{i=0}^d\sum_{j=1}^{m_i}\bigl[(k-1-i)a^{k-1-i}_j\,\dd\ti x^i_j+
(-1)^{(i+1)k}i\ti x^i_j\dd a_j^{k-1-i}\bigr]\\
&+\sum_{i=0}^e\sum_{j=1}^{n_i}\bigl[i\,u^i_j\,\dd v^{k-1-i}_j +(-1)^{(i+1)k}(k-1-i)v^{k-1-i}_j\,\dd u^i_j \bigr].
\end{split}
\label{ln4eq27}
\e

\subsection{\texorpdfstring{Determining the equations for even $k<0$}{Determining the equations for even k<0}}
\label{ln44}

We continue in the situation of \S\ref{ln43}. Using \eq{ln4eq13}, define $\Psi\in B^k$ by
\e
\Psi=\Xi-\sum_{i=0}^d\sum_{j=1}^{m_i}a_j^{k-1-i}\d\ti x^i_j=\Xi-\sum_{i=0}^d\sum_{j=1}^{m_i}(-1)^{i+1}a_j^{k-1-i}\al_+\bigl(\Phi_j^{i+1}\bigr).
\label{ln4eq28}
\e
Writing things in terms of $\Psi$ rather than $\Xi$ will make many of the following formulae simpler. This section will compute the p.d.e.\ satisfied by $\Psi$, and expressions for $\d$ in $B^\bu=(B^*,\d)$, and for~$\al:A^\bu\ra B^\bu$. 

By \eq{ln4eq8} we have $\dd\Xi+\d\psi=-\al_*(\phi+\phi_+)$. Expanding this equation using \eq{ln2eq9}, \eq{ln2eq15}, \eq{ln4eq27}, \eq{ln4eq28}, and $\ti x^i_j=\al(x^i_j)$ yields
\begin{small}
\ea
&\sum_{i=0}^d\sum_{j=1}^{m_i}(-1)^{(i+1)k}\frac{\pd\Psi}{\pd\ti x_j^i}\,\dd\ti x_j^i+\sum_{i=0}^e\sum_{j=1}^{n_i}(-1)^{(i+1)k}\biggl[\frac{\pd\Psi}{\pd u_j^i}\,\dd u_j^i+\frac{\pd\Psi}{\pd v_j^{k-1-i}}\,\dd v_j^{k-1-i}\biggr]
\nonumber\\
&+\sum_{i=0}^d\sum_{j=1}^{m_i}\begin{aligned}[t]\bigl[&(k-1-i)\d a^{k-1-i}_j\,\dd\ti x^i_j+(-1)^{k-i}(k-2-i)a^{k-1-i}_j\dd\ci\d\ti x^i_j\\
&+(-1)^{(i+1)k}(i+1)\d\ti x^i_j\dd a_j^{k-1-i}+(-1)^{(i+1)(k+1)}i\ti x^i_j\dd\ci\d a_j^{k-1-i}\bigr)\bigr]
\end{aligned}
\nonumber\\
&+\sum_{i=0}^e\sum_{j=1}^{n_i}\begin{aligned}[t]\bigl[&i\bigl(\d u^i_j\,\dd v^{k-1-i}_j-(-1)^iu^i_j\,\dd\ci\d v^{k-1-i}_j\bigr)\\
&+(-1)^{(i+1)k}(k-1-i)\bigl(\d v^{k-1-i}_j\,\dd u^i_j-(-1)^{k-1-i}v^{k-1-i}_j\,\dd\ci\d u^i_j\bigr)\bigr]\end{aligned}
\nonumber\\
&=-\sum_{i=0}^d\sum_{j=1}^{m_i}\bigl[i\,\ti x^i_j\,\dd\ci \al(y^{k-i}_j)+(-1)^{(i+1)(k+1)}(k-1-i)\al(y^{k-i}_j)\,\dd\ti x^i_j\bigr].
\label{ln4eq29}
\ea
\end{small}

Rewriting terms of the form $\dd\ci\d(\cdots)$ using equations such as
\begin{align*}
&\dd\ci\d u^i_j=\sum_{i'=0}^d\sum_{j'=1}^{m_{i'}}(-1)^{(i+1)(i'+1)}\frac{\pd(\d u^i_j)}{\pd\ti x_{j'}^{i'}}\,\dd \ti x_{j'}^{i'}\\
&+\!\sum_{i'=0}^e\sum_{j'=1}^{n_{i'}}\biggl[(-1)^{(i+1)(i'+1)}\frac{\pd(\d u^i_j)}{\pd u_{j'}^{i'}}\,\dd  u_{j'}^{i'}\!+\!(-1)^{(i+1)(k-i')}\frac{\pd(\d u^i_j)}{\pd v_{j'}^{k-1-i'}}\,\dd  v_{j'}^{k-1-i'}\biggr],
\end{align*}
and noting that $\frac{\pd}{\pd u^i_j}(\d\ti x^{i'}_{j'})=\frac{\pd}{\pd v^{k-1-i}_j}(\d\ti x^{i'}_{j'})=0$ as $\d\ti x^i_j=\al(\d x^i_j)$ for $\d x^i_j\in A_+^\bu$, we can express \eq{ln4eq29} solely in terms of multiples of $\dd \ti x_j^i,\dd u_j^i$ and $\dd v_j^{k-1-i}$. Then taking coefficients of these, multiplying by $(-1)^{(i+1)k}$, and rearranging, gives three equations:
\ea
&(-1)^i(k-1)\al(y^{k-i}_j)-(-1)^{(i+1)k}(k-1)\d a^{k-1-i}_j
\nonumber\\
&+(k-1)\sum_{i'=0}^d\sum_{j'=1}^{m_{i'}}(-1)^{i(k-1-i')}a_{j'}^{k-1-i'}\frac{\pd(\d\ti x^{i'}_{j'})}{\pd\ti x^i_j}
\nonumber\\
&=\frac{\pd}{\pd\ti x_j^i}\biggl[\Psi+\sum_{i'=0}^d\sum_{j'=1}^{m_{i'}}\begin{aligned}[t]& \bigl[(-1)^{i'}i'\,\ti x^{i'}_{j'}\al(y^{k-i'}_{j'})+(i'+1)a_{j'}^{k-1-i'}\d\ti x^{i'}_{j'}\\
&+(-1)^{k-1-i'}i'\d a_{j'}^{k-1-i'}\ti x^{i'}_{j'}\bigr]
\end{aligned}
\nonumber\\
&-\sum_{i'=0}^e\sum_{j'=1}^{n_{i'}}\bigl(i'u^{i'}_{j'}\d v^{k-1-i'}_{j'}+(-1)^{(i'+1)k}(k-1-i')v^{k-1-i'}_{j'}\d u^{i'}_{j'}\bigr)\biggr],
\label{ln4eq30}
\allowdisplaybreaks\\
&-(k-1)\d v^{k-1-i}_j
\nonumber\\
&=\frac{\pd}{\pd u_j^i}\biggl[\Psi+\sum_{i'=0}^d\sum_{j'=1}^{m_{i'}}\begin{aligned}[t]& \bigl[(-1)^{i'}i'\,\ti x^{i'}_{j'}\al(y^{k-i'}_{j'})+(i'+1)a_{j'}^{k-1-i'}\d\ti x^{i'}_{j'}\\
&+(-1)^{k-1-i'}i'\d a_{j'}^{k-1-i'}\ti x^{i'}_{j'}\bigr]
\end{aligned}
\nonumber\\
&-\sum_{i'=0}^e\sum_{j'=1}^{n_{i'}}\bigl(i'u^{i'}_{j'}\d v^{k-1-i'}_{j'}+(-1)^{(i'+1)k}(k-1-i')v^{k-1-i'}_{j'}\d u^{i'}_{j'}\bigr)\biggr],
\label{ln4eq31}
\allowdisplaybreaks\\
&-(-1)^{(i+1)k}(k-1)\d u^i_j
\nonumber\\
&=\frac{\pd}{\pd v_j^{k-1-i}}\biggl[\Psi+\sum_{i'=0}^d\sum_{j'=1}^{m_{i'}}\begin{aligned}[t]& \bigl[(-1)^{i'}i'\,\ti x^{i'}_{j'}\al(y^{k-i'}_{j'})+(i'+1)a_{j'}^{k-1-i'}\d\ti x^{i'}_{j'}\\
&+(-1)^{k-1-i'}i'\d a_{j'}^{k-1-i'}\ti x^{i'}_{j'}\bigr]
\end{aligned}
\nonumber\\
&-\sum_{i'=0}^e\sum_{j'=1}^{n_{i'}}\bigl(i'u^{i'}_{j'}\d v^{k-1-i'}_{j'}+(-1)^{(i'+1)k}(k-1-i')v^{k-1-i'}_{j'}\d u^{i'}_{j'}\bigr)\biggr],
\label{ln4eq32}
\ea
where \eq{ln4eq30} holds for all $i=0,\ldots,d$ and $j=1,\ldots,m_i$, and \eq{ln4eq31}--\eq{ln4eq32} hold for all $i=0,\ldots,e$ and $j=1,\ldots,n_i$.

Writing $F$ for the function $[\cdots]$ on the r.h.s.\ of \eq{ln4eq30}--\eq{ln4eq32}, we have
\begin{equation*}
kF=\sum_{i=0}^d\sum_{j=1}^{m_i}i\ti x^i_j\frac{\pd F}{\pd\ti x^i_j}+\sum_{i=0}^e\sum_{j=1}^{n_i}\biggl[iu^i_j\frac{\pd F}{\pd u^i_j}+(k-1-i)v^{k-1-i}_j\frac{\pd F}{\pd v^{k-1-i}_j}\biggr],
\end{equation*}
since $F$ has degree $k$. Thus, multiplying \eq{ln4eq30} by $i\ti x^i_j$, and \eq{ln4eq31} by $iu^i_j$, and \eq{ln4eq32} by $(k-1-i)v^{k-1-i}_j$, and summing all three over all $i,j$, yields
\ea
&k\biggl[\Psi+\sum_{i=0}^d\sum_{j=1}^{m_i}\begin{aligned}[t]& \bigl[(-1)^ii\,\ti x^i_j\al(y^{k-i}_j)+(i+1)a_j^{k-1-i}\d\ti x^i_j\\
&+(-1)^{k-1-i}i\d(a_j^{k-1-i})\ti x^i_j\bigr]
\end{aligned}
\nonumber\\
&-\sum_{i=0}^e\sum_{j=1}^{n_i}\bigl(iu^i_j\d v^{k-1-i}_j+(-1)^{(i+1)k}(k-1-i)v^{k-1-i}_j\d u^i_j\bigr)\biggr]
\nonumber\\
&=(k-1)\sum_{i=0}^d\sum_{j=1}^{m_i}\begin{aligned}[t]&\bigl[(-1)^ii\ti x^i_j\al(y^{k-i}_j)+(i+1)a^{k-1-i}_j\d\ti x^i_j\\
&+(-1)^{k-1-i}i\d(a^{k-1-i}_j)\ti x^i_j\bigr]
\end{aligned}
\nonumber\\
&-(k-1)\sum_{i=0}^e\sum_{j=1}^{n_i}\bigl[
iu^i_j\d v^{k-1-i}_j+(-1)^{(i+1)k}(k-1-i)v^{k-1-i}_j\d u^i_j\bigr].
\label{ln4eq33}
\ea
Here the second term on the r.h.s.\ of \eq{ln4eq33} comes from the third term on the l.h.s.\ of \eq{ln4eq30} using $\sum_{i,j}i\ti x^i_j\frac{\pd(\d\ti x^{i'}_{j'})}{\pd\ti x^i_j}=(i'+1)\d\ti x^{i'}_{j'}$, as $\frac{\pd(\d\ti x^{i'}_{j'})}{\pd u^i_j}=\frac{\pd(\d\ti x^{i'}_{j'})}{\pd v^{k-1-i}_j}=0$.

Rearranging \eq{ln4eq33} shows that $F=-(k-1)\Psi$. Substituting this into \eq{ln4eq30}--\eq{ln4eq32}, dividing by $1-k$, rearranging, and using \eq{ln4eq13} in \eq{ln4eq34} yields
\ea
\begin{split}
\al(y^{k-i}_j)
&=(-1)^{i+1}\frac{\pd\Psi}{\pd\ti x_j^i}-(-1)^{(i+1)(k+1)}\d a^{k-1-i}_j\\
&\qquad\qquad -\sum_{i'=0}^d\sum_{j'=1}^{m_{i'}}(-1)^{i+(i'+1)k}\al_+\biggl(\frac{\pd\Phi_{j'}^{i'+1}}{\pd x^i_j}\biggr)a_{j'}^{k-1-i'},
\end{split}
\label{ln4eq34}\\
\d v^{k-1-i}_j&=\frac{\pd\Psi}{\pd u_j^i},
\label{ln4eq35}\\
\d u^i_j&=(-1)^{(i+1)k}\frac{\pd\Psi}{\pd v_j^{k-1-i}}.
\label{ln4eq36}
\ea

Using equations \eq{ln2eq14}, \eq{ln4eq13}, \eq{ln4eq28}, \eq{ln4eq35} and \eq{ln4eq36} we see that
\ea
&\d\Xi\!=\!\!\sum_{i=-1}^d\sum_{j=1}^{m_i}\biggl(\d\ti x_j^i\frac{\pd\Psi}{\pd\ti x_j^i}\!+\!\d a_j^{k-1-i}\d\ti x^i_j\biggr)\!+\!\!\sum_{i=-1}^e\sum_{j=1}^{n_i}\biggl(\d u_j^i\frac{\pd\Psi}{\pd u_j^i}\!+\!\d v_j^{k-1-i}\frac{\pd\Psi}{\pd v_j^{k-1-i}}\biggr)
\nonumber\\
&=\sum_{i=-1}^d\sum_{j=1}^{m_i}\al_+\bigl(\Phi_j^{i+1}\bigr)\biggl((-1)^{i+1}\frac{\pd\Psi}{\pd\ti x_j^i}+(-1)^{(i+1)(k+1)}\d a_j^{k-1-i}\biggr)
\label{ln4eq37}\\
&\qquad +2\sum_{i=-1}^e\sum_{j=1}^{n_i}\frac{\pd\Psi}{\pd u_j^i}\frac{\pd\Psi}{\pd v_j^{k-1-i}}.
\nonumber
\ea
Also we have
\ea
&\al(\Phi+\Phi_+)=2\al_+(\Phi_+)+\sum_{i=-1}^d\sum_{j=1}^{m_i}\al_+\bigl(\Phi_j^{i+1}\bigr)\al\bigl(y_j^{k-i}\bigr)
\label{ln4eq38}\\
&=2\al_+(\Phi_+)+
\sum_{i=-1}^d\sum_{j=1}^{m_i}\al_+\bigl(\Phi_j^{i+1}\bigr)\biggl[(-1)^{i+1}\frac{\pd\Psi}{\pd\ti x_j^i}-(-1)^{(i+1)(k+1)}\d a^{k-1-i}_j\biggr]
\nonumber
\ea
Here in the first step we use \eq{ln2eq11}, and in the second we use \eq{ln4eq34} to substitute for $\al(y^{k-i}_j)$, and note that the terms coming from the second line of \eq{ln4eq34} are zero by \eq{ln2eq13}. Since $\d\Xi=-\al(\Phi+\Phi_+)$ by \eq{ln4eq8}, equations \eq{ln4eq37}--\eq{ln4eq38} give
\e
\sum_{i=-1}^e\sum_{j=1}^{n_i}\frac{\pd\Psi}{\pd u_j^i}\frac{\pd\Psi}{\pd v_j^{k-1-i}}+\al_+(\Phi_+)+\sum_{i=-1}^d\sum_{j=1}^{m_i}(-1)^{i+1}\al_+\bigl(\Phi_j^{i+1}\bigr)\,\frac{\pd\Psi}{\pd\ti x_j^i}=0,
\label{ln4eq39}
\e
which is equation \eq{ln3eq4} in Example~\ref{ln3ex1}.

\subsection[\texorpdfstring{Improving $\al:A^\bu\ra B^\bu$ and completing the proof for even $k<0$}{Improving a : A --> B and completing the proof for even k<0}]{Improving $\al$ and completing the proof for even $k<0$}
\label{ln45}

We continue in the situation of \S\ref{ln43}--\S\ref{ln44}. Observe that the expressions \eq{ln4eq13}, \eq{ln4eq35}, \eq{ln4eq36} for $\d$ in $B^\bu=(B^*,\d)$, and the p.d.e.\ \eq{ln4eq39}, all depend on the coordinates $\ti x^i_j,u^i_j,v^{k-1-i}_j$ and functions $\Psi,\al_+(\Phi_+),\al_+(\Phi_j^{i+1})$ in $B^\bu$, but are independent of the $a^{k-i-1}_j$, although the expressions \eq{ln4eq34} for $\al:A^\bu\ra B^\bu$, and \eq{ln4eq26}, \eq{ln4eq27} for $h^0\in(\La^2\Om^1_{B^\bu})^{k-1}$ and $\psi\in(\Om^1_{B^\bu})^{k-1}$ do depend on the~$a^{k-1-i}_j$.

Now $\al:A^\bu\ra B^\bu$ was chosen in \S\ref{ln42} to make \eq{ln4eq6} homotopy commute, and thus $\al$ is unique only up to homotopy in $\cdga^\iy_\K$, subject to the condition $\al\ci\io=\al_+$. We will now show that keeping $B^\bu,\al_+,\ti x^i_j,u^i_j,v^{k-1-i}_j,\Psi$ fixed, we can replace $\al$ by a homotopic morphism $\hat\al:A^\bu\ra B^\bu$ with $\hat\al\ci\io=\al_+$, and $h^0,\psi$ by homotopic $\hat h^0,\hat\psi$, so as to make the $a^{k-1-i}_j$ zero.

Define a morphism of graded $\K$-algebras $\hat\al:A^*\ra B^*$ by $\hat\al\vert_{A^*_+}=\al_+$ and 
\e
\hat\al(y^{k-i}_j)=(-1)^{i+1}\frac{\pd\Psi}{\pd\ti x_j^i}
\label{ln4eq40}
\e
for $i=0,-1,\ldots,d$ and $j=1,\ldots,m_i$, as in \eq{ln3eq5}, which would be \eq{ln4eq34} if we had $a^{k-1-i}_j=0$ for all $i,j$. This is well-defined as $A^*$ is freely generated over $A^*_+$ by the $y^{k-i}_j$. The proof in equation \eq{ln3eq11} in Example \ref{ln3ex1} shows that $\hat\al:A^\bu\ra B^\bu$ is a cdga morphism. Define $\hat h^0\in(\La^2\Om^1_{B^\bu})^{k-1}$, $\hat\psi\in(\Om^1_{B^\bu})^{k-1}$ by
\ea
\hat h^0&=\ts\sum\limits_{i=0}^e\sum\limits_{j=1}^{n_i}\dd u^i_j\,\dd v^{k-1-i}_j,
\label{ln4eq41}\\
\hat\psi&=\sum_{i=0}^e\sum_{j=1}^{n_i}\bigl[i\,u^i_j\,\dd v^{k-1-i}_j +(-1)^{(i+1)k}(k-1-i)v^{k-1-i}_j\,\dd u^i_j\bigr],
\label{ln4eq42}
\ea
as in \eq{ln3eq12} and \eq{ln3eq21}, which would be $h^0,\psi$ in \eq{ln4eq26}--\eq{ln4eq27} if we had $a^{k-1-i}_j=0$ for all $i,j$. The proofs of \eq{ln3eq22}--\eq{ln3eq24} in Example \ref{ln3ex1} show that
\begin{equation*}
\d\Psi=-\hat\al(\Phi+\Phi_+),\;\> \dd\Psi+\d\hat\psi=-\hat\al_*(\phi+\phi_+),\;\> (k-1)\hat h^0=\dd\hat\psi,
\end{equation*}
as in Proposition \ref{ln4prop1} with $\Psi,\hat\psi,\hat h^0$ in place of $\Xi,\psi,\ti h^0$. Also Example \ref{ln3ex1} shows that $\hat h^0$ is a Lagrangian structure for $\hat\al:A^\bu\ra B^\bu$ and $\om$.

\begin{prop} 
\label{ln4prop2}
In the situation above, $\al,\hat\al:A^\bu\ra B^\bu$ are homotopic cdga morphisms, and under this homotopy, the Lagrangian structures $(h^0,0,\ldots)$ for $\al:A^\bu\ra B^\bu,$ $\om$ and\/ $(\hat h^0,0,\ldots)$ for $\hat\al:A^\bu\ra B^\bu,$ $\om$ are also homotopic.
\end{prop}

\begin{proof} By definition, a homotopy $H$ from $\hat\al$ to $\al$ is a cdga morphism
\begin{equation*}
H:A^\bu\longra B^\bu\ot_\K\K[s,t],
\end{equation*}
where $\K[s,t]$ is the cdga over $\K$ in {\it nonnegative\/} degrees which as a graded $\K$-algebra is freely generated by variables $s$ in degree 0 and $t$ in degree 1 with differential given by $\d s=t$, $\d t=0$, such that the following commutes:
\e
\begin{gathered}
\xymatrix@C=100pt@R=15pt{ & A^\bu \ar@/_.5pc/[dl]_{\hat\al} \ar[d]^(0.4)H \ar@/^.5pc/[dr]^\al \\
*+[r]{B^\bu} & B^\bu\ot_\K\K[s,t] \ar[l]_(0.4){s=0,\;\; t=0} \ar[r]^(0.4){s=1,\;\; t=0} & *+[l]{B^\bu,} }
\end{gathered}
\label{ln4eq43}
\e
where the bottom morphisms are evaluation at $s=t=0$ and at $s=1$, $t=0$. These are 1-simplices in the simplicial model category $\cdga_\K$, as explained in Remark~\ref{ln2rem1}.

Note that as in \S\ref{ln21}, all cdgas $C^\bu$ considered so far have been concentrated in nonpositive degrees, so that $C^i=0$ for $i>0$. Here, however, $\K[s,t]$ lives in degrees $0,1$ (since $t^2=0$), and $B^\bu\ot_\K\K[s,t]$ in degrees $1,0,-1,-2,\ldots.$ So $\bSpec\K[s,t],\bSpec\bigl(B^\bu\ot_\K\K[s,t]\bigr)$ {\it do not exist\/} as derived schemes in the usual sense. Nonetheless, this is a good definition of homotopy of cdga morphisms.

Heuristically, we can pretend ``$\bSpec\K[s,t]\cong[0,1]$'' is an interval, so that ``$\bSpec\bigl(B^\bu\ot_\K\K[s,t]\bigr)\cong(\bSpec B^\bu)\t[0,1]$'', and \eq{ln4eq43} corresponds to a diagram
\e
\begin{gathered}
\xymatrix@C=70pt@R=15pt{ & \bSpec A^\bu   \\
*+[r]{\bSpec B^\bu} \ar[r]^(0.5){``\id\t 0\text{''}} \ar@/^.5pc/[ur]^{\bSpec\hat\al} & ``(\bSpec B^\bu)\t[0,1]\text{''} \ar[u]_(0.5){``\bSpec H\text{''}} & *+[l]{\bSpec B^\bu,}  \ar@/_.5pc/[ul]_{\bSpec\al} \ar[l]_(0.5){``\id\t 1\text{''}} }
\end{gathered}
\label{ln4eq44}
\e
so that ``$\bSpec H$'' is a homotopy from $\bSpec\hat\al$ to $\bSpec\al$ in the usual sense in topology. But we do not claim that \eq{ln4eq44} actually makes sense in~$\dSch_\K$.

Define $H:A^*\longra B^*\ot_\K\K[s,t]$ to be the morphism of graded $\K$-algebras given by $H\vert_{A_+^\bu}=\al_+\ot 1$ and 
\ea
H(y^{k-i}_j)
&=(-1)^{i+1}\frac{\pd\Psi}{\pd\ti x_j^i}-(-1)^{(i+1)(k+1)}s\d a^{k-1-i}_j-(-1)^{(i+1)(k+1)}t a^{k-1-i}_j
\nonumber\\
&\qquad\qquad -\sum_{i'=0}^d\sum_{j'=1}^{m_{i'}}(-1)^{i+(i'+1)k}s\al_+\biggl(\frac{\pd\Phi_{j'}^{i'+1}}{\pd x^i_j}\biggr)a_{j'}^{k-1-i'}
\label{ln4eq45}
\ea
for $i=0,-1,\ldots,d$ and $j=1,\ldots,m_i$. This is well-defined as $A^*$ is freely generated over $A_+^*$ by the $y^{k-i}_j$. To see that $H$ is a cdga morphism, note that $H\vert_{A_+^\bu}:A_+^\bu\ra B^\bu\ot_\K\K[s,t]$ is a cdga morphism, and
\ea
&\d\ci H(y^{k-i}_j)=(-1)^{i+1}\d\biggl[\frac{\pd\Psi}{\pd\ti x_j^i}+\sum_{i'=-1}^d\sum_{j'=1}^{m_{i'}}(-1)^{(i'+1)k}s\al_+\biggl(\frac{\pd\Phi_{j'}^{i'+1}}{\pd x^i_j}\biggr)a_{j'}^{k-1-i'}\biggr]
\nonumber\\
&=(-1)^{i+1}\sum_{i'=-1}^d\sum_{j'=1}^{m_{i'}}\d\ti x^{i'}_{j'}\frac{\pd^2\Psi}{\pd\ti x^{i'}_{j'}\pd\ti x^i_j}
\nonumber\\
&+(-1)^{i+1}\sum_{i'=-1}^e\sum_{j'=1}^{n_{i'}}\biggl[\d u^{i'}_{j'}\frac{\pd^2\Psi}{\pd u^{i'}_{j'}\pd\ti x^i_j}+\d v^{k-1-i'}_{j'}\frac{\pd^2\Psi}{\pd v^{k-1-i'}_{j'}\pd\ti x^i_j}\biggr]
\nonumber\\
&-\sum_{i'=-1}^d\sum_{j'=1}^{m_{i'}}\sum_{i''=-1}^d\sum_{j''=1}^{m_{i''}}(-1)^{i+(i''+1)k}s\d\ti x^{i'}_{j'}\al_+\biggl(\frac{\pd\Phi_{j''}^{i''+1}}{\pd x^{i'}_{j'}\pd x^i_j}\biggr)a_{j''}^{k-1-i''}
\nonumber\\
&-\sum_{i'=-1}^d\sum_{j'=1}^{m_{i'}}(-1)^{(i'+1)(k+1)}s\al_+\biggl(\frac{\pd\Phi_{j'}^{i'+1}}{\pd x^i_j}\biggr)\d a_{j'}^{k-1-i'}
\nonumber\\
&-\sum_{i'=-1}^d\sum_{j'=1}^{m_{i'}}(-1)^{i+(i'+1)k}t\al_+\biggl(\frac{\pd\Phi_{j'}^{i'+1}}{\pd x^i_j}\biggr)a_{j'}^{k-1-i'}
\allowdisplaybreaks\nonumber\\
&=-\frac{\pd}{\pd\ti x^i_j}\biggl[\sum_{i'=-1}^e\sum_{j'=1}^{n_{i'}}\frac{\pd\Psi}{\pd u_{j'}^{i'}}\frac{\pd\Psi}{\pd v_{j'}^{k-1-i'}}
+\sum_{i'=-1}^d\sum_{j'=1}^{m_{i'}}(-1)^{i'+1}\al_+\bigl(\Phi_{j'}^{i'+1}\bigr)\,\frac{\pd\Psi}{\pd\ti x_{j'}^{i'}}\biggr]
\nonumber\\
&+\sum_{i'=-1}^d\sum_{j'=1}^{m_{i'}}(-1)^{i'+1}\al_+\biggl(\frac{\pd\Phi_{j'}^{i'+1}}{\pd x^i_j}\biggr)\,\frac{\pd\Psi}{\pd\ti x_{j'}^{i'}}
\nonumber\\
&+\sum_{i'=-1}^d\sum_{j'=1}^{m_{i'}}\sum_{i''=-1}^d\sum_{j''=1}^{m_{i''}}(-1)^{i+i'+(i''+1)k}s\al_+\bigl(\Phi_{j'}^{i'+1}\bigr)\al_+\biggl(\frac{\pd\Phi_{j''}^{i''+1}}{\pd x^{i'}_{j'}\pd x^i_j}\biggr)a_{j''}^{k-1-i''}
\nonumber\\
&-\sum_{i'=-1}^d\sum_{j'=1}^{m_{i'}}(-1)^{(i'+1)(k+1)}s\al_+\biggl(\frac{\pd\Phi_{j'}^{i'+1}}{\pd x^i_j}\biggr)\d a_{j'}^{k-1-i'}
\nonumber\\
&-\sum_{i'=-1}^d\sum_{j''=1}^{m_{i'}}(-1)^{i+(i'+1)k}t\al_+\biggl(\frac{\pd\Phi_{j'}^{i'+1}}{\pd x^i_j}\biggr)a_{j'}^{k-1-i'},
\label{ln4eq46}
\ea
using \eq{ln4eq45} and $\d\bigl[s\d a^{k-1-i}_j+ta^{k-1-i}_j\bigr]=0$ in the first step, and \eq{ln4eq13}, \eq{ln4eq35} and \eq{ln4eq36} in the third. Also equations \eq{ln2eq14} and \eq{ln4eq45} imply that
\ea
H\ci\d y^{k-i}_j&=\frac{\pd}{\pd\ti x^i_j}\bigl[\al_+\bigl(\Phi_+\bigr)\bigr]+\sum_{i'=i-1}^d\sum_{j'=1}^{m_{i'}}
\al_+\biggl(\frac{\pd\Phi_{j'}^{i'+1}}{\pd x^i_j}\biggr)\,
\biggl[(-1)^{i'+1}\frac{\pd\Psi}{\pd\ti x_{j'}^{i'}}
\nonumber\\
&-(-1)^{(i'+1)(k+1)}s\d a^{k-1-i'}_{j'}-(-1)^{(i'+1)(k+1)}t a^{k-1-i'}_{j'}
\nonumber\\
&-\sum_{i''=0}^d\sum_{j''=1}^{m_{i''}}(-1)^{i'+k(i''+1)}s\al_+\biggl(\frac{\pd\Phi_{j''}^{i''+1}}{\pd\ti x^{i'}_{j'}}\biggr)a_{j''}^{k-1-i''}\biggr].
\label{ln4eq47}
\ea

Combining \eq{ln4eq46}, \eq{ln4eq47} and adding $\frac{\pd}{\pd\ti x^i_j}$ applied to \eq{ln4eq39} yields
\ea
&\d\ci H(y^{k-i}_j)-H\ci\d y^{k-i}_j
\label{ln4eq48}\\
&=\sum_{i'=-1}^d\sum_{j'=1}^{m_{i'}}\sum_{i''=-1}^d\sum_{j''=1}^{m_{i''}}(-1)^{i+i'+(i''+1)k}s\al_+\bigl(\Phi_{j'}^{i'+1}\bigr)\al_+\biggl(\frac{\pd\Phi_{j''}^{i''+1}}{\pd x^{i'}_{j'}\pd x^i_j}\biggr)a_{j''}^{k-1-i''}
\nonumber\\
&+\sum_{i'=i-1}^d\sum_{j'=1}^{m_{i'}}
\sum_{i''=0}^d\sum_{j''=1}^{m_{i''}}(-1)^{i'+k(i''+1)}s\al_+\biggl(\frac{\pd\Phi_{j'}^{i'+1}}{\pd x^i_j}\biggr)\,\al_+\biggl(\frac{\pd\Phi_{j''}^{i''+1}}{\pd\ti x^{i'}_{j'}}\biggr)a_{j''}^{k-1-i''}.
\nonumber
\ea
But by applying $\frac{\pd}{\pd x^i_j}$ to \eq{ln2eq13} for $i'',j''$, we see that the r.h.s.\ of \eq{ln4eq48} is zero. Hence $\d\ci H(y^{k-i}_j)=H\ci\d y^{k-i}_j$, and $H$ is a cdga morphism. From \eq{ln4eq34}, \eq{ln4eq40} and \eq{ln4eq45} we see that restricting $H$ to $s=t=0$ gives $\hat\al$, and restricting $H$ to $s=1$, $t=0$ gives $\al$. Thus \eq{ln4eq43} commutes, and $H$ is a homotopy of cdga morphisms from $\hat\al$ to $\al$, as we have to prove.

Next we must show that the Lagrangian structures $(\hat h^0,0,\ldots)$ for $\hat\al:A^\bu\ra B^\bu,$ $\om$ and $(h^0,0,\ldots)$ for $\al:A^\bu\ra B^\bu,$ $\om$ are homotopic, over the homotopy $H$ from $\hat\al$ to $\al$. It is enough to find $\dot h^0\in(\La^2\Om^1_{B^\bu\ot_\K\K[s,t]})^{k-1}$ satisfying
\e
\begin{aligned}
\dot h^0\vert_{s=0,\; t=0}&=\hat h^0, & \qquad\dot h^0\vert_{s=1,\; t=0}&=h^0, \\
\dd\dot h^0&=0, & \d\dot h^0&=H_*(\om^0),
\end{aligned}
\label{ln4eq49}
\e
since then $(\dot h^0,0,\ldots)$ is a homotopy from $(\hat h^0,0,\ldots)$ to $(h^0,0,\ldots)$ over $H$. Set
\e
\dot h^0=\ts\sum\limits_{i=0}^d\sum\limits_{j=1}^{m_i}\dd\bigl(sa^{k-1-i}_j\bigr)\,\dd\ti x^i_j+\sum\limits_{i=0}^e\sum\limits_{j=1}^{n_i}\dd u^i_j\,\dd v^{k-1-i}_j.
\label{ln4eq50}
\e
The first two equations of \eq{ln4eq49} follow from \eq{ln4eq26} and \eq{ln4eq41}, and the third is also immediate. For the fourth equation, we have
\begin{align*}
&\d\dot h^0=\sum_{i=0}^d\sum_{j=1}^{m_i}\bigl[(\d\ci\dd(sa^{k-1-i}_j))\,\dd\ti x^i_j\!+\!(-1)^{(i+1)k}(\d\ci\dd\ti x^i_j)\,\dd(sa^{k-1-i}_j)\bigr]\\
&+\sum_{i=0}^e\sum_{j=1}^{n_i}\bigl[(\d\ci\dd u^i_j)\,\dd v^{k-1-i}_j\!+\!(-1)^{(i+1)k}(\d\ci\dd v^{k-1-i}_j)\,\dd u^i_j\bigr]
\allowdisplaybreaks\\
&=-\sum_{i=0}^d\sum_{j=1}^{m_i}\bigl[(\dd\ci\d(sa^{k-1-i}_j))\,\dd\ti x^i_j\!+\!(-1)^{(i+1)k}(\dd\ci\d\ti x^i_j)\,\dd(sa^{k-1-i}_j)\bigr]\\
&-\sum_{i=0}^e\sum_{j=1}^{n_i}\bigl[(\dd\ci\d u^i_j)\,\dd v^{k-1-i}_j\!+\!(-1)^{(i+1)k}(\dd\ci\d v^{k-1-i}_j)\,\dd u^i_j\bigr]
\allowdisplaybreaks\\
&=-\sum_{i=0}^d\sum_{j=1}^{m_i}\biggl[(-1)^{(i+1)(k+1)}\dd\ti x^i_j\,(\dd\ci\d(sa^{k-1-i}_j))\\
&+\!(-1)^{(i+1)k}\sum_{i'=0}^d\sum_{j'=1}^{m_{i'}}(-1)^{i+1}\dd\ti x^{i'}_{j'}\,H\biggl(\frac{\pd\Phi_j^{i+1}}{\pd x^{i'}_{j'}}\biggr)\,\dd(sa^{k-1-i}_j)\biggr]\\
&-\sum_{i=0}^e\sum_{j=1}^{n_i}(-1)^{(i+1)k}\biggl[
\dd\biggl(\frac{\pd\Psi}{\pd v^{k-1-i}_j}\biggr)\,\dd v^{k-1-i}_j
+\dd\biggl(\frac{\pd\Psi}{\pd u^i_j}\biggr)\,\dd u^i_j\biggr]
\allowdisplaybreaks\\
&=\dd\biggl[\sum_{i=0}^d\sum_{j=1}^{m_i}\biggl[-(-1)^{(i+1)k}\dd\ti x^i_j\,
\bigl[s\d a^{k-1-i}_j+ta^{k-1-i}_j\bigr]\\
&+\sum_{i'=0}^d\sum_{j'=1}^{m_{i'}}(-1)^{(i'+1)k}\dd\ti x^i_j\,sH\biggl(\frac{\pd\Phi_{j'}^{i'+1}}{\pd\ti x^i_j}\biggr)a_{j'}^{k-1-i'}\biggr]\\
&-\sum_{i=0}^e\sum_{j=1}^{n_i}\biggl[\dd v^{k-1-i}_j\,\frac{\pd\Psi}{\pd v^{k-1-i}_j}+\dd u^i_j\,\frac{\pd\Psi}{\pd u^i_j}\biggr]\biggr]
\allowdisplaybreaks\\
&=\dd\biggl[-\dd\Psi+\sum_{i=0}^d\sum_{j=1}^{m_i}\dd\ti x^i_j\,
\biggl[\frac{\pd\Psi}{\pd x_j^i}-(-1)^{(i+1)k}s\d a^{k-1-i}_j\\
&-(-1)^{(i+1)k}t a^{k-1-i}_j+\sum_{i'=0}^d\sum_{j'=1}^{m_{i'}}(-1)^{(i'+1)k}sH\biggl(\frac{\pd\Phi_{j'}^{i'+1}}{\pd\ti x^i_j}\biggr)a_{j'}^{k-1-i'}\biggr]
\allowdisplaybreaks\\
&=\sum_{i=0}^d\sum_{j=1}^{m_i}\dd\Bigl(\dd\bigl(H(x^i_j)\bigr)\cdot(-1)^{i+1}H(y_j^{k-i})\Bigr) \\
&=\sum_{i=0}^d\sum_{j=1}^{m_i}\dd\bigl(H(x^i_j)\bigr)\dd\bigl(H(y_j^{k-i})\bigr)=H_*\biggl[\sum_{i=0}^d\sum_{j=1}^{m_i}\dd x^i_j\,\dd y_j^{k-i}\biggr]=H_*(\om^0),
\end{align*}
using \eq{ln4eq50} in the first step, equations \eq{ln2eq14}, \eq{ln4eq35} and \eq{ln4eq36} and $H\vert_{A_+^\bu}=\al_+\ot 1$ in the third, $\d s=t$ in the fourth, $H\vert_{A_+^\bu}=\al_+\ot 1$ and \eq{ln4eq45} in the sixth, and \eq{ln2eq6} in the ninth. This completes the proof of Proposition~\ref{ln4prop2}.
\end{proof}

Proposition \ref{ln4prop2} shows that we may replace $\al,h^0$ by $\hat\al,\hat h^0$. We have now proved all the assumptions of Example \ref{ln3ex1} in the case $k<0$ with $k$ even. Equation \eq{ln3eq3} was the definition of $\ti x^i_j$. The classical master equation \eq{ln3eq4} is \eq{ln4eq39}. The definition \eq{ln3eq5} of $\al(y^{k-i}_j)$ is \eq{ln4eq40}, and the definition \eq{ln3eq6} of $\d$ in $B^\bu=(B^*,\d)$ is equations \eq{ln4eq13}, \eq{ln4eq35} and \eq{ln4eq36}. The definitions \eq{ln3eq12} and \eq{ln3eq21} of $h^0$ and $\psi$ are equations \eq{ln4eq41} and \eq{ln4eq42}. Thus, $A^\bu,\om,B^\bu,\al,h$ are in Lagrangian Darboux form. This proves Theorem \ref{ln3thm2}(i) for $k<0$ with $k$ even.

\subsection{Modifications to the proof for \texorpdfstring{$k<0$}{k<0} with \texorpdfstring{$k\equiv 1\mod 4$}{k=1 mod 4}}
\label{ln46}

Next we explain how to modify \S\ref{ln42}--\S\ref{ln45} to prove Theorem \ref{ln3thm2} when $k<0$ with $k\equiv 1\mod 4$. In the notation of Example \ref{ln3ex1} we have $d=[(k+1)/2]$, $e=[k/2]$, so that $e=d-1$ and $k=2d-1=2e+1$, with $d,k$ odd and $e$ even.

First follow \S\ref{ln42} without change. Then follow \S\ref{ln43} with the following modifications. In place of \eq{ln4eq12}, we take $B^*$ to be freely generated over $B^0$ by
\e
\begin{aligned}
&\ti x_1^i,\ldots,\ti x^i_{m_i} &&\text{in degree $i$ for
$i=-1,-2,\ldots,d$, and} \\
&u_1^i,\ldots,u^i_{n_i} &&\text{in degree $i$ for
$i=-1,-2,\ldots,d$, and} \\
&w_1^e,\ldots,w^e_{p_e} &&\text{in degree $e$, and} \\
&v_1^{k-1-i},\ldots,v^{k-1-i}_{n_i'} &&\text{in degree $k-1-i$ for $i=0,-1,\ldots,d$,}
\end{aligned}
\label{ln4eq51}
\e
where $\ti x^i_j=\al_+(x^i_j)$. In place of \eq{ln4eq14}, we write
\e
\begin{split}
&\frac{1}{k-1}\psi=\ts\sum\limits_{i=0}^d\ts\sum\limits_{j=1}^{m_i}a^{k-1-i}_j\,\dd\ti x^i_j+\sum\limits_{i=0}^d\sum\limits_{j=1}^{n_i}b^{k-1-i}_j\,\dd u^i_j\\
&+\ts\sum\limits_{j=1}^{p_e}\sum\limits_{j'=1}^{p_e}c^0_{j'j}\,w^e_{j'}\dd w^e_j+\sum\limits_{j=1}^{p_e}d^e_j\,\dd w^e_j+\sum\limits_{i=0}^d\sum\limits_{j=1}^{n_i}e^i_j\,\dd v^{k-1-i}_j,
\end{split}
\label{ln4eq52}
\e
where $a^i_j,b^i_j,c^i_{jj'},d^i_j,e^i_j\in B^i$, and $d^e_j$ includes no terms in $w^e_{j'}$, as these are written separately in the $c^0_{j'j}$ terms.

As in \eq{ln4eq15}, by leaving $h^0$ unchanged but replacing $\Xi,\psi$ by
\begin{align*}
\ti\Xi&=\Xi-(k-1)\d\Bigl[\ha\ts\sum\limits_{j=1}^{p_e}\sum\limits_{j'=1}^{p_e}c^0_{j'j}\,w^e_{j'}w^e_j+\sum\limits_{j=1}^{p_e}d^e_jw^e_j+
\sum\limits_{i=0}^e\sum\limits_{j=1}^{n_i}(-1)^ie^i_jv^{k-1-i}_j\Bigr],\\
\ti\psi&=\psi-(k-1)\dd\Bigl[\ha\ts\sum\limits_{j=1}^{p_e}\sum\limits_{j'=1}^{p_e}c^0_{j'j}\,w^e_{j'}w^e_j+\sum\limits_{j=1}^{p_e}d^e_jw^e_j+
\sum\limits_{i=0}^e\sum\limits_{j=1}^{n_i}(-1)^ie^i_jv^{k-1-i}_j\Bigr],
\end{align*}
we may assume that $c^0_{j'j}=-c^0_{jj'}$ for all $j,j'$ and $d^e_j=e^i_j=0$ for all $i,j$, as $\dd d^e_j,\dd e^i_j$ involve no terms in $\dd w^e_{j'},\dd v^{k-1-i'}_{j'}$. Here the minus sign in $c^0_{j'j}=-c^0_{jj'}$ occurs as $e$ is even. Thus the analogue of \eq{ln4eq17} is
\e
\begin{split}
h^0=\ts\sum_{i=0}^d&\ts\sum_{j=1}^{m_i}\dd a^{k-1-i}_j\,\dd\ti x^i_j+\sum_{i=0}^d\sum_{j=1}^{n_i}\dd b^{k-1-i}_j\,\dd u^i_j\\
+&\ts\sum_{j=1}^{p_e}\sum_{j'=1}^{p_e}\dd(c^0_{j'j}\,w^e_{j'})\,\dd w^e_j.
\end{split}
\label{ln4eq53}
\e

The argument of \eq{ln4eq18}--\eq{ln4eq24} now shows that $h$ nondegenerate at $q$ is equivalent to $n_i=n'_i$ for $i=0,-1,\ldots,d$, and the following being an invertible matrix over $\K$ for all~$i=0,-1,\ldots,d:$
\e
\biggl(\frac{\pd b_j^{k-1-i}}{\pd v^{k-1-i}_{j'}}\bigg\vert_q\biggr)_{j,j'=1}^{n_i},
\label{ln4eq54}
\e
and for the case $i=e$, the following being an invertible matrix over $\K:$
\e
\bigl(c^0_{j'j}\vert_q\bigr)_{j,j'=1}^{p_e}.
\label{ln4eq55}
\e
As \eq{ln4eq55} is an invertible and antisymmetric, $p_e$ is even, so we write $p_e=2n_e$.

Regard $\bigl(c^0_{j'j}\bigr)_{j,j'=1}^{p_e}$ as an antisymmetric form on the trivial vector bundle $\K^{2n_e}\t V\ra V$ over $V=\Spec B^0$, which is nondegenerate at $q\in V$, and hence near $q$. Since nondegenerate antisymmetric forms can be standardized Zariski locally by a change of basis, after localizing $B^\bu$ we can choose an invertible change of variables of the $w_1^e,\ldots,w_{2n_e}^e$, such that w.r.t.\ the new $w^e_j$ we have
\begin{equation*}
c^0_{j'j}=\begin{cases} \phantom{-}\ha, & j'=1,\ldots,n_e,\;\> j=j'+n_e, \\
-\ha, & j=1,\ldots,n_e,\;\> j'=j+n_e, \\
\phantom{-}0, & \text{otherwise.} \end{cases}
\end{equation*}

As in \S\ref{ln43}, set $\hat v_j^{k-1-i}=(-1)^{i+1}b^{k-i-i}_j$ for all $i=0,-1,\ldots,d$ and $j=1,\ldots,n_i$. Since \eq{ln4eq54} is invertible, localizing $B^\bu$ we can suppose $\ti x^i_j,u^i_j,\hat v^i_j,w^e_j$ is an alternative set of coordinates for $B^\bu$. Also define $u^e_j=w^e_j$ and $v^e_j=w^e_{j+n_e}$ for $j=1,\ldots,n_e$. Then, modifying \eq{ln4eq51}, $B^*$ is freely generated over $B^0$ by
\begin{align*}
&\ti x_1^i,\ldots,\ti x^i_{m_i} &&\text{in degree $i$ for
$i=-1,-2,\ldots,d$, and} \\
&u_1^i,\ldots,u^i_{n_i} &&\text{in degree $i$ for
$i=-1,-2,\ldots,e$, and} \\
&v_1^{k-1-i},\ldots,v^{k-1-i}_{n_i'} &&\text{in degree $k-1-i$ for $i=0,-1,\ldots,e$.}
\end{align*}
Also, from \eq{ln4eq52} we have
\e
\begin{split}
\frac{1}{k-1}\psi=\,&\ts\sum\limits_{i=0}^d\ts\sum\limits_{j=1}^{m_i}a^{k-1-i}_j\,\dd\ti x^i_j+\sum\limits_{i=0}^d\sum\limits_{j=1}^{n_i}(-1)^{i+1}v^{k-1-i}_j\,\dd u^i_j\\
&+\ts\sum\limits_{j=1}^{n_e}\bigl[\ha u^e_j\,\dd v^e_j-\ha v^e_j\,\dd u^e_j\bigr].
\end{split}
\label{ln4eq56}
\e
Leaving $h^0$ unchanged but replacing $\Xi,\psi$ by
\begin{align*}
\ti\Xi=\Xi-(k-1)\d\bigl[\ha\ts\sum_{j=1}^{n_e}u^e_jv^e_j\bigr],\qquad
\ti\psi=\psi-(k-1)\dd\bigl[\ha\ts\sum_{j=1}^{n_e}u^e_jv^e_j\bigr],
\end{align*}
equation \eq{ln4eq56} becomes 
\begin{equation*}
\frac{1}{k-1}\psi=\ts\sum\limits_{i=0}^d\ts\sum\limits_{j=1}^{m_i}a^{k-1-i}_j\,\dd\ti x^i_j+\sum\limits_{i=0}^e\sum\limits_{j=1}^{n_i}(-1)^{i+1}v^{k-1-i}_j\,\dd u^i_j,
\end{equation*}
which agrees with \eq{ln4eq25}, so $h^0=\frac{1}{k-1}\dd\psi$ yields
\begin{equation*}
h^0=\ts\sum\limits_{i=0}^d\sum\limits_{j=1}^{m_i}\dd a^{k-1-i}_j\,\dd\ti x^i_j+\sum\limits_{i=0}^e\sum\limits_{j=1}^{n_i}\dd u^i_j\,\dd v^{k-1-i}_j,
\end{equation*}
as in \eq{ln4eq26}. The rest of the proof, from \eq{ln4eq27} to the end of \S\ref{ln45}, works without further changes. This proves Theorem \ref{ln3thm2}(i) when $k<0$ with $k\equiv 1\mod 4$, and so completes the proof of Theorem~\ref{ln3thm2}(i).

\subsection{Modifications to the proof for \texorpdfstring{$k<0$}{k<0} with \texorpdfstring{$k\equiv 3\mod 4$}{k=3 mod 4}}
\label{ln47}

We now explain how to modify \S\ref{ln42}--\S\ref{ln46} to prove Theorem \ref{ln3thm2} when $k<0$ with $k\equiv 3\mod 4$, that is, to prove Theorem \ref{ln3thm2}(ii),(iii). In the notation of Example \ref{ln3ex1} we have $d=[(k+1)/2]$, $e=[k/2]$, so that $e=d-1$ and $k=2d-1=2e+1$, with $d$ even and $e,k$ odd.

The first part of the proof follows that for $k\equiv 1\mod 4$ in \S\ref{ln46}, as far as equation \eq{ln4eq55}. The only difference is that just before \eq{ln4eq53} we have $c^0_{j'j}=c^0_{jj'}$, since $e$ is odd rather than even. So $\bigl(c^0_{j'j}\bigr)_{j,j'=1}^{p_e}$ is now a symmetric matrix of functions on $V$, rather than an antisymmetric matrix. Write~$n_e=p_e$.

Now in general, nondegenerate quadratic forms cannot be trivialized Zariski locally, but they can at least be diagonalized. That is, in general we cannot find a (Zariski local) change of variables of the $w_1^e,\ldots,w_{n_e}^e$ to make $\bigl(c^0_{j'j}\bigr)_{j,j'=1}^{n_e}$ the identity matrix, but we can change variables to make $\bigl(c^0_{j'j}\bigr)_{j,j'=1}^{n_e}$ a diagonal matrix. Thus, after localizing $B^\bu$ if necessary and changing variables $w_j^e$, we can suppose there are invertible elements $q_1,\ldots,q_{n_e}$ in $B^0$ such that $c^0_{jj}=q_j$ and $c^0_{j'j}=0$ if $j'\ne j$. Also replacing $v_j^{k-1-i}$ by $\hat v_j^{k-1-i}=(-1)^{i+1}b^{k-i-i}_j$ for all $i=0,-1,\ldots,d$ and $j=1,\ldots,n_i$, we now have the analogue of \eq{ln4eq25}:
\begin{align*}
\frac{1}{k-1}\psi=\,&\ts\sum_{i=0}^d\sum_{j=1}^{m_i}a^{k-1-i}_j\,\dd\ti x^i_j+\sum_{i=0}^d\sum_{j=1}^{n_i}(-1)^{i+1}v^{k-1-i}_j\,\dd u^i_j\\
&+\ts\sum_{j=1}^{n_e}q_jw^e_j\,\dd w^e_j,
\end{align*}
so that $h^0=\frac{1}{k-1}\dd\psi$ yields the analogue of \eq{ln4eq26}:
\begin{align*}
h^0=\,&\ts\sum_{i=0}^d\sum_{j=1}^{m_i}\dd a^{k-1-i}_j\,\dd\ti x^i_j+\sum_{i=0}^d\sum_{j=1}^{n_i}\dd u^i_j\,\dd v^{k-1-i}_j\nonumber\\
&+\ts\sum_{j=1}^{n_e}\dd(q_jw^e_j)\,\dd w^e_j.
\end{align*}

Leaving $h^0$ unchanged, but by replacing $\Xi,\psi$ by
\begin{align*}
\ti\Xi&=\Xi+\d\biggl[\ts\sum\limits_{i=0}^d\sum\limits_{j=1}^{m_i}(-1)^{k-i}ia^{k-1-i}_j\ti x^i_j+\sum\limits_{i=0}^d\sum\limits_{j=1}^{n_i}(-1)^iiu^i_jv^{k-1-i}_j\biggr],\\
\ti\psi&=\psi+\dd\biggl[\ts\sum\limits_{i=0}^d\sum\limits_{j=1}^{m_i}(-1)^{k-i}ia^{k-1-i}_j\ti x^i_j+\sum\limits_{i=0}^d\sum\limits_{j=1}^{n_i}(-1)^iiu^i_jv^{k-1-i}_j\biggr],
\end{align*}
we have the analogue of \eq{ln4eq27}:
\begin{align*}
\psi&=\ts\sum_{i=0}^d\sum_{j=1}^{m_i}\bigl[(k-1-i)a^{k-1-i}_j\,\dd\ti x^i_j+
(-1)^{i+1}i\ti x^i_j\dd a_j^{k-1-i}\bigr]\\
&+\ts\sum_{i=0}^d\sum_{j=1}^{n_i}\bigl[i\,u^i_j\,\dd v^{k-1-i}_j +(-1)^{i+1}(k-1-i)v^{k-1-i}_j\,\dd u^i_j \bigr]\\
&+\ts\sum_{j=1}^{n_e}\dd(q_jw^e_j)\,\dd w^e_j.
\end{align*}

We now follow \S\ref{ln44} adding extra terms involving the $q_j,w^e_j$, which we leave as an exercise. The definition \eq{ln4eq28} of $\Psi$ is unchanged. For the final results, equation \eq{ln4eq34} is unchanged, the analogues of \eq{ln4eq35}--\eq{ln4eq36} are 
\begin{align*}
\d u^i_j&=(-1)^{i+1}\frac{\pd\Psi}{\pd v^{k-1-i}_j},&&\begin{subarray}{l}\ts i=0,-1,\ldots,d,\\[3pt] \ts j=1,\ldots,n_i,\end{subarray} \\
\d v^{k-1-i}_j&=\frac{\pd\Psi}{\pd u^i_j},&&\begin{subarray}{l}\ts i=-1,-2,\ldots,d,\\[3pt] \ts j=1,\ldots,n_i,\end{subarray}\\
\d v^{k-1}_j&=\frac{\pd\Psi}{\pd u^0_j}-\sum_{j'=1}^{n_e} \frac{w_{j'}^e}{2q_{j'}}\,\frac{\pd q_{j'}}{\pd u^0_j}\,\frac{\pd\Psi}{\pd w^e_{j'}},&& j=1,\ldots,n_0,\\
\d w^e_j&=\frac{1}{2q_j}\,\frac{\pd\Psi}{\pd
w^e_j}, && j=1,\ldots,n_e,
\end{align*}
as in \eq{ln3eq35}, and as in \eq{ln3eq33}, equation \eq{ln4eq39} is replaced by
\begin{align*}
\sum_{i=-1}^d&\sum_{j=1}^{n_i}\frac{\pd\Psi}{\pd u^i_j}\,
\frac{\pd\Psi}{\pd v^{k-1-i}_j}+\frac{1}{4}\sum_{j=1}^{n_e}
\frac{1}{q_j}\,\biggl(\frac{\pd\Psi}{\pd w^e_j}\biggr)^2\\
&+\al_+(\Phi_+)+\sum_{i=-1}^d\sum_{j=1}^{m_i}(-1)^{i+1}\al_+(\Phi_j^{i+1})\frac{\pd\Psi}{\pd\ti x^i_j}=0.
\end{align*}

Next we follow \S\ref{ln45} adding extra terms involving the $q_j,w^e_j$, to show that we can define an alternative cdga morphism $\hat\al:A^\bu\ra B^\bu$ by $\hat\al\vert_{A^\bu_+}=\al_+$ and $\hat\al(y^{k-i}_j)$ as in \eq{ln3eq34}, and a Lagrangian structure $(\smash{\hat h^0},0,\ldots)$ for $\hat\al:A^\bu\ra B^\bu,$ $\om$ with $\smash{\hat h^0}$ as in \eq{ln3eq36}, such that $\al,\hat\al:A^\bu\ra B^\bu$ are homotopic cdga morphisms, and under this homotopy, the Lagrangian structures $(h^0,0,\ldots)$ for $\al,\om$ and $(\hat h^0,0,\ldots)$ for $\hat\al,\om$ are also homotopic. Replacing $\al,h^0$ by $\hat\al,\hat h^0$, we can set $a_j^{k-1-i}=0$ for all $i,j$. Then $A^\bu,\om,B^\bu,\al,h$ are in weak Lagrangian Darboux form, in the sense of Example \ref{ln3ex2}. This completes the proof of Theorem~\ref{ln3thm2}(ii).

All the proof so far has worked with $B^\bu,\al,\bs i,\bs j$ in the homotopy commutative diagram \eq{ln4eq6} with $\bZ\simeq\bs L\t_{\bs f,\bX,\bs i}\bSpec A^\bu$ and $\bs e:\bSpec B^\bu\hookra\bZ$ a Zariski open inclusion, so that $\bs j$ is a Zariski open inclusion (or \'etale) if $\bs i$ is a Zariski open inclusion (or \'etale), as required by Theorem \ref{ln3thm2}(i),(ii). For Theorem \ref{ln3thm2}(iii) we allow $\bs j$ to be \'etale even if $\bs i$ is a Zariski open inclusion, so $\bs e:\bSpec B^\bu\hookra\bZ$ in \eq{ln4eq6} can be \'etale, and we can use \'etale local operations to construct~$B^\bu,\al,\bs j$.

In the proof above we have invertible elements $q_1,\ldots,q_{n_e}$ in $B^0$. Write $\check B^\bu=B^\bu[q_1^{1/2},\ldots,q_{n_e}^{1/2}]$ for the cdga obtained by adjoining square roots of $q_1,\ldots,q_{n_e}$ to $B^\bu$. The inclusion $\jmath:B^\bu\hookra\check B^\bu$ is an \'etale cover of degree $2^{n_e}$. Let $\check q\in\bSpec\check B^\bu$ be one of the $2^{n_e}$ preimages of $q\in\bSpec B^\bu$. Write $\check\al_+=\jmath\ci\al_+:A_+^\bu\ra\check B^\bu$, $\check\al=\jmath\ci\al:A^\bu\ra\check B^\bu$, $\bs{\check e}=\bs e\ci\bSpec\jmath:\bSpec\check B^\bu\ra\bZ$, and $\bs{\check\imath}=\bs i\ci\bSpec\jmath:\bSpec\check B^\bu\ra\bs L$. In the proof in \S\ref{ln42}, replace $B^\bu,q,\al_+,\al,\bs e,\bs i$ by $\check B^\bu,\check q,\check\al_+,\check\al,\bs{\check e},\bs{\check\imath}$, respectively. 

The new features are that $\bs e:\bSpec B^\bu\hookra\bZ$ in \eq{ln4eq6} is now \'etale rather than a Zariski open inclusion, and the invertible functions $q_1,\ldots,q_{n_e}\in B^0$ above now have square roots $q_j^{1/2}$ in $B^0$. Thus, in the first part of \S\ref{ln47}, we may change variables from $w_1^e,\ldots,w_{n_e}^e$ to $\check w_1^e,\ldots,\check w_{n_e}^e$, where $\check w_j^e=q_j^{1/2}w_j^e$, while leaving the other variables $\ti x^i_j,u^i_j,v^{k-1-i}_j$ unchanged. This has the effect of setting $q_j=1$ for all $j=1,\ldots,n_e$, so at the end of the argument above, $A^\bu,\om,B^\bu,\al,h$ are in strong Lagrangian Darboux form, in the sense of Example \ref{ln3ex2}, at the cost of working \'etale locally rather than Zariski locally. This proves Theorem \ref{ln3thm2}(iii), and finally completes the proof of Theorem~\ref{ln3thm2}.


\medskip

\noindent{\small\sc 

\noindent The Mathematical Institute, Radcliffe Observatory Quarter, Woodstock Road, Oxford, OX2 6GG, U.K.

\noindent E-mails: {\tt joyce@maths.ox.ac.uk}, {\tt safronov@maths.ox.ac.uk}.}

\end{document}